\documentclass[reqno]{amsart}

\usepackage[all]{xy}
\usepackage{graphicx}

\usepackage{amssymb}
\usepackage{amsmath}
\usepackage{mathrsfs}
\usepackage{epsfig}
\usepackage{amscd}
\usepackage{graphicx, color}

\def\E{\ifmmode{\mathbb E}\else{$\mathbb E$}\fi} 
\def\N{\ifmmode{\mathbb N}\else{$\mathbb N$}\fi} 
\def\R{\ifmmode{\mathbb R}\else{$\mathbb R$}\fi} 
\def\Q{\ifmmode{\mathbb Q}\else{$\mathbb Q$}\fi} 
\def\C{\ifmmode{\mathbb C}\else{$\mathbb C$}\fi} 
\def\H{\ifmmode{\mathbb H}\else{$\mathbb H$}\fi} 
\def\Z{\ifmmode{\mathbb Z}\else{$\mathbb Z$}\fi} 
\def\P{\ifmmode{\mathbb P}\else{$\mathbb P$}\fi} 
\def\T{\ifmmode{\mathbb T}\else{$\mathbb T$}\fi} 
\def\SS{\ifmmode{\mathbb S}\else{$\mathbb S$}\fi} 
\def\DD{\ifmmode{\mathbb D}\else{$\mathbb D$}\fi} 

\newcommand{\del}{\partial}
\newcommand{\Cont}{{\operatorname{Cont}}}

\newcommand{\ben}{\begin{enumerate}}
\newcommand{\een}{\end{enumerate}}
\newcommand{\be}{\begin{equation}}
\newcommand{\ee}{\end{equation}}
\newcommand{\bea}{\begin{eqnarray}}
\newcommand{\eea}{\end{eqnarray}}
\newcommand{\beastar}{\begin{eqnarray*}}
\newcommand{\eeastar}{\end{eqnarray*}}
\newcommand{\bc}{\begin{center}}
\newcommand{\ec}{\end{center}}

\theoremstyle{theorem}
\newtheorem{thm}{Theorem}[section]
\newtheorem{cor}[thm]{Corollary}
\newtheorem{lem}[thm]{Lemma}
\newtheorem{prop}[thm]{Proposition}

\newtheorem{sublem}[thm]{Sublemma}

\theoremstyle{definition}
\newtheorem{defn}[thm]{Definition}
\newtheorem{rem}[thm]{Remark}

\newtheorem{exm}[thm]{Example}

\newtheorem{choice}[thm]{Choice}

\newtheorem{notation}[thm]{\rm\bfseries{Notation}}

\newtheorem*{thm*}{Theorem}

\numberwithin{equation}{section}

\def\R{{\mathbb R}}

\def\Crit{{\hbox{Crit}}}

\def\E{{\mathbb E}}
\def\Z{{\mathbb Z}}
\def\C{{\mathbb C}}
\def\R{{\mathbb R}}
\def\P{{\mathbb P}}

\def\N{{\mathbb N}}

\def\11{{\mathbb I}}

\def\C{\mathbb{C}}
\def\Z{\mathbb{Z}}

\def\T{\mathbb{T}}

\def\Q{\mathbb{Q}}
\def\X{\mathbb{X}}

\def\E{\ifmmode{\mathbb E}\else{$\mathbb E$}\fi} 
\def\N{\ifmmode{\mathbb N}\else{$\mathbb N$}\fi} 
\def\R{\ifmmode{\mathbb R}\else{$\mathbb R$}\fi} 
\def\Q{\ifmmode{\mathbb Q}\else{$\mathbb Q$}\fi} 
\def\C{\ifmmode{\mathbb C}\else{$\mathbb C$}\fi} 
\def\H{\ifmmode{\mathbb H}\else{$\mathbb H$}\fi} 
\def\Z{\ifmmode{\mathbb Z}\else{$\mathbb Z$}\fi} 
\def\P{\ifmmode{\mathbb P}\else{$\mathbb P$}\fi} 
\def\SS{\ifmmode{\mathbb S}\else{$\mathbb S$}\fi} 
\def\DD{\ifmmode{\mathbb D}\else{$\mathbb D$}\fi} 

\def\R{{\mathbb R}}

\def\Crit{{\hbox{Crit}}}
\def\E{{\mathbb E}}
\def\Z{{\mathbb Z}}
\def\C{{\mathbb C}}
\def\R{{\mathbb R}}

\def\N{{\mathbb N}}

\def\CA{{\mathcal A}}

\def\CE{{\mathcal E}}

\def\CL{{\mathcal L}}

\def\CP{{\mathcal P}}

\def\CP{{\mathcal P}}

\def\CV{{\mathcal V}}

\def\tildeX{{\widetilde X}}
\def\tildeP{{\widetilde P}}
\def\tildegamma{{\widetilde \gamma}}

\def\darr#1{\raise1.5ex\hbox{$\leftrightarrow$}
\mkern-16.5mu #1}

\def\roughly#1{\raise.3ex\hbox{$#1$\kern-.75em
\lower1ex\hbox{$\sim$}}}

\def\opname#1{\mathop{\kern0pt{\rm #1}}\nolimits}

\def\dim{\opname{dim}}

\def\supp{\operatorname{supp}}

\def\Dev{\operatorname{Dev}}

\def\leng{\operatorname{leng}}

\def\span{\operatorname{span}}

\def\Cont{\operatorname{Cont}}
\def\Crit{\operatorname{Crit}}
\def\Spec{\operatorname{Spec}}
\def\Sing{\operatorname{Sing}}
\def\GFQI{\frak{G}}
\def\Index{\operatorname{Index}}

\def\Image{\operatorname{Image}}
\def\ev{\operatorname{ev}}

\DeclareFontFamily{U}{MnSymbolC}{}
\DeclareSymbolFont{MnSyC}{U}{MnSymbolC}{m}{n}
\DeclareFontShape{U}{MnSymbolC}{m}{n}{
    <-6>  MnSymbolC5
   <6-7>  MnSymbolC6
   <7-8>  MnSymbolC7
   <8-9>  MnSymbolC8
   <9-10> MnSymbolC9
  <10-12> MnSymbolC10
  <12->   MnSymbolC12}{}
\DeclareMathSymbol{\intprod}{\mathbin}{MnSyC}{'270}

\begin{document}

\quad \vskip1.375truein

\def\mq{\mathfrak{q}}
\def\mp{\mathfrak{p}}
\def\mH{\mathfrak{H}}
\def\mh{\mathfrak{h}}
\def\ma{\mathfrak{a}}
\def\ms{\mathfrak{s}}
\def\mm{\mathfrak{m}}
\def\mn{\mathfrak{n}}
\def\mz{\mathfrak{z}}
\def\mw{\mathfrak{w}}
\def\Hoch{{\tt Hoch}}
\def\mt{\mathfrak{t}}
\def\ml{\mathfrak{l}}
\def\mT{\mathfrak{T}}
\def\mL{\mathfrak{L}}
\def\mg{\mathfrak{g}}
\def\md{\mathfrak{d}}
\def\mr{\mathfrak{r}}
\def\Cont{\operatorname{Cont}}
\def\Crit{\operatorname{Crit}}
\def\Spec{\operatorname{Spec}}
\def\Sing{\operatorname{Sing}}
\def\GFQI{\text{\rm GFQI}}
\def\Index{\operatorname{Index}}
\def\Cross{\operatorname{Cross}}
\def\Ham{\operatorname{Ham}}
\def\Fix{\operatorname{Fix}}
\def\Graph{\operatorname{Graph}}
\def\id{\text{\rm Id}}
\def\Exp{\operatorname{Exp}}

\title[Canonical generating function]
{Contact action functional, calculus of variations and 
canonical generating functions of Legendrian submanifolds}

\author{Yong-Geun Oh, Seungook Yu}
\address{Center for Geometry and Physics, Institute for Basic Science (IBS),
77 Cheongam-ro, Nam-gu, Pohang-si, Gyeongsangbuk-do, Korea 37673
\& POSTECH, Gyeongsangbuk-do, Korea}
\email{yongoh1@postech.ac.kr}
\address{CNRS, Laboratoire de Mathématiques Jean Leray, Nantes Université, 
2 Chemin de la Houssinière, 44322 Nantes, France}
\curraddr{POSTECH,
77 Cheongam-ro, Nam-gu, Pohang-si, Gyeongsangbuk-do, Korea 37673}
\email{seungook.yu@univ-nantes.fr,yso1460@postech.ac.kr}
\thanks{This work is supported by the IBS project \# IBS-R003-D1}

\date{November 26, 2023; revised on June 19, 2025}

\begin{abstract}
In the present paper, we formulate a contact analogue on the one-jet bundle $J^1B$
of Weinstein's observation which says the classical action functional on the cotangent bundle 
is a generating function of any
Hamiltonian isotope of the zero section. We do this by identifying the correct action functional
which is defined on the space of Hamiltonian-translated (piecewise smooth) \emph{horizontal}
curves of the contact distribution, which we call the \emph{Carnot path space}. 
Then we give a \emph{canonical construction} of
generating functions of Legendrians in the 1-jet bundle (contact isotopic to the zero section),
 which is the Legendrian counterpart of
Laudenbach-Sikorav's canonical construction of the
generating function of Hamiltonian isotope
of the zero section on the cotangent bundle which utilizes
a finite dimensional approximation of the action functional.

Motivated by this construction, we develop a Floer theoretic construction of spectral invariants
for the Legendrian submanifolds in the sequel \cite{oh-yso:spectral} which is the contact analogue
to the construction given in \cite{oh:jdg,oh:cag} for the Lagrangian submanifolds in the cotangent bundle.
\end{abstract}



\maketitle

\tableofcontents

\section{Introduction}

\emph{Generating functions} play an important role in
Hamiltonian mechanics and in micro-local analysis, especially in
H\"ormander's calculus of Fourier integral operators \cite{hormander:fourier}.
This construction can be seen as  a particular case of more general push-forward operations
in the calculus of Lagrangian submanifolds. More recently this calculus
of Lagrangian submanifolds has received much attention in relation to
functorial constructions of morphisms and operations in
the Fukaya category \cite{ww-3}. 

Chaperon \cite{chaperon} and Laudenbach-Sikorav \cite{laud-sikorav} used
the \emph{broken geodesic} approximation of
the action functional in the proof of Arnold's conjecture in the
cotangent bundle. Sikorav \cite{sikorav:cmh} identified this approximation 
as a canonical construction of a generating function (quadratic at infinity), and 
also used the calculus of Lagrangian submanifolds (and the
corresponding generating functions) to put the construction of \cite{laud-sikorav}
in a natural setting. (See also  \cite{alan:lecture} and \cite{chekanov:generating}.) 
This technique of generating functions was then
culminated by Viterbo \cite{viterbo:generating} into a construction of some symplectic invariants. 
(See also \cite{traynor}.)
In this construction, the notions of \emph{generating functions quadratic
at infinity} (abbreviated as $\GFQI$) and stable Morse theory play an important role.
(See \cite{eliash-gromov:stable} for a detailed exposition on this aspect.) 
Kragh has been making a series of homotopical constructions to the level of a construction of Thom spectra
by the technique of generating functions and applied them to the nearby Lagrangian conjecture 
\cite{kragh:nearby1,kragh:nearby2}. (See also the joint works with others
\cite{abouzaid-kragh:jtop,abouzaid-kragh:acta,ACGK,AA-GCK}.)

In \cite{oh:jdg,oh:cag}, the first named author utilizes Lagrangian intersection Floer theory to 
define Floer theoretic Lagrangian spectral invariants on the cotangent bundle whose coincidence 
with Viterbo's $\GFQI$ spectral invariants \cite{viterbo:generating} (up to a uniform addition by constant) 
is proven by Milinkovi\'c \cite{milinkovic1,milinkovic2}. The starting point of the construction in \cite{oh:jdg}
is Weinstein's ingenious observation \cite{alan:observation}, which we call \emph{Weinstein's de-approximation},
that the classical Hamilton's action functional, denoted by $\CA_H^{\text{\rm cl}}$,
 is a generating function of
the time-one image $\phi_H^1(o_{T^*B})$ of the zero section of $T^*B$ under the Hamiltonian flow of $H = H(t,x)$.
(See Section \ref{sec:effective}  for the details.)
\begin{rem}[Concerning Weinstein's de-approximation]  At this point it might be useful to set straight the history around
Laudenbach-Sikorav's finite dimensional construction and Weinstein's infinite dimensional construction
which we call `Weinstein's de-approximation'. (See Section \ref{sec:alan-observation} for 
details.) 
When the preprint of Laudenbach-Sikorav's 
paper \cite{laud-sikorav} first came out at the end of 1984, Oh was Weinstein's graduate student. 
(Both Weinstein and Oh studied the preprint before its publication.) In the fall of 1987,  
while Oh was still a graduate student, Weinstein discovered a direct variational explanation 
of the essence of what Laudenbach-Sikorav did in \cite{laud-sikorav}. In one of 
the lectures given in his graduate course \cite{alan:observation}, he
explained his precise formulation of \emph{the Morse family framework} on
the space of paths \emph{as described in the present paper}.
He called his variational construction the \emph{de-approximation of Laudenbach-Sikorav's broken-trajectory approximation} at that time. 
Then in \cite{oh:jdg}, Oh utilized 
this Weinstein's observation as the starting point of his Floer theoretic construction of
Viterbo's generating function invariants \cite{viterbo:generating}, and 
called it `Weinstein's de-approximation'.
After then, Miliknovi\'c \cite{milinkovic1,milinkovic2}
employed this connection between the generating function and the action functional 
in his study of the equivalence between the stable Morse homology of GFQI
 \cite{viterbo:generating,traynor} and the Lagrangian intersection Floer homology \cite{oh:jdg} 
 of the Hamiltonian isotope of the zero section of $T^*B$.
 \end{rem}
The main purpose of the present paper is to formulate the Legendrian counterpart thereof
and provide a canonical construction of $\GFQI$ for any contact Hamiltonian isotope of the
zero section of a one-jet bundle $J^1B$.
This is also the first step of the sequel \cite{oh-yso:spectral} where the present authors construct
the Legendrian counterpart on the one-jet bundle
of the first named author's Floer theoretic construction of Lagrangian spectral invariants on the cotangent bundle.
We believe that the invariants coincide with the $\GFQI$ spectral invariants constructed by
Th\'eret in \cite{theret}, Bhupal \cite{bhupal} and Sandon \cite{sandon:homology} following Viterbo's scheme of stable Morse theory.
The main step of the $\GFQI$ invariants is the existence and certain uniqueness results of the generating function
quadratic at infinity which are established in \cite{laud-sikorav}, \cite{sikorav:cmh} for the existence
and in \cite{viterbo:generating}, \cite{theret} for the uniqueness.
(See Section \ref{sec:gfqi-generation}  for some quick review on the $\GFQI$ construction.)

The known construction of generating functions quadratic at infinity for Legendrian submanifolds in
the one-jet bundle \cite{chaperon}, \cite{chekanov:generating}, \cite{ferrand} is rather different from the natural construction 
of Laudenbach-Sikorav which uses a Hamiltonian variation of Bott's finite dimensional geodesic approximation 
in his Morse theoretic study of the topology of loop spaces. 
 This construction \emph{{in the chain level}}
 depends on the rank $N$ of the associated vector bundle
such that the associated stable Morse homology has the limit as $N \to \infty$ 
relating to Floer homology for the action functional. (See \cite{milinkovic1,milinkovic2}.)
This chain level study is fundamental for 
the aforementioned Kragh's homotopical construction above.

It turns out that finding the contact counterpart of the aforementioned Weinstein's de-approximation of 
Laudenbach-Sikorav's construction of a `canonical generating function'
described in \cite{alan:observation} is rather nontrivial largely because
the commonly used action functional carries large degeneracy in its variational study arising from
the whole set of reparametrizations of the curves. To remove this degeneracy and rigidify the relevant
calculus of variations, we need to put some constraints. For this purpose,
we need to do some general study of the space of \emph{(translated) horizontal paths} of
the contact distribution of a contact manifold $(M,\xi)$ and get into a little bit
of Carnot-Carath\'eodory geometry. (See \cite{mitchell}
for the relevant results and \cite{montgomery} for a nice panoramic exposition on Carnot-Carath\'eodory
geometry.)

In the sequel \cite{oh-yso:spectral}, utilizing these background geometric preparations 
and analysis of (perturbed) contact instantons established in 
\cite{oh:contacton-Legendrian-bdy, oh:entanglement1, oh:contacton-transversality},
we carry out the Floer theoretic construction of Legendrian spectral invariants as
mentioned above, which is the contact counterpart of the construction given in \cite{oh:jdg,oh:cag}.

\subsection{The space of Hamiltonian-translated horizontal paths}

Let $(M,\xi)$ be a contact manifold. 
A contact diffeomorphism, abbreviated as a contactomorphism,
$\psi: M \to M$ is a diffeomorphism satisfying
$d\psi(\xi) \subset \xi$, and a contact vector field $X$ is one whose flow 
generates an isotopy of contactomorphisms. Equivalently, $X$ is contact if and only if 
the Lie derivative
$\CL_XY = [X,Y]$ is tangent to $\xi$ whenever the vector field $Y$ is tangent to $\xi$.

Now let $\xi$ be coorientable which we will 
assume from now on. Then we choose a contact form $\lambda$, i.e., one form with $\xi = \ker \lambda$.
A choice of contact form induces a decomposition
\be\label{eq:lambda-decompos}
TM = \xi \oplus \R \langle R_\lambda \rangle
\ee
where $R_\lambda$ is the $\lambda$-Reeb vector field uniquely determined by the equations
$$
R_\lambda \intprod \lambda = 1, \quad R_\lambda \intprod  d\lambda = 0.
$$
The choice of contact form also enables us to associate a natural $\R$-valued function
given by $H = -\lambda(X)$ to \emph{any} vector field.

Conversely, for any given $\R$-valued 
 function $H$,
we can associate a \emph{contact vector field} by the defining equation
\be\label{eq:XH}
\begin{cases}
X \intprod \lambda = -H \\
X  \intprod d\lambda = dH - R_\lambda[H]\lambda.
\end{cases}
\ee
(See \cite[Theorem 2.3]{geiges} for the proof, for example, but with different sign convention
of $H: = \lambda(X)$. In the present paper, we systematically follow the sign convention 
of \cite{BCT,dMV,oh:contacton-Legendrian-bdy} in our sometimes rather elaborate calculations. 
We refer the readers to Section \ref{sec:preliminaries} for a summary of contact Hamiltonian calculus 
in our convention.  To avoid readers' confusion arising
from this difference of conventions from other literature, 
we make our contact Hamiltonian calculus self-contained.)

The contact Hamilton's equation $\dot y = X_H(t,y)$ can be rewritten as
\be\label{eq:contact-Hamilton's-eq}
\begin{cases}
(\dot \gamma - X_H(t,\gamma(t)))^\pi = 0 \\
\gamma^*\lambda + H(t,\gamma(t))\, dt = 0
\end{cases}
\ee
 by splitting the equation into the contact distribution direction and the Reeb direction. This decomposition
 suits well with the analytic study of perturbed contact instanton equation
as presented in \cite{oh:contacton-Legendrian-bdy,oh:entanglement1}.
A particular type of perturbed action functional is introduced by the first-named author
in \cite{oh:contacton-Legendrian-bdy}, \cite{oh:entanglement1}  for the variational study of
this equation under the general Legendrian boundary condition in any contact manifold.
\begin{rem}
We would like to emphasize that readers should take our sign convention 
\cite{BCT,dMV,oh:contacton-Legendrian-bdy}  (or others, if any) \emph{as a whole
package} which should be consistently used in the various calculations we involve for the purpose of
 the present paper. 
\end{rem}

By definition of (coorientation-preserving) contactomorphism $\psi$ of a coorientable contact structure
$\xi$ equipped with a contact form $\lambda$, we can write $\psi^*\lambda = f \lambda$ for
some positive function $f = f_\psi^\lambda$. Writing $g = g_\psi^\lambda := \log f$, which we call
the \emph{conformal exponent} of contactomorphism $\psi$, we have
$$
\psi^*\lambda = e^{g_\psi} \lambda.
$$

\begin{notation} Let $H = H(t,y)$ be a contact Hamiltonian. We denote by $\psi_H^t$ the contact Hamiltonian
flow of $H$. We denote by $\phi_H^t$ the isotopy defined by
\be\label{eq:phiHt}
\phi_H^t: = \psi_H^t (\psi_H^1)^{-1}.
\ee
\end{notation}

We remark that $\phi_H^1 = \id$ and $\phi_H^0 = (\psi_H^1)^{-1}$. 
The curve  $\gamma$ defined by $\gamma(t) := \phi_H^t(y)$ solves the final-value problem
$$
\dot \gamma = X_H(\gamma), \quad  \gamma(1) = y
$$
for the given point $y$.
(See \cite{oh:cag} for a systematic discussion 
on this point of view.)

\begin{defn}[Perturbed action functional]
Let $H = H(t,y)$ be a contact Hamiltonian and denote by $\CL(M)$ the free path space, i.e., the set of 
smooth maps $\gamma: [0,1] \to M$. We
 define the functional $\CA_{H;\lambda}: \CL(M) \rightarrow \R$  by 
\bea\label{eq:perturbed-action}
\CA_{H;\lambda}(\gamma) =  \int_0^1 e^{g_{(\phi_H^t)^{-1}} (\gamma(t))} 
\left(\lambda(\dot \gamma (t)) + H_t (\gamma (t))\right) \, dt.
\eea
\end{defn}
Note that, when $H = 0$, we have
$$
\CA_{0;\lambda} (\gamma): = \int_0^1 \gamma^* \lambda
$$
which is the standard contact action functional for the Reeb flow. Since we will
not vary $\lambda$, we will simply omit the dependence of $\lambda$ and write $\CA_H$ and $\CA_0$
respectively.

\begin{rem} \begin{enumerate}
\item 
Under the assignment
$$
\Phi_H(\gamma)(t) = (\phi_H^t)^{-1}(\gamma(t)) = \overline\gamma(t),
$$
the map $\Phi_H$ transforms paths with Legendrian boundary condition 
 from $(R_0,R_1)$ to $(\psi_H(R_0),R_1)$ of $\gamma$ at $t = 0, \, 1$ for a given pair of
 Legendrian submanifolds $(R_0,R_1)$. (See \eqref{eq:PhiH} for the precise definition.)
\item In fact, one can associate an action functional to each time-dependent family of contact forms
$\{\lambda^s\}_{s \in [0,1]}$ by
$$
\CA_{\{\lambda^{(\cdot)}\}}(\gamma): = \int_0^1 \lambda^t(\dot \gamma(t))\, dt
$$
which we will again simply write as $\CA_0$ for the simplicity of notation.
The action functional $\CA_H$ is obtained by considering the family given by
$$
\lambda^s = ((\phi_H^s)^{-1})^*\lambda = (\phi_H^s)_*\lambda.
$$
See Lemma \ref{lem:CAHu=CAw}.
\end{enumerate}
\end{rem}

We attract readers' attention to the fact that
the value of $\CA_H$ at any Hamiltonian trajectory is \emph{automatically} zero which
exhibits the nature of \emph{constrained variation} 
with the contact Hamiltonian trajectories in the point of view of the calculus of variation.  
(See the discussion of Section \ref{sec:Carnot-pathspace}, especially Theorem \ref{thm:Carnot-Frechet},
 for  further elaboration.)
This has led us to consider the  restriction
of $\CA_H$ to the following Carnot-type path space.

\begin{defn}[Translated horizontal paths]
For each given Hamiltonian $H = H(t,y)$, we define the
subset
$$
\CL^{\text{\rm Carnot}}(M,H) \subset \CL(M)
$$
to be the set of piecewise smooth continuous curves $\gamma:[0,1] \to M$ that satisfy 
\be\label{eq:Carnot-path}
\gamma^*\lambda + H(t,\gamma(t)) \, dt = 0.
\ee
We call an element $\gamma$ therein an \emph{$H$-translated horizontal path} of the contact
distribution $\xi$, and the set of such paths the \emph{$H$-Carnot path space}.
\end{defn}

\begin{rem} Recall that in thermodynamics the notion of a \emph{Carnot path} is a finite succession of
paths which are alternately adiabatic and isothermal. Such a path belongs to the realm of piecewise
smooth  horizontal paths for the thermodynamic contact form, e.g.,
$\lambda = dU - TdS + PdV$ in $J^1\R^2 \cong \R^5$. This is the reason
for our naming of (Hamiltonian-translated) Carnot path space here.
\end{rem}

In fact, if we define $\overline \gamma(t): = (\phi_H^t)^{-1}(\gamma(t))$, \eqref{eq:Carnot-path} is
equivalent to saying that $\overline \gamma$ is tangent to the contact distribution, i.e., that it satisfies
\be\label{eq:horizontal}
\frac{d \overline \gamma}{dt}\in \xi.
\ee
In Carnot-Carath\'eodory geometry, such a path $\overline \gamma$ is called a 
\emph{horizontal curve} of the distribution $\xi$. 
Since the contact distribution satisfies the 
H\"ormander's condition \cite{hormander:hypoelliptic}, it is an immediate consequence of the general result
by Chow \cite{chow} that $\CL^{\text{\rm Carnot}}(M,H)$ is \emph{locally transitive} in the sense that any point
$p\in M$ can be joined to a given point $m \in M$, provided $p$ is sufficiently close to $m$.
(See \cite{mitchell} for a quick summary of the Carnot-Carath\'eodory geometry needed for our discussion.
For readers' convenience, we provide some quick discussion on the Carnot-Carath\'eodory metrics
on $M$ equipped with a distribution that satisfies H\"ormander's condition 
\cite{hormander:hypoelliptic} and its relationship
with the above translated Carnot paths in Appendix \ref{sec:Carnot-horizontal}.)

However, a priori it is not clear whether this subset carries a smooth structure in the sense of
an infinite dimensional Fr\'echet manifold on which we can do calculus of variation.
In this regard, we prove the following.

\begin{thm}\label{thm:Carnot} For any smooth Hamiltonian $H =H(t,y)$, the subset
$$
\CL^{\text{\rm Carnot}}(M,H) \subset \CL(M)
$$
is a smooth submanifold of the free path space $\CL(M)$ and so carries the induced Fr\'echet manifold structure.
\end{thm}
We would like to highlight that \emph{this smooth Fr\'echet manifold structure does not involve
any type of transversality or nondegeneracy hypothesis on the Hamiltonian $H$ 
but holds for any smooth Hamiltonian.}
We alert readers that when $H = 0$ this is nothing but
the set of horizontal curves of the contact distribution $\xi$.

\begin{cor} The set of piecewise smooth horizontal paths of the contact distribution of
$(M,\xi)$ is a Fr\'echet manifold.
\end{cor}

\subsection{Morse family and generation of Legendrian submanifolds}
\label{subsec:Morsefamily}

To facilitate our discussion, especially when we deal with the infinite dimensional case 
of the action functional,
 we start by considering the following general notion of \emph{Morse family}
in the theory of \emph{generating functions}. (See \cite{alan:family,lees} for the earlier usage 
of generating functions defined on the general fiber bundles
and  \cite[Section 3, Definition 1]{castrillon-ratiu}  for the explicitly written
definition of the term \emph{Morse family}  in the sense of Definition \ref{defn:Morsefamily}
below for the finite dimensional case, and see 
\cite{do-oh:reduction} for the treatment of the infinite dimensional case.)

\begin{rem} See Section \ref{sec:gfqi-generation} for the definition and properties of
generating functions.  Mostly in the literature, especially in relation to the study of
symplectic topology of cotangent bundles (resp. to that of contact topology of one-jet bundles),
generating functions defined on a \emph{vector bundle} are considered 
\cite{laud-sikorav}, \cite{sikorav:cmh}, \cite{viterbo:generating}, \cite{traynor}, \cite{theret}, \cite{bhupal}, 
\cite{sandon:homology}. However one can consider a general Morse family 
to construct a priori a more general class of exact Lagrangian (resp. Legendrian) submanifolds than
the Hamiltonian isotopes as shown in \cite{alan:family,lees,castrillon-ratiu}, \cite{do-oh:reduction}.
It is an interesting open problem to see whether this family of exact Lagrangian submanifolds
would contain a counter-example of the nearby Lagrangian conjecture by developing a more
functorial study of the construction utilizing some `additive structures' of the Morse family and 
of the family of associated Lagrangian submanifolds \cite{hormander:fourier,viterbo:generating}.
\end{rem}

Let $\pi:\CE \to B$ be a fiber bundle, where $B$ is finite dimensional but $\CE$ is not
necessarily so. We choose a suitable Ehresmann connection $\Gamma$ for $\pi$, i.e., a splitting
\be\label{eq:Gamma-splitting}
\Gamma: \quad T\CE = VT\CE \oplus HT\CE
\ee
where $VT_e\CE: = (d_e\pi)^{-1}(0)$ is a vertical tangent space at $e$, and $HT_e\CE$ is a horizontal one. 
The derivative $d\pi$ restricts to an isomorphism
\be\label{eq:dpie}
d\pi_e: HT_e \CE \to T_{\pi(e)} B
\ee
at all $e \in E$. 
When we are given a smooth function $S:\CE \to \R$, we consider the diagram
\be\label{eq:relative-Morse}
\xymatrix{\CE  \ar[d]^{\pi} \ar[r]^S & \R \\
B &
}
\ee
and decompose its differential
$dS = d^vS + D^hS$ where we define
$$
d^v S(e) = dS(e)|_{VT_e\CE}, \quad D^hS(e) = dS(e)|_{HT_e \CE}
$$
where the notation indicates the vertical derivative $d^vS$ does not depend on the
connection $\Gamma$ while $D^hS$ does. (See \eqref{eq:fiber-derivative} for precise
description of $d^vS$ as the \emph{fiber derivative} which is defined canonically.)

\begin{rem}
At the moment, we avoid attempting to provide precise meaning of the dual space of an
infinite dimensional space (e.g., in the definition of `cotangent bundle')
which will inevitably involve nontrivial functional analytical issues.
(See \cite{kriegl-michor} for an excellent exposition on such matters.)
Therefore readers had better take the current discussion in the formal level (as in the 
variational calculus) to motivate further discussion below. 
For the readers who feel uneasy of the formal discussion, we refer them 
to  \cite[Appendix B]{do-oh:reduction} for some detailed rigorous discussion 
on this dimension counting in the infinite dimensional fibration $\CE \to B$.
\end{rem}

We define
$$
\Sigma_{S}: = (d^vS)^{-1}(o_{(VT)^*\CE}) = \{ e \in \CE \mid d^vS(e) = 0\}
$$
where $o_{ (VT)^*\CE}$ is the zero section 
of the bundle $(VT)^*\CE \to \CE$ dual to $VT\CE \to \CE$.
We  then define a map $\iota_S: \Sigma_{S} \to J^1B$ by
\bea
\iota_{S}(e) &:= &  \left(\pi(e), D^hS(e) \circ (d\pi_e|_{HT_e\CE })^{-1}, S(e)\right) \label{eq:iotaCS-intro}\\
& = & \left(\pi(e), dS(e) \circ (d\pi_e|_{T_e \CE})^{-1}, S(e)\right)
\eea
where the latter identity follows since $e \in \Sigma_S$.  It shows that the definition of
the map $\iota_S$ involves an Ehresmann connection but its image
does not depend on the choice of an Ehresmann connection.

\begin{rem}\label{rem:nonsmooth}
\begin{enumerate}
\item 
We remark that the image of $\iota_S$  could a priori be a pathological space
by two reasons: Firstly $\Sigma_{S}$ can be a pathological space
unless $S$ satisfies suitable transversality hypothesis. Secondly even if
$\Sigma_S$ carries a smooth structure, the map $\iota_S$ can be singular.
However both maps, $d^vS: E \to E^*$ and 
$\iota_S:  \Sigma_S = (d^vS)^{-1}(0)  \subset E \to J^1B$, 
are defined one after the other for any differentiable function $S$ 
\emph{without any kind of transversality condition}. 
We denote by $R_S$ the image,  $\iota_S(\Sigma_S) \subset J^1B$.
 \item Under the  transversality hypothesis \eqref{eq:transversality-dvCSCF} below, 
$R_S$ becomes the (smooth) \emph{push-forward} of the Legendrian graph
$$
\text{\rm Image} \, j^1S \subset  J^1\CE
$$
under the projection $\pi_\CE: \CE \to B$ in the calculus of 
Legendrian submanifolds. 
\item We recall that under the same transversality hypothesis on $S$, the map 
$i_S: \Sigma_S \to T^*B$ given by
$$
i_S(e): = (D^hS(e) \circ (d\pi_e|_{HT_e\CE})^{-1})
$$
defines a Lagrangian immersion which is the push-forward of $\Graph dS \subset T^*\CE$
under the same projection $\pi_\CE: \CE \to B$ in \emph{calculus of  Lagrangian submanifolds} \cite{alan:lecture}, \cite{ww-3}. The above construction is nothing but the contact analogue of this operation.
(We refer to \cite{oh-park} for the details on the Legendrian correspondence.)
\end{enumerate}
\end{rem}
We now consider the following standard transversality hypothesis
\be\label{eq:transversality-dvCSCF}
d^v S  \pitchfork o_{ (VT)^*\CE}
\ee
This ensures smoothness of $\Sigma_S$.

\begin{defn}[Morse family]\label{defn:Morsefamily}
Consider a smooth fibration $\pi: \CE \to B$ and a smooth function $S : \CE \to \R$
visualized in the diagram \eqref{eq:relative-Morse}. 
We say $S$ is \emph{$\pi$-relative Morse}  with respect to the fibration $\pi: \CE \to B$,
if it satisfies \eqref{eq:transversality-dvCSCF}, or simply a \emph{Morse family} when no 
explicitly mentioning with respect to the fibration is needed.
\end{defn}
A basic result is
that the map $\iota_S : \Sigma_S \to J^1B$ given in \eqref{eq:iotaCS-intro}
defines a Legendrian immersion for any Morse family $S: \CE \to \R$.
(See \cite{hormander:fourier,alan:family,lees} in the symplectic/Lagrangian
 context which automatically imply the contact/Legendrian context,
and \cite[Appendix B]{do-oh:reduction}
for detailed explanation in the infinite dimensional setting.)

\begin{defn} \label{defn:S-generates-R} Let $\CE \to B$ be a fiber bundle and
 $S: \CE \to \R$ be a Morse family. 
We say $S$ \emph{generates} a Lagrangian submanifold $L \subset T^*B$ (resp. a Legendrian submanifold 
$R \subset J^1B$), if $L$ (resp. $R$)  is the image of a Lagrangian immersion $i_S: \Sigma_S \to T^*B$
(resp. a Legendrian immersion $\iota_S: \Sigma_S \to J^1B)$ defined above, 
and call $S$ a \emph{generating function} of  $L$ (resp. of $R$).
\end{defn}

We denote by 
$$
\mathfrak{Leg}(J^1B) \, \text{(resp. $\mathfrak{Leg}_0(J^1B)$)}
$$
the set of  Legendrian submanifolds of $J^1B$ (resp. those Legendrian 
isotopic to the zero section $R_0 = o_{J^1B}$).

\subsection{Action functional as a canonical generating function on $J^1B$}
\label{subsec:action-functional}
Now we restrict contact manifolds to the case of one jet bundles
$J^1B$ of a smooth manifold $B$ equipped with the
standard contact structure $\xi =\ker \lambda$ given by the standard contact form
$\lambda = dz - pdq$.
We then consider the class of Legendrian submanifolds given by the Hamiltonian isotope
 $R = \psi_H^1(o_{J^1B})$ of the contact Hamiltonian $H = H(t,y)$ with $y = (x,z) \in J^1B$. 
We denote by
\be\label{eq:piBpiT*B}
\pi_B: J^1B \to B,\quad \pi_{T^*B}: J^1B \to T^*B, \quad z: J^1B \to \R
\ee
the obvious projections. Then we can write
$$
\lambda = dz - \pi^*\theta, \quad \pi = \pi_{T^*B}
$$
where $\theta = pdq$ is the Liouville one-form on the cotangent bundle $T^*B$.

We also specialize  the action functional \eqref{eq:perturbed-action}
to the case of $J^1B$. As mentioned above
the value of $\CA_H$ at any Hamiltonian trajectory vanishes. Largely because of this,
finding a right formulation of the contact counterpart of Weinstein's de-approximation \cite{alan:observation},
more precisely a \emph{canonical variational construction of a generating function from Hamilton's action
 functional},
is much more nontrivial than expected, whose explanation is now in order.
First of all, we need to consider the following variation of $\CA_H$, which we call
the \emph{effective action functional}, on the aforementioned Carnot-type path space.

\begin{defn}[Effective action functional] We define $\widetilde \CA_H: \CL(J^1B)\to \R$ to be
\bea\label{eq:tildeCAH}
\widetilde{\mathcal{A}}_H(\gamma) & = & - \CA_H (\gamma) + z(\gamma(1)) \nonumber \\
& = & - \int_0^1 e^{g_{(\phi_H^t)^{-1}} (\gamma (t))} (\lambda_{\gamma(t)} (\dot \gamma (t)) + H(\gamma(t))) dt + z(\gamma(1)).
\eea
\end{defn}
We mention that  if $\gamma$ is a Hamiltonian trajectory, the action is reduced to
$$
\widetilde{\mathcal{A}}_H(\gamma) = z(\gamma(1)).
$$
(We refer to Section \ref{sec:effective}, especially Remark \ref{rem:adoption}, 
for the explanation on our adoption of the sign used here and for further expounding 
of the definition of $\widetilde \CA_H$.)

Then we denote by 
$$
\CA _H^{\rm CC} : \CL ^{\rm Carnot} (J^1B,H) \to \R
$$
the restriction of $\widetilde \CA _H$ to $\CL ^{\rm Carnot} (J^1B,H)$ and also call 
it (restricted) effective action functional. (Here `CC' stands for Carnot and Carath\'eodory.)

Then we consider the subset
$$
\CL^{\text{\rm Carnot}}_{0}(J^1B,H;o_{J^1B}) \subset \CL^{\text{\rm Carnot}}(J^1B,H)
$$
consisting of the paths $\gamma \in \CL^{\text{\rm Carnot}}(J^1B,H)$
 satisfying $\gamma(0) \in o_{J^1B}$ in addition. 
  Finally we  consider the fibration
$$
\Pi: \CL_{0}^{\text{\rm Carnot}}(J^1B, H;o_{J^1B}) \to B
$$
with $\Pi = \pi_B \circ \ev_{1}$ for the evaluation map $\ev_{1}$ given by $\ev_{1}(\gamma) = \gamma(1)$.

The following theorem is the contact counterpart of Weinstein's variational construction of
a generating function from Hamilton's action functional in symplectic geometry.
It also proves that the functional $\CA _H^{\rm CC} : \CL ^{\rm Carnot} (J^1B,H;o_{J^1B}) \to \R$
is a Morse family for \emph{any} Hamiltonian $H$.
 
\begin{thm}\label{thm:contact-analog}
Consider the space $\CL_{0}^{\text{\rm Carnot}}(J^1B, H;o_{J^1B})$ and the diagram
$$
\xymatrix{\CL_{0}^{\text{\rm Carnot}}(J^1B, H;o_{J^1B}) \ar[d]^{\Pi} \ar[r]^>>>>>{\CA_H^{\rm CC}} &\R\\
B &
}
$$
for any $H$. Then the following holds:
\begin{enumerate}
\item The functional $\CA_H^{\rm CC} : \CL ^{\rm Carnot} (J^1B,H;o_{J^1B}) \to \R$ satisfies \eqref{eq:transversality-dvCSCF}
with $S$ replaced by $\CA_H^{\rm CC}$.
\item It generates $R = \psi_H^1(o_{J^1B})$ with respect to the fibration 
$$
\Pi: \CL_{0}^{\text{\rm Carnot}}(J^1B, H;o_{J^1B}) \to B.
$$
\end{enumerate}
\end{thm}

\subsection{Broken trajectory approximations and contact Hamiltonian calculus}
\label{subsec:approximation}

After we identify the correct action functional generating the given contact Hamiltonian isotope
of the Legendrian submanifold $o_{J^1B}$, we now apply the broken trajectory approximation
to  give a `canonical' construction of a $\GFQI$ of any Hamiltonian isotope of the zero section 
in the one-jet bundle. Motivated by Theorem \ref{thm:contact-analog} and Chaperon \cite{chaperon}, 
Laudenbach-Sikorav's construction 
\cite{laud-sikorav} in the symplectic case, we perform the contact counterpart of Laudenbach-Sikorav's 
broken trajectory approximation for the effective action functional $\CA_H^{\text{\rm CC}}$.

We start by briefly recalling the construction given in \cite{chaperon,laud-sikorav}. 
To avoid notational complications, we consider the case of $T^*T^n$ as in \cite{chaperon}
and utilize the `linear' structure of $T^n$ in the following summary.
We take a partition
$$
0 = t_0 < t_1 < \cdots < t_N = 1, \quad t_i = \frac{i}{N}
$$
and consider the set of piecewise smooth paths $\ell:[0,1] \to T^*T^n$
with jump at each $t_i$ 
satisfying the initial condition $\ell(0) \in o_{T^*T^n}$, smooth on each semi-open interval
$[t_i,t_{i+1})$ for $i =0, \cdots N-1$ and having the left limit $\lim_{t \nearrow t_{i+1}}\ell(t)$.  
We denote
$$
\ell(t_i) = (q_i^+, p_i^+), \quad \lim_{t \nearrow t_i}\ell(t) = (q_i^-, p_i^-).
$$
Denote $(X_i,P_i) = (q_i^+, p_i^+) -  (q_i^-, p_i^-)$ and regard $(X,P) = \prod_{i=1}^{N-1} (X_i,P_i)$
as an element of $T^n \times (T^*T^n)^{N-1}$ whose dimension is given by $n + 2(N-1) n < \infty$.
To get rid of 
the additional freedom, we specify the path segment $\ell_{[t_i,t_{i+1}]}$ to be
\be\label{eq:elli}
\ell_i(t) = \phi_H^t(q_i^+,p_i^+), \quad t \in [t_i, t_{i+1}].
\ee
Then we denote by $\ell_{(X,P)}$ the corresponding piecewise smooth paths.
We denote the corresponding set of piecewise smooth paths with jumps by $E_N$,
define the projection map $\pi_N: E_N \to T^n$ to be
$$
\pi_N(X,P) = \pi_{T^*T^n}(q_N^-, p_N^-) = \pi_{T^*T^n}\left(\ell_{(X,P)}(1)\right)
$$
and the function $S_N: E_N \to \R$ by 
\be\label{eq:SN}
S_N(X,P) :=  \CA^{\text{\rm cl}}_H(\ell_{(X,P)})
\ee
where $\CA^{\text{\rm cl}}_H$ is Hamilton's action functional given by
$$
\CA^{\text{\rm cl}}_H(\zeta) = \int \zeta^*\theta - H(t,\zeta(t))\, dt
$$
for a path $\zeta:[0,1] \to T^*T^n$ with the Liouville one-form $\theta = p\, dq$ on $T^*T^n$. In this way,
we have a diagram \eqref{eq:relative-Morse} for each given $N$. These being said
the main result of \cite{chaperon,laud-sikorav} is that this diagram generates the 
Lagrangian submanifold $\phi_H^1(o_{T^*T^n})$ for all sufficiently large $N$.

The following point is the main difference  of our construction from that of 
\cite{chaperon}, \cite{laud-sikorav}
in addition to our usage of rather peculiar type of contact action functional $\CA_H^{\text{\rm CC}}$
to handle the large degeneracy of the more standard action functional.

\medskip

\noindent{\bf Upshot:} \emph{One crucial deviation of our broken trajectory approximation
 from that of \cite{chaperon}, \cite{laud-sikorav} is that in the current contact case, 
we also need to insert some additional horizontal trajectories
in between every consecutive pair of contact Hamiltonian trajectories by bisecting
each interval $[t_i,t_{i+1}]$ for $i =1, \cdots, N-1$}. (We refer to Section \ref{sec:approximation} for 
the details. See Figure \ref{fig:broken-traj} for its visualization.)

\medskip

So the relevant broken trajectories are the consecutive unions 
of the concatenated pairs of contact Hamiltonian trajectories and horizontal paths. 
This whole process is canonical modulo the degrees of freedom, again given by 
$$
\dim T^n + 2(N-1) \dim T^n = n+2(N-1)n.
$$

We summarize our construction into the following theorem
somewhat vaguely here,  but the full explanations and details of the algorithm appear
in Section \ref{sec:approximation} - \ref{sec:definitionGFQI} and its proof in Section \ref{sec:proofGFQI}.

\begin{thm} 
Let $B$ be any compact manifold. Then there is a `canonical' algorithm 
of a finite dimensional Morse family \eqref{eq:relative-Morse} with $\CE = E^N$, 
$S = S_N$ and $\pi = \pi_N$ which `approximates'
the effective action functional $\CA_H^{\text CC}$ that gives rise to a
$\GFQI$ of any contact Hamiltonian isotope of the zero section $o_{J^1B}$.
\end{thm}

Largely due to the nontrivial presence of \emph{conformal exponent} 
of the contactomorphism unlike the symplectomorphism, the relevant
contact Hamiltonian calculus is significantly more difficult than the symplectic one.
We would like to recommend readers to compare the degree of complexity of the calculations performed here 
and those in \cite{laud-sikorav}.
It makes the computations entering in the critical point analysis of 
the action functional or its broken trajectory approximations much harder than
the symplectic case of \cite{chaperon}, \cite{laud-sikorav}. Unlike the symplectic case,
it also requires us to perform very careful and precise tensorial calculations 
based on the systematic contact Hamiltonian calculus. Such a \emph{systematic calculus}
has been developed and used by the first named author in relation to
the series of  works on contact instantons starting from \cite{oh:contacton-Legendrian-bdy}.
(This calculus is a continuation of those adopted in \cite{BCT} with all the sign conventions 
both in symplectic and contact Hamiltonian calculi which are also compatible with those
listed in {\bf List of Conventions} of \cite{oh:book1} in the symplectic case.)

The organization of the paper is now in order. In Section \ref{sec:preliminaries}, 
we collect some useful results concerning the contact Hamiltonian dynamics
which will enable us to systematically study the calculus of variations of the perturbed
action functional $\CA_H$. 
We mostly follow the exposition of 
\cite{BCT}  on contact Hamiltonian mechanics,
and \cite{oh:contacton-Legendrian-bdy} for a more
extensive systematic exposition on contact Hamiltonian calculus. In Section \ref{sec:Carnot-pathspace}, 
we give the proof of Theorem \ref{thm:Carnot}. After these preliminaries on general contact
Hamiltonian geometry are given, the paper is divided into two parts: Part I contains the formulation 
of the contact analogue of Weinstein's de-approximation, and the proof of Theorem \ref{thm:contact-analog}.
Part II then contains the aforementioned broken trajectory approximation of
the effective action functional $\CA_H^{\text{\rm CC}}$
as a canonical construction of a $\GFQI$ for any Hamiltonian isotope of the zero section 
in the one-jet bundle. Finally we collect some background material in Appendix \ref{sec:Carnot-horizontal}
 and technical calculations in Appendix \ref{sec:deltax} and \ref{sec:invertibility} 
 which are postponed from the main text of the paper.

\bigskip

\noindent{\bf Conventions and Notations:}

\medskip

\begin{enumerate}
\item {(Contact Hamiltonian)} The contact Hamiltonian of a time-dependent contact vector field $X_t$ is
given by
$$
H: = - \lambda(X_t).
$$
We denote by $X_H$ the contact vector field whose associated contact Hamiltonian is given by $H = H(t,y)$, and its flow by
$ \psi^t_H $.
\item When $\psi = \psi^1_H$, we say $H$ generates $\psi$ and write $H \mapsto \psi$.
\item We use the notation $\phi_H^t: = \psi_H^t (\psi_H^1)^{-1}$.
(We warn readers that we also use the same notation $\phi_H^t$ for the symplectic Hamiltonian flow
to be consistent with the notations used in \cite{oh:jdg,oh:cag}. The meaning thereof should be
clear from the context, hopefully.)
\item We will try to consistently use the following notations whenever appropriate:
\begin{itemize}
\item $(q,p,z)$ a point of $J^1B$ or the canonical coordinates thereof
\item $x = (q,p)$ a point of $T^*B$,
\item $y = (x,z)$ a point in $J^1B$,
\item $e$ or $(q,e)$ a point in a vector bundle $E \to B$ with $q = \pi_E(e)$.
\end{itemize}
\item {(Reeb vector field)} We denote by $R_\lambda$ the Reeb vector field for the contact form $\lambda$
and its flow by $\phi^t_{R_\lambda}$.
\item $R$: a Legendrian submanifold.
\item $\CL _0 (J^1B; R)$:  The space of paths $\gamma$ with $\gamma(0) \in R$.
\item $ \CL^\pi_0 (J^1B; R)$: The space of \emph{horizontal paths} $\gamma$ with $\gamma(0) \in R$.
\item $\CL^{\rm Carnot}_0 (J^1B,H; R)$: The space of \emph{$H$-Carnot paths} 
(or simply \emph{Carnot paths}) $\gamma$ with 
$\gamma(0) \in R$.
\end{enumerate}

\section{Preliminaries}
\label{sec:preliminaries}

Let $(M,\xi)$ be a cooriented contact manifold and let $\lambda$ be a
contact form with $\xi = \ker \lambda$. Denote by $\Cont(M,\xi)$ (resp. $\Cont_0(M,\xi)$)
the set of contact diffeomorphisms (resp. the identity component thereof).

\begin{defn} For a given coorientation preserving contact diffeomorphism $\psi$ of $(M,\xi)$
we call the function $g$ appearing in
$$
\psi^*\lambda = e^g \lambda
$$
the \emph{conformal exponent} for $\psi$ and denote it by $g = g_\psi$.
\end{defn}

\begin{defn}\label{defn:Hamiltonian} A vector field $X$ on $(M,\xi)$ is called \emph{contact} if
there exists a smooth function $f: M \to \R$ such that
$$
\CL_X \lambda = f \lambda.
$$
The associated function $H$ defined by
\be\label{eq:contact-Hamiltonian}
H = - \lambda(X)
\ee
is called the \emph{contact Hamiltonian} of $X$. We also call $X$ the
contact Hamiltonian vector field associated to $H$.
\end{defn}
A straightforward calculation shows
$$
f = - R_\lambda[H].
$$
Here we denote by $R_\lambda$ the Reeb vector field of $\lambda$, i.e., the unique vector
field $X$ satisfying $X \rfloor d\lambda = 0, \, X\rfloor \lambda = 1$. It is nothing but the
Hamiltonian vector field associated to the constant function $H \equiv -1$.

The following systematic notation will be useful to designate the Hamiltonian 
to a given contact isotopy as in symplectic topology \cite{oh:hameo1}. As practiced in 
the first-named author's articles  \cite{oh:contacton-Legendrian-bdy,oh:entanglement1},
we adopt the following definition.

\begin{defn}[Developing map]\label{defn:developing-map} Let $\lambda$ be a contact form of $(M,\xi)$. 
We define
\be\label{eq:Devlambda}
\Dev_\lambda: \CP(\Cont(M,\xi)) \to C^\infty([0,T] \times M,\R)
\ee
by assigning its $\lambda$-contact Hamiltonian functions 
\be\label{eq:lambda-contact-Hamiltonian}
\Dev_\lambda(\ell)(t,x): = -\lambda\left(\frac{\del \ell}{\del t}(t,\ell_t^{-1}(x))\right).
\ee
\end{defn}
Unravelling this definition, we have $\Dev_\lambda(\ell)(t,x) = H(t,x)$ if $X_t$ is the
contact vector field as in Definition \ref{defn:Hamiltonian}. 
We denote by $X \mapsto \psi$ if $\psi = \psi_H^1$.

For a given general function $H$, the associated contact Hamiltonian vector field
$X_H$ has decomposition
$$
X_H = X_H^\pi - H R_\lambda \in \xi \oplus \R \langle R_\lambda \rangle
$$
from \eqref{eq:XH} where the projection $X_H^\pi$ to $\xi$ is uniquely determined by the equation
$$
X_H^\pi \rfloor d\lambda = dH - R_\lambda[H] \lambda
$$
or equivalently
\be\label{eq:dH=}
dH = X_H \rfloor d\lambda + R_\lambda[H] \lambda.
\ee
\begin{exm}\label{exm:XH-in-Darboux}
Let $H:\R^{2n+1} \to \R$ be a smooth function on $\R^{2n+1} = J^1\R^n$ (or in Darboux coordinates).
Then the contact Hamiltonian vector field $X_H$ is given by
\bea\label{eq:XH-in-Darboux}
X_H & =& X_H^\pi -H R_\lambda = 
\sum_{i=1}^n \left(\frac{\del H}{\del p_i} \frac{D}{\del q_i} - \frac{D H}{\del q_i} \frac{\del}{\del p_i}\right)
- H \frac{\del}{\del z} \nonumber \\
& = & \sum_{i=1}^n \frac{\del H}{\del p_i} \frac{\del}{\del q_i} -
\left(\frac{\del H}{\del q_i} + p_i \frac{\del H}{\del z}\right)\frac{\del}{\del p_i}
+ \left(\left\langle p, \frac{\del H}{\del p}\right \rangle  - H\right)\frac{\del}{\del z}
\nonumber \\
&{}& 
\eea
where the set $\left\{ \frac{D}{\del q_i}, \frac{\del}{\del p_i}, \frac{\del}{\del z} \right\}$
with $\frac{D}{\del q_i} : = \frac{\del}{\del q_i} + p_i \frac{\del}{\del z}$ is the associated
Darboux frame, i.e., a frame of $TM$ which satisfies 
$$
R_\lambda = \frac{\del}{\del z}, \quad \xi = \span \left\{ \frac{D}{\del q_i}, \frac{\del}{\del p_i}\right\},
$$
and $d\lambda( \frac{D}{\del q_i}, \frac{\del}{\del p_i}) = \delta_{ij}$ for $1 \leq i, \, j \leq n$.
Therefore the associated contact Hamilton's equation is given by
\be\label{eq:Hamilton-eq-Darboux}
\begin{cases}
\dot q_i = \frac{\del H}{\del p_i}, \\
 \dot p_i = -\frac{D H}{\del q_i}  \left(= -\frac{\del H}{\del q_i} - p_i \frac{\del H}{\del z}\right), \\
 \dot z = \left\langle p, \frac{\del H}{\del p}\right \rangle  - H.
 \end{cases}
\ee
\end{exm}

\begin{lem}\label{lem:XH-decompose} Solving contact Hamilton's equation $\dot y = X_H(t,y)$ is equivalent to
finding $\gamma:\R \to M$ that satisfies
\be\label{eq:XH-decompose}
(\dot \gamma - X_H(t,\gamma(t)))^\pi = 0, \quad \gamma^*\lambda + H(t,\gamma(t))\, dt = 0.
\ee
\end{lem}

For  the later purpose, we introduce the following proposition providing 
an explicit relationship between the contact Hamiltonians and
the conformal exponents of a given contact isotopy.
\begin{prop}[\cite{oh:contacton-Legendrian-bdy}, Proposition 2.17]
Let $\Psi = \{\psi_t\}$ be a contact isotopy of $(M, \xi = \ker \lambda)$ with $\psi_t^* \lambda = e^{g_t}\lambda$
generated by $H = H(t,y)$. We write $g_\Psi (t,y) = g_t(y)$. Then
\be\label{eq:conformalRlambda}
\frac{\del g_\Psi}{\del t}(t,y) = - R_\lambda[H](t,\psi_t(y)).
\ee
In particular, if $\psi_0 = \id$,
$$
g_\Psi (t,y) = \int^t_0 - R_\lambda[H](u,\psi_u(y))\, du.
$$
\end{prop}
The same result, with different sign convention, is implicitly given in the proof of \cite[Corollary 2.3.2]{geiges}.
(See also \cite{bhupal} for the one-jet bundle case.)

To obtain a good variational problem for contact Hamiltonian dynamics, 
we use the following \emph{perturbed action functional}
introduced in \cite{oh:perturbed-contacton}.

\begin{defn}[Perturbed action functional]
Let $H = H(t,y)$ be a contact Hamiltonian on $M$
and $\phi_H^t$ be as in \eqref{eq:phiHt}.
Consider the free path space
$$
\CL(M) := C^{\infty}([0,1];M) = \{ \gamma: [0,1] \rightarrow M \}.
$$
We define a functional $\CA_H: \CL(M) \rightarrow \R$ given by
\be\label{eq:CAH}
\CA_H(\gamma) := \int_0^1 e^{g_{(\phi_H^t)^{-1}} (\gamma(t))} \gamma^* \lambda_H
= \int_0^1 e^{g_{(\phi_H^t)^{-1}} (\gamma(t))} (\lambda(\dot \gamma (t)) + H_t (\gamma (t))) dt
\ee
as in \eqref{eq:perturbed-action}.
Here $\lambda_H := \lambda + H dt$ with slight abuse of notation as in \cite{oh:perturbed-contacton}.
\end{defn}

The following lemma connects the perturbed action functional with the unperturbed (standard) one.

\begin{lem}[Lemma 2.2, \cite{oh:perturbed-contacton}]\label{lem:CAHu=CAw}
For given path $\gamma \in \CL(M)$, consider the path $\overline \gamma$ defined by
$$
\overline \gamma (t) := (\phi_H^t)^{-1}(\gamma(t)).
$$
Then we have
\be\label{eq:CAHu=CAw}
\CA_H(\gamma) = \CA_0(\overline \gamma)
\ee
where $\CA_0$ is the standard (unperturbed) contact action
functional given by $\CA_0(\gamma)  = \int_\gamma \lambda$.
\end{lem}
Following \cite{oh:cag} and \cite{oh:entanglement1}, we call the transformation $\gamma \mapsto \overline \gamma$
(resp. its inverse $\overline \gamma \to \gamma$) 
\emph{gauge transformations}. 
Since this transformation will enter in later discussion, we formalize its definition
now.
\begin{defn}[Gauge transformation]\label{defn:gauge-transform}
 Let $\gamma, \, \nu \in C^\infty([0,1], M)$.  We define the map
$\Phi_H: \gamma \mapsto \overline \gamma$
to be 
\be\label{eq:PhiH}
\Phi_H(\gamma)(t): = \overline \gamma(t) = (\phi_H^t)^{-1}(\gamma(t))
\ee
and denote by $\Psi_H=(\Phi_H)^{-1}$ its inverse which is given by
\be\label{eq:PsiH}
\Psi_H(\nu) (t)= \phi_H^t(\nu(t)) = \psi_H^t \left((\psi_H^1)^{-1}(\nu(t))\right)
\ee
\end{defn}
 
We now quote from \cite{oh:perturbed-contacton}
the following first variation formula of the action functional
\eqref{eq:perturbed-action} on the free path space $\CL(M)$. 

\begin{prop}[Proposition 2.3, \cite{oh:perturbed-contacton}]
For any vector field $\eta$ along $\gamma \in \CL(M)$, we have the first variation
\bea\label{eq:perturbed-1st-variation}
\delta \CA_H (\gamma)(\eta)
&=& \int_0^1 d\lambda((d \phi_H^t)^{-1}(\eta(t)), (d \phi_H^t)^{-1}(\dot \gamma - X_H(t,\gamma(t)))) 
\, dt \nonumber \\
& & + \lambda(\eta(1)) - e^{g_{\psi^1_H}(\gamma(0))} \lambda(\eta(0)).
\eea
\end{prop}

\begin{rem}
As in \cite[p.21]{arnold:book}, 
we follow the standard notation $\delta$ in the calculus of variations to denote the first variation, 
or the \emph{formal derivative}, of a functional.
In general, we reserve the notation `$d$' for the derivative in the more rigorous contexts.
\end{rem}

An immediate corollary of this first variation formula shows that the Legendrian boundary 
condition is a natural boundary condition for the action functional $\CA_H$ in that it kills the boundary contribution in the first variation:

\begin{cor}\label{cor:pi-critical}
Let $R$ be any given Legendrian submanifold of $(M,\xi)$. Consider the subset
$$
\CL_0(M;R): = \{ \gamma \in  \CL(M) \mid \gamma(0) \in R\}.
$$
Then we have  
$$ 
\delta \CA_H (\gamma)(\eta)
= \int_0^1 d\lambda\left(d(\phi_H^t)^{-1}(\eta(t)), d(\phi_H^t)^{-1}(\dot \gamma - X_H(t,\gamma(t)))\right) 
\, dt  + \lambda(\eta(1))
$$
on $\CL_0(M;R)$.
\end{cor}

\section{Carnot path space of translated horizontal paths}
\label{sec:Carnot-pathspace}

An immediate consequence of Corollary \ref{cor:pi-critical} is as follows.
By the vanishing of $\lambda$ on the Legendrian submanifold, we have derived that
$\delta \CA_H(\gamma)(\eta) = 0$ for all variations $\eta$ tangent to paths satisfying Legendrian boundary conditions if and only if
\be\label{eq:pi-critical}
((d\phi^t_H)^{-1}(\dot\gamma - X_{H_t}(\gamma(t))))^\pi = 0
\ee
i.e., it is equivalent to
$$
(d\phi^t_H)^{-1}(\dot \gamma - X_H(\gamma(t))) = f(t)R_\lambda(\overline{\gamma}(t))
$$
where  $f(t)$ is a smooth function.

This has led us to consider the following space

\begin{defn}[Carnot path space]\label{defn:Carnot-path}
Consider the subsets of $\CL(M)$ given by
$$
\CL^{\text{\rm Carnot}}(M,H)
= \{ \gamma : [0,1] \rightarrow M \; \mid \; \lambda_{\gamma}(\dot \gamma) + H(\gamma) = 0 \}.
$$
\end{defn}
The defining equation $ \lambda_{\gamma}(\dot \gamma) + H(\gamma) = 0$ is equivalent to 
the statement
$$
\dot \gamma(t) -X_H(\gamma(t)) \in \xi_{\gamma(t)}
$$
for all $t \in [0,1]$.
We note that \emph{since the distribution $\xi$ is not integrable}
a priori $\CL ^{\rm Carnot}(M, H)$ could be a pathological space, e.g., may
not carry the structure of Fr\'echet manifold on which we can do some differential calculus.
In this regard, we prove the following result which shows that $\CL ^{\rm Carnot} (M,H)$ carries a Fr\'echet manifold
structure and hence we can do the calculus of constrained variation of $\CA_H$ restricted thereto.

\begin{thm}\label{thm:Carnot-Frechet} The subset
$\CL^{\rm Carnot}(M,H) \subset \CL(M)$ is a Fr\'echet submanifold and so carries a natural smooth structure
induced from $\CL(M)$ for any given Hamiltonian $H = H(t,y)$.
\end{thm}

The rest of the section will be occupied by the proof of this theorem. For this purpose, we start with
the following lemma.

\begin{lem}\label{prop:Carnot-Frechet} Let $H = H(t,y)$ be any smooth Hamiltonian on $M$ and consider
the map
$$
\Upsilon: \gamma \in \CL(M) \mapsto \lambda (\dot \gamma) + H_t(\gamma) \in C^\infty ([0,1],\R).
$$
Then we have its derivative
\be\label{eq:DGamma}
d_\gamma \Upsilon(\eta) = d\lambda(\eta,\dot \gamma-X_{H_t}(\gamma)) + \frac{d}{dt}(\lambda(\eta)) + R_\lambda[H_t] \, \lambda(\eta).
\ee
\end{lem}

\begin{proof} Note that the codomain of $\Upsilon$ is the set $C^\infty([0,1],\R)$.
Denote a variation of $\gamma$ by $\eta$ and the differential of $\Upsilon$  by
\be\label{eq:Dgamma}
d_\gamma \Upsilon: T_\gamma \CL(M) \to C^\infty([0,1], \R).
\ee
Then we compute the derivative 
\bea\label{eq:dUpsiloneta}
d_\gamma \Upsilon(\eta) & = & \frac{d}{ds}\Big|_{s=0} \Upsilon(\gamma_s) 
=  \frac{d}{ds}\Big|_{s=0}\Big((\lambda(\dot \gamma_s)) + H(t, \gamma_s(t))\Big) \nonumber\\
& = &  \frac{d}{ds}\Big|_{s=0}\left(\lambda(\dot \gamma_s)\right) + dH_t(\gamma(t))(\eta(t)).
\eea
Utilizing the identity \eqref{eq:dH=}, we evaluate
\bea\label{eq:dHtetat}
dH_t(\gamma(t))(\eta(t)) & = & (X_{H_t} \intprod d\lambda + R_\lambda[H_t] \lambda)(\eta(t))
\nonumber \\
& = & d\lambda(X_{H_t}, \eta(t)) + R_\lambda[H_t]\lambda(\eta(t)).
\eea
It remains to  evaluate $\frac{d}{ds}\big|_{s=0}(\lambda(\dot \gamma_s))$. For this,
 we introduce
 the variation function $c(s,t): = \gamma_s(t)$ on $(-\epsilon,\epsilon) \times [0,1]$
 for some small $\epsilon > 0$ that satisfies
$$
c(0,t) = \gamma(t), \quad \frac{\del c}{\del s}(0,t) = \eta(t).
$$
Then we have
\bea\label{eq:CLs}
\frac{d}{ds}\Big|_{s=0}(\lambda(\dot \gamma_s)) & = & 
\frac{d}{ds}\Big|_{s=0}(\gamma_s^*\lambda(\del_t))
= \left(\frac{d}{ds}\Big|_{s=0}\gamma_s^*\lambda\right)(\del_t) \nonumber \\
& = & \CL_{\del_s}(\gamma_s^*\lambda)|_{s=0}(\del_t).
\eea
By Cartan's magic formula (applied to the Lie derivative over the map $c$), we have
$$
\CL_{\del_s}(\gamma_s^*\lambda)=  \del_s \rfloor d(\gamma_s^*\lambda) + d (\del_s \rfloor \gamma_s^*\lambda),
\quad \gamma_s(t) = c(s,t)
$$
By evaluating this against $\del_t$ at $s = 0$, we  obtain
\beastar
\CL_{\del_s}(\gamma_s^*\lambda)(\del_t)|_{s = 0}
& = & d\lambda(dc(\del_s),dc(\del_t))|_{s=0} + d(\lambda(dc(\del_s)))(\del_t)|_{s = 0} \\
& = & d\lambda(\eta,\dot \gamma) + \frac{d}{dt}(\lambda(\eta)).
\eeastar
By substituting this into \eqref{eq:CLs}, we get
$$
\frac{d}{ds}\Big|_{s=0}(\lambda(\dot \gamma_s)) =d\lambda(\eta,\dot \gamma) + \frac{d}{dt}(\lambda(\eta)).
$$
Then substituting this and \eqref{eq:dHtetat} into \eqref{eq:dUpsiloneta}, we have
derived \eqref{eq:DGamma} which finishes the proof of the lemma.
\end{proof}

By the expression of the derivative $d_\gamma \Upsilon$ in \eqref{eq:DGamma},
it follows that
it continuously extends to a bounded operator
\be\label{eq:completion}
d_\gamma \Upsilon: W^{1,2}(\gamma^*TM) \to L^2([0,1],\R).
\ee
We call this the $W^{1,2}$-completion of the derivative \eqref{eq:DGamma}.
Using this, we now prove the following important property of the map $\Upsilon$.

\begin{prop}\label{prop:submersion} The $W^{1,2}$ completion \eqref{eq:completion} of the
derivative $d_\gamma \Upsilon$ is a surjective map at all $\gamma$ in $\CL^{\text{\rm Carnot}}(M,H)$.
\end{prop}
\begin{proof} Let $\gamma$ be any such a curve and write $D_\gamma = d_\gamma \Upsilon$ as before.
We will prove surjectivity by Fredholm alternative by considering the $L^2$-pairing
$$
\langle\langle D_\gamma(\eta), g\rangle\rangle = \int_0^1 ( D_\gamma(\eta)\, g)\, dt
$$
for the test function $\eta \in W^{1,2}(\gamma^*TM)$ and a function $g \in L^2([0,1],\R)$ to study.

For this purpose, we assume $\langle\langle D_\gamma(\eta), g\rangle\rangle = 0$ for all $\eta \in W^{1,2}(\gamma^*TM)$.
We would like to conclude $g = 0$. For this purpose, it will be enough to consider \emph{smooth} test functions
$\eta$ WLOG. We compute
\beastar
0 & = & \int_0^1 ( D_\gamma(\eta)\, g)\, dt \\
& = & \int_0^1\left( d\lambda(\eta, \dot \gamma - X_H(t,\gamma)) + \frac{d}{dt}(\lambda(\eta)) + R_\lambda[H_t]\lambda(\eta)\right)\,
 g\, dt \\
& = &  \int_0^1 d\lambda(\eta, \dot \gamma - X_H(t,\gamma))\, g(t) \, dt
+ \int_0^1 \left(\frac{d}{dt}(\lambda(\eta)) + R_\lambda[H_t]\lambda(\eta)\right)\, g(t)\, dt.
\eeastar
By doing integration by parts, we rewrite the second integral into
\beastar
&{}& \int_0^1 \left(\frac{d}{dt}(\lambda(\eta)) + R_\lambda[H_t]\lambda(\eta)\right)\, g\, dt\\
& = &  \int_0^1 \left(-\frac{dg}{dt} + R_\lambda[H_t]g\right) \lambda(\eta) dt
+ g(1) \lambda(\eta(1)) - g(0)\lambda(\eta(0)).
\eeastar
Therefore we have derived
\beastar
0 & = &  \int_0^1 d\lambda(\eta, \dot \gamma - X_H(t,\gamma))\, g(t) \, dt\\
&{}& \quad +  \int_0^1 \left(-\frac{dg}{dt} + R_\lambda[H_t]g\right) \lambda(\eta) dt
+ g(1) \lambda(\eta(1)) - g(0)\lambda(\eta(0))
\eeastar
for all $\eta$. By writing $\eta = \eta^\pi + \lambda(\eta) R_\lambda$, we can separate the two lines in
our consideration.

The first line can be written as
$$
\int_0^1 d\lambda(\eta, \dot \gamma - X_H(t,\gamma))\, g(t) \, dt
= \int_0^1 d\lambda(\eta^\pi, g(t) (\dot \gamma - X_H(t,\gamma))^\pi) \, dt.
$$
Since $\eta^\pi \in \gamma^*\xi$ is arbitrary and by nondegeneracy of $d\lambda|_\xi$, we have derived
\be\label{eq:1st-line}
g(t) (\dot \gamma - X_H(t,\gamma))^\pi = 0.
\ee
The second line gives rise to
\be\label{eq:coker-eq}
-\frac{dg}{dt} + R_\lambda[H_t]g = 0, \quad g(1) = 0 = g(0)
\ee
since $\lambda(\eta)$ can be an arbitrary real valued function on $[0,1]$.
Note that the equation  from \eqref{eq:coker-eq}
 is a linear homogeneous first-order ODE for $g$ on $\R$. Therefore, from the initial condition $g(0) = 0$ alone,
we have derived $g \equiv 0$. 
By the Hahn-Banach theorem, this finishes the proof of surjectivity of the operator
$$
d_\gamma \Upsilon: W^{1,2}(\gamma^*TM) \to L^2([0,1],\R).
$$
\end{proof}

We now need to derive the following smooth counterpart from this proof.
The following theorem is an immediate corollary of Proposition \ref{prop:submersion} and
the regularity theorem for the inhomogeneous first-order linear ODE
$$
\frac{df}{dt} + R_\lambda[H_t] f = h
$$
on $\R$. (See \cite[Section 32]{arnold:ODE} for example.)
\begin{thm}\label{thm:submersion} The derivative  \eqref{eq:Dgamma}
is surjective at every element $\gamma \in \CL^{\text{\rm Carnot}}(M,H)$.
\end{thm}
\begin{proof} Let $g \in C^\infty([0,1],\R)$. By Proposition \ref{prop:submersion},
we have an element $\eta \in W^{1,2}(\gamma^*TM)$ such that
$d_\gamma \Upsilon(\eta) = g$, i.e.,
\be\label{eq:surjectivity-eq}
d\lambda(\eta,\dot \gamma-X_{H_t}(\gamma)) + \frac{d}{dt}(\lambda(\eta)) 
+ R_\lambda[H_t] \, \lambda(\eta) = g.
\ee
We will adjust $\eta$ to another \emph{smooth solution} $\eta'$.

Take any smooth section $\zeta$ satisfying $\zeta(0) = \eta(0)$.
We then consider the function $g'$ defined by
$$
g' := d\lambda(\zeta, \dot \gamma - X_{H_t}(\gamma))
$$
which is smooth by definition. Then we solve the equation
\be\label{eq:1st-order-ODE}
\begin{cases}
\frac{df}{dt} + R_\lambda[H_t] \, f = g - g',\\
f(0) = \lambda(\eta(0))
\end{cases}
\ee
for $f$. Here we remark that since $\eta \in W^{1,2}(\gamma^*TM)$
it is continuous by the Sobolev embedding theorem and hence the value $\eta(0)$ has well-defined 
meaning.
Since $g - g'$ is smooth by hypothesis on $g$ and by the choice of $g'$,
the unique solution $f$ of \eqref{eq:1st-order-ODE} is smooth by the continuity
theorem of the first order ODE: For this, we observe that the 
directional derivative $R_\lambda[H_t]$ is a smooth function of $t$.

Now we denote by $\zeta^\pi$ the $\xi$-component of $\zeta$ which is smooth.
Then  we consider the smooth section $\eta' \in \Gamma(\gamma^*TM)$ defined by
$$
\eta' = \zeta^\pi + f R_\lambda.
$$
By the definition of $\eta'$, it is smooth and satisfies 
\beastar
d_\gamma\Upsilon(\eta') & = &  d\lambda(\eta',\dot \gamma-X_{H_t}(\gamma)) + \frac{d}{dt}(\lambda(\eta')) 
+ R_\lambda[H_t] \, \lambda(\eta') \\
& = &  d\lambda\left(\zeta^\pi + f R_\lambda ,\dot \gamma-X_{H_t}(\gamma)\right) 
+ \frac{d}{dt}\left(\lambda(\zeta^\pi + f R_\lambda)\right) 
+ R_\lambda[H_t] \, \lambda(\zeta^\pi + f R_\lambda) \\
& = & d\lambda(\zeta^\pi ,\dot \gamma-X_{H_t}(\gamma)) 
+ \frac{d}{dt}(\lambda(f R_\lambda)) 
+ R_\lambda[H_t] \, \lambda( f R_\lambda) \\
& = & g' + \left( \frac{df}{dt}
+ R_\lambda[H_t] \, f \right) = g' + (g-g') = g.
\eeastar
Furthermore, we check its initial value
$$
\eta'(0) = \zeta^\pi(0) + f(0) R_\lambda = \eta^\pi(0) + \lambda(\eta(0)) R_\lambda = \eta(0)
$$
where the second equality follows from the standing requirement $\zeta(0) = \eta(0)$ imposed in the beginning
and the initial condition for $f$ imposed in \eqref{eq:1st-order-ODE}.
This finishes the proof.
\end{proof}

\begin{rem} By a mollifier smoothing theorem (e.g., see \cite[Lemma 7.1]{gilbarg-trudinger})  of $W^{1,2}$ by $C^\infty$, we can actually 
take $\eta'$ as close to $\eta$ in the $W^{1,2}$ topology as we want. But we do not need this approximation result
for the proof of the above surjectivity.
\end{rem}

Theorem \ref{thm:submersion} in particular implies that $0 \in C^\infty([0,1],\R)$ is a regular value of the
(Fr\'echet) smooth map $\Upsilon$ the linearization of which is a linear first order ODE operator
which can be seen from the formula \eqref{eq:DGamma}.
Recalling by definition that 
$$
\CL^{\rm Carnot}(M,H) = \Upsilon^{-1}(0),
$$
the implicit function theorem \cite{sergeraert}
finishes the proof of Theorem \ref{thm:Carnot-Frechet}.

\part{Contact analogue to Weinstein's canonical generating function}

To motivate our discussion and for the convenience of readers,
we recall Weinstein's de-approximation \cite{alan:observation}
on the cotangent bundle in Section \ref{sec:alan-observation}.
In this de-approximation, he showed that the classical action functional 
$$
\CA_H^{\text{\rm cl}}: \CL_0(T^*B;o_{T^*B}) \to \R
$$
is a generating function of the Hamiltonian isotope
$$
L = \phi_H^1(o_{T^*B})
$$
where $\phi_H^t$ is the \emph{symplectic Hamiltonian} isotopy generated by
$H$. Here the domain, denoted by $\CL_0(T^*B;o_{T^*B})$, 
of the functional is the space consisting of the paths $\gamma: [0,1] \to T^*B$
satisfying the \emph{initial boundary condition} $\gamma(0) \in o_{T^*B}$.

In this Part, we will construct  the contact counterpart of this Weinstein's 
\emph{canonical, but infinite dimensional} generating function, on the one-jet bundle
$$
(J^1B, \lambda), \quad \lambda : = dz - \pi_{T^*B}^*\theta =dz - pdq,
$$
where $\pi_{T^*B}: J^1B \to T^*B$ is the canonical projection in \eqref{eq:piBpiT*B}.
 
\section{Weinstein's de-approximation of Laudenbach-Sikorav's construction}
\label{sec:alan-observation}

In this section, we explain Weinstein's de-approximation \cite{alan:observation} of Laudenbach-Sikorav's
`canonical' generating function of the time-one image $\phi_H^1(o_{T^*B})$  in the cotangent bundle. 
It is a generating function, canonically arising from the classical action functional $\CA_H^{\text{\rm cl}}$,
in the sense of the Morse family on the general fiber bundle
presented in  Definition \ref{defn:Morsefamily} whose fiber is an infinite-dimensional manifold.

A straightforward calculation gives rise to the first variation of the classical action functional
$$
\CA_H^{\text{\rm cl}}(\gamma) = \int \gamma^*\theta - H(t,\gamma(t))\, dt
$$
on the path space $\CL(T^*B)$ is given by
\beastar
\delta \CA_H^{\text{\rm cl}}(\gamma)(\eta) & =  &
\int_0^1 \omega_0(\eta(t), \dot \gamma(t) - X_{H_t}(\gamma(t)))\, dt
+ \langle \eta(1), \theta(\gamma(1)) \rangle - \langle \eta(0), \theta(\gamma(0)) \rangle\\
& = & \int_0^1 \omega_0(\eta(t), \dot \gamma(t) - X_{H_t}(\gamma(t)))\, dt
 + \langle d\pi (\eta(1)), \gamma(1) \rangle - \langle d\pi (\eta(0)), \gamma(0) \rangle.
 \eeastar
 (See  \cite{alan:Hamilton's}, \cite{arnold:book}, \cite{oh:jdg} for an explicit derivation.)

Then Weinstein directly put the action functional into the following \emph{infinite dimensional} framework on
the general fibration, replacing Laudenbach-Sikorav's finite dimensional approximation
of the action functional.

\begin{prop}[Weinstein \cite{alan:observation}]\label{lem:weinstein}
Let $H= H(t,x)$ be a time-dependent Hamiltonian on
 the cotangent bundle $T^*B$. Then the action functional $\CA_H^{\text{\rm cl}}$ 
restricted to the subset 
$$
\CL_0(T^*B,o_{T^*B}) = \{ \gamma: [0,1] \to T^*B \mid \gamma(0) \in o_{T^*B}\}
$$
of the full path space is a generating function of the time-one image $\phi_H^1(o_{T^*B})$ of the zero section
under the Hamiltonian flow of $H$.
\end{prop}
\begin{proof}
We consider the path space
$\CL_0(T^*B,o_{T^*B})$ and the diagram
\begin{equation}\label{eq:weinstein-diagram}
\xymatrix{\CL_0(T^*B,o_{T^*B}) \ar[d]^{\Pi} \ar[r]^(.7){\CA_H^{\text{\rm cl}}} & \R \\
B &
}
\end{equation}
where the fibration $\Pi: \CL_0(T^*B,o_{T^*B}) \to B$ is defined by
$\Pi: = \pi \circ \ev_1$ and $\ev_1: \CL_0(T^*B;o_{T^*B}) \to T^*B$ is the evaluation map
$$
\ev_1(\gamma) : = \gamma(1), \quad \text{\rm for } \, \gamma \in \CL_0(T^*B;o_{T^*B}).
$$
The above first variation formula for
$\eta$ tangent to $\CL_0(T^*B,o_{T^*B})$ satisfies $p(\gamma(0))= 0$ and hence
we have
$$
\delta \CA_H^{\text{\rm cl}}(\gamma)(\eta) =\int_0^1 \omega_0(\eta(t), \dot \gamma(t) - X_{H_t}(\gamma(t)))\, dt
 + \langle d\pi (\eta(1)), \gamma(1) \rangle.
$$
It follows from the definition of the fibration $\Pi$ the vertical derivative of $\CA_H$ for $\Pi$
is nothing but
$$
\delta^{\text{\rm vert}}\CA_H^{\text{\rm cl}}(\gamma)(\eta): = \int_0^1 \omega_0(\eta(t), 
\dot \gamma(t) - X_{H_t}(\gamma(t)))\, dt
$$
and hence the vertical critical point $\gamma$ is nothing but the Hamiltonian trajectory
satisfying
$$
\dot \gamma(t) = X_H(t,\gamma(t)), \qquad \gamma(0) \in o_{T^*B}.
$$
In other words, we have
$$
\Sigma_{\CA_H^{\text{\rm cl}}} = \{ \gamma \in \CL(T^*B) \mid \dot \gamma(t) = X_H(t,\gamma(t)),
 \, \gamma(0) \in o_{T^*B}\}.
$$
Then we obtain the horizontal derivative
$$
\delta^{\text{\rm hor}}\CA_H^{\text{\rm cl}}(\gamma) = \gamma(1).
$$
Since $\gamma(t)$ is a Hamiltonian trajectory whose final point lies in $\phi_H^1(o_{T^*B})$,
we can write
$$
\gamma(t) = \phi_H^t(\phi_H^1)^{-1}(x), \, \gamma(1) = x \in \phi_H^1(o_{T^*B}).
$$
This proves that $\CA_H^{\text{\rm cl}}|_{\CL_0(T^*B;o_{T^*B})}$ is a generating function of
$L: = \phi_H^1(o_{T^*B})$.
\end{proof}

The content and its proof of the above Weinstein's formulation of the action functional as a 
generating function can be succinctly summarized as follows:
\begin{itemize}
\item First solve the fiberwise (or vertical) critical point equation. This provides
the \emph{equation of motion}, Hamilton's equation $\dot x = X_H(t,x)$.
\item Then push forward the image of the differential $\delta\CA_H^{\text{\rm cl}}$ by
the map $\pi \circ \ev_1$ to the cotangent bundle $T^*B$. \emph{This provides
the final location of the particle.}
\item The resulting push-forward image is precisely $\phi_H^1(o_{T^*B})$.
\emph{This provides the final momentum of the particle.}
\end{itemize}

The rest of the paper will be occupied by our construction of
 the contact/Legendrian counterpart of Laudenbach-Sikorav's construction which we apply to our effective 
contact action functional $\CA _H^{\rm CC}$  on the one-jet bundle $J^1B$.

\section{Perturbed action functional as a contact generating function}
\label{sec:effective}

In this section, we assume that $H$ is a compactly supported Hamiltonian function on $J^1B$.
By specializing Definition \ref{defn:Carnot-path} to the case $M = J^1B$, we consider 
the subset
$$
\CL ^{\rm Carnot}_{0}(J^1B, H; R) := \CL^{\rm Carnot} (J^1B, H) \cap \CL_{0}(J^1B;R) 
$$
by considering the Legendrian boundary condition at $t = 0$.
More explicitly, we have 
\be\label{eq:CL-Carnot}
\CL ^{\rm Carnot}_{0}(J^1B, H; R) = \{ \gamma \in \CL (J^1B)  \mid \lambda_{\gamma}(\dot \gamma) + H(\gamma) = 0,
\, \gamma(0) \in R \}.
\ee

To construct a contact analogue to the above mentioned Weinstein's de-approximation \cite{alan:observation}
of the case of the 1-jet bundles $J^1B$, we need to consider the following action functional
modified from $\CA_H$. One reason for this modification is that
\emph{by definition we have
$$
\CA_H(\gamma) = 0
$$ 
for any contact Hamiltonian trajectory $\gamma$}.
We now  introduce  the aforementioned \emph{effective action functional}.

\begin{defn}[Effective action functional] We define $\widetilde \CA_H: \CL(J^1B)\to \R$ to be
\beastar
\widetilde{\mathcal{A}}_H(\gamma) & = & - \CA_H (\gamma) +
z(\gamma(1)) \nonumber \\
& = & - \int_0^1 e^{g_{(\phi_H^t)^{-1}}(\gamma (t))} (\lambda_{\gamma(t)} (\dot \gamma (t)) + H(\gamma(t))) dt +  z(\gamma(1)).
\eeastar
\end{defn}

\begin{rem}\label{rem:adoption} 
\begin{enumerate} 
\item The adoption of this functional for our purpose involves two reasons.
Note that in \emph{off-shell}, we can rewrite $\widetilde A_H$ as 
\beastar
- \CA_H (\gamma) + z(\gamma(1)) & = & \int  e^{g_{(\phi_H^t)^{-1}}} \left((\pi\circ \gamma)^*\theta - H(t, \gamma(t))\right)\, dt\\
&{}& + z(\gamma(1)) - \int e^{g_{(\phi_H^t)^{-1}}(\gamma(t))} \gamma^* dz.
\eeastar
Here
the first integral evaluated at $\gamma$ is the same as $\CA_H^{\text{\rm cl}}(\ell)$ if $\ell(t) = \pi(\gamma(t))$.
This is one reason for the adoption of the negative sign
in our definition of $\widetilde \CA_H$. (One should also use the fact that such a contact Hamiltonian flow is 
always \emph{strict}, i.e., $g_{\phi_H^t} \equiv 0$.) Furthermore we show in \cite[Proposition 2.11]{oh-yso:spectral}
that this functional exactly coincides with the \emph{effective action functional} used by
the first-named author in \cite{oh:jdg} when the contact Hamiltonian $H = H(t,q,p,z)$ on $J^1B$
is a lift of the symplectic Hamiltonian on the cotangent bundle $T^*B$. 
\item On the other hand, in \emph{on-shell}, i.e., for the curve satisfying the equation $\dot \gamma = X_H(t,\gamma(t))$
with $\gamma(0) \in o_{J^1B}$, we have $\CA_H(\gamma) = 0$, and so 
\be\label{eq:onshell-formula}
\widetilde \CA_H(\gamma) = z(\gamma(1)) = \int_0^1 \dot z(t)\, dt = 
\int_0^1 \left(p \frac{\del H}{\del p}(t,\gamma(t)) - H(t, \gamma(t))\right) \, dt
\ee
which reveals how the boundary value $z(\gamma(1))$ depends on $H$ and the history of the trajectory $\gamma$.
In \eqref{eq:onshell-formula}, the last equality follows from \eqref{eq:XH-in-Darboux}.

\end{enumerate}
\end{rem}

For later purpose, we write this functional in terms of the gauge-transformed path $\overline \gamma$
defined in \eqref{eq:PhiH}. 

\begin{prop} Let $\overline \gamma$ be the curve given by  $\overline \gamma(t) = (\phi_H^t)^{-1}(\gamma(t))$.
Then
$$
\widetilde \CA _H (\gamma) = - \int {\overline \gamma}^*\lambda + z(\overline \gamma (1)).
$$
\end{prop}

We will consider
the variational study of the critical point problem of $\widetilde \CA_H$  
as a constrained variational problem. In this regard, we
restrict $\widetilde \CA_H$ to the Carnot path space
$\CL ^{\rm Carnot}_{0}(J^1B, H; o_{J^1B})$ 
 to make Hamilton's equation
$\dot y = X_H(t,y)$ as the critical point equation or as
the \emph{equation of motion} in the physics term. 

\begin{defn}\label{defn:CAHCC}
We define the functional
$$
\CA _H^{\rm CC} : \CL_0 ^{\rm Carnot} (J^1B,H;o_{J^1B}) \to \R
$$
to be the restriction of $\widetilde \CA _H$ to $\CL_0 ^{\rm Carnot} (J^1B,H;o_{J^1B})$.
 (Here 'CC' stands for Carnot and Carath\'eodory.)
 \end{defn}

\section{Vertical critical point analysis}

We now specialize the discussion around \eqref{eq:tildeCAH} in Subsection \ref{subsec:action-functional} of the introduction
to the function 
$$
\CA_H^{\rm CC}:\CL_0^{\rm Carnot}(J^1B,H;o_{J^1B}) \to \R
$$
of the following diagram
\be\label{eq:diagram-ACCH}
\xymatrix{ \CL ^{\rm Carnot}_{0}(J^1B,H;o_{J^1B}) \ar[d]^{\Pi} \ar[r]^(0.8){\CA _H^{\rm CC}} & \R \\
B
}
\ee
for the setting of generating functions: the vertical arrow map
$$
\Pi := \pi_B \circ {\rm ev}_{1} : \CL_{0}^{\rm Carnot}(J^1B, H; o_{J^1B}) \rightarrow B, \qquad \gamma \mapsto 
\pi_B\circ\gamma(1)
$$
is a fibration and the horizontal arrow map is the real valued functional
 $\CA_H^{\rm CC}$.
We want to regard this diagram as the contact counterpart of the diagram associated to
Weinstein's de-approximation.

\subsection{Vertical critical points: equation of motion}
\label{sec:equationofmotion}

We first need to determine the vertical
critical points of the functional $\CA _H^{\rm CC}$ of
 the fibration $\Pi$ above. By definition, we have
$$
VT_\gamma \CE = T_\gamma(\Pi^{-1}(q))
$$
for all elements $\gamma \in \CL_0^{\rm Carnot}(J^1B,H;o_{J^1B})$
satisfying the boundary condition
$$
\pi_B\circ\gamma(1) = q
$$
at each given point $q \in B$.
Note that the fiber $\Pi^{-1}(q)$ at $q \in B$ is expressed by
\bea\label{eq:Pi-1q}
\Pi^{-1}(q) & = & \{ \gamma \in  \CL ^{\rm Carnot}_{0}(J^1B, H; o_{J^1B}) \; \mid \; \gamma(1) \in J^1_q B \} \nonumber\\
& =: & \CL ^{\rm Carnot} (J^1B,H; o_{J^1B}, J^1_q B).
\eea
\begin{rem} We would like to attract readers' attention that while $o_{J^1B}$ is a Legendrian submanifold,
the fiber $J_q^1B$ is a co-Legendrian submanifold of dimension $\dim B+1$. 
We refer to \cite{oh:entanglement1} for the definition.
\end{rem}
Note that the first variation formula \eqref{eq:perturbed-1st-variation} gives the vertical derivative
\bea\label{eq:dvertAA}
\delta^{\rm vert} \CA _H^{\rm CC}  (\gamma) (\eta)
&=& -\int_0^1 d\lambda ((d\phi_H^t)^{-1} \eta(t), (d\phi_H^t)^{-1} (\dot \gamma (t) - X_H (\gamma (t)))) dt  
\nonumber\\
&{}& + e^{g_{(\psi^1_H)^{-1}}(\gamma(0))} \lambda_{\gamma(0)} (\eta(0)) - \lambda_{\gamma(1)} (\eta(1))
+ dz_{\gamma(1)} (\eta(1)) \nonumber\\
&{}& 
\eea
for a variation $\eta \in T_\gamma \Pi^{-1}(q)$.  Note that
\beastar
- \lambda_{\gamma(1)} (\eta(1)) + dz_{\gamma(1)} (\eta(1)) 
& = & \theta(\pi_{T^*B}(\gamma(1))) (d\pi_{T^*B}(\eta(1))\\
& = & \pi_{T^*B}(\gamma(1))(d\pi_B(\eta(1))= \langle p(\gamma(1)), d\pi_B(\eta(1)) \rangle.
\eeastar
On the other hand we have
\be\label{eq:eta0-eta1}
\eta (0) \in To_{J^1B} \subset \xi, \qquad  \eta (1) \in T(J_q^1B)
\ee
by definition and so we obtain
$$
\lambda (\eta (0)) = 0, \qquad \langle p(\gamma(1)), d\pi_B(\eta(1)) \rangle = 0
$$
for all $\eta \in T\Pi^{-1}(q)$.

Combining the above discussion, the first variation formula is reduced to
$$
\delta^{\rm vert} \CA _H^{\rm CC}  (\gamma) (\eta)
= -\int_0^1 d\lambda 
\left((d\phi_H^t)^{-1} \eta(t), (d\phi_H^t)^{-1} (\dot \gamma (t) - X_H (\gamma (t)))\right) dt. 
$$
It means that
\bea \label{eq:vert-crit-pt}
\delta^{\rm vert} \CA _H^{\rm CC}  (\gamma) \equiv 0
&\Longleftrightarrow& ((d\phi_H^t)^{-1}(\dot \gamma (t) - X_H (\gamma(t)) ) )^\pi = 0 \nonumber \\
&\Longleftrightarrow& (d\phi^t_H)^{-1}(\dot \gamma (t) - X_H (\gamma(t))) 
= f(t) R_\lambda (\overline{\gamma} (t))
\eea
for some function $f=f(t)$.
On the other hand, by definition any element $\gamma \in \CL_{0}^{\rm Carnot}(J^1B, H; o_{J^1B})$ 
satisfies
\be
0 =  \lambda \left(\dot \gamma (t)) + H_t (\gamma (t)\right).
\ee
Furthermore, we also have
$$
\lambda(\dot \gamma (t)) + H_t (\gamma (t)) =  \lambda\left(\dot \gamma (t) - X_H (\gamma (t))\right)
$$
since $\lambda(X_H) = - H$ by definition of $X_H$. By combining  the last three equalities, we obtain
$$
0=  \lambda\left(\dot \gamma (t) - X_H (\gamma (t))\right) = \lambda \left(f(t) d\phi^t_H(R_\lambda(\overline{\gamma}(t)))\right) = f(t) e^{g_{\phi_H^t} (\overline \gamma (t))}
 $$
and hence  $ f(t) \equiv 0 $.

We summarize the above discussion into the following.
\begin{prop}[Equation of motion]
\label{prop:equationofmotion}
An element
$\gamma \in \CL ^{\rm Carnot}_{0} (J^1B,H; o_{J^1B})$ is a vertical critical point of
$\CA _H^{\rm CC}$ if and only if it is a contact Hamiltonian chord,
that is, $\gamma$ satisfies
$$
\dot \gamma (t) = X_H(\gamma (t)), \quad \gamma(0) \in o_{J^1B}.
$$
\end{prop}

\subsection{Vertical critical point set and horizontal derivative}

We now consider the vertical critical point set of the diagram \eqref{eq:diagram-ACCH}, which 
we denote by
\beastar
\Sigma_{\CA _H^{\rm CC}} & = & 
\left \{ \gamma  \in \CL ^{\rm Carnot}_{0} (J^1B,H; o_{J^1B}) \; \mid \; 
\delta^{\rm vert} \CA _H^{\rm CC} (\gamma) = 0 \right \}
\\
& = & \left\{ \gamma  \in \CL (J^1B) \; \mid \; \dot \gamma = X_H (\gamma), \; 
\gamma(0) \in o_{J^1B} \right \}.
\eeastar
We consider the restriction of $\Pi$ in \eqref{eq:diagram-ACCH} to 
the subset $\Sigma_{\CA _H^{\rm CC}}$.

We write the fibration $\Pi$ given in \eqref{eq:diagram-ACCH} 
as $\Pi:\CE \to B$ with
$$
\CE: = \CL^{\text{\rm Carnot}}_0(J^1B,H;o_{J^1B}) 
$$
and take an Ehresmann connection
$$
T\CE = VT \CE \oplus HT \CE
$$
where the vertical tangent space of $\Pi$ at $\gamma$ 
is nothing but $\ker d_\gamma \Pi$, and $HT_\gamma \CE \cong T_{\pi_B(\gamma(1))}B$.
We have the associated splitting of the cotangent bundle
$$
T^*\CE = \left(HT\CE\right)^\perp \oplus \left(VT\CE\right)^\perp
$$
where we denote by
$$
\left(HT_\gamma\CE\right)^\perp, \, \,  \left(VT_\gamma\CE\right)^\perp \subset T_\gamma^*\CE
$$
the annihilators of $HT_\gamma\CE, \, \,  VT_\gamma\CE \subset T_\gamma \CE$ respectively.
We mention that 
via the map
$d(\pi_B \circ \ev_1)$,
$\left(HT_\gamma\CE\right)^\perp$ can be identified with the \emph{vertical cotangent
space}, at least formally.

\begin{defn}[The vertical cotangent bundle $\CV \to \CE$]
We form the union 
$$
\CV = \bigcup_{\gamma \in \CE} \{\gamma\} \times \CV_\gamma
$$
and regard it as a vector bundle over $\CE$
with its fiber given by the \emph{vertical cotangent space} identified
with the fiber
$$
\CV_\gamma :=
\left(VT_\gamma
\CL^{\text{\rm Carnot}}_0(J^1B,H;o_{J^1B})\right)^*.
$$
\end{defn} 
We regard the map
$$
\gamma \mapsto \delta^{\rm vert} \CA _H^{\rm CC}(\gamma)
$$
as a section of the (infinite dimensional) vector bundle $\CV \to \CE$.
Obviously if the map
$\delta^{\rm vert} \CA _H^{\rm CC}: \CE \to \CV$
is transversal to the zero section, then
$\Sigma_{\CA _H^{\text{\rm CC}}}$
becomes a smooth embedded submanifold of dimension $n$.

Even without the
transversality hypothesis, we can still say something about
the image by the $\delta \CA_H^{\text{\rm CC}}$ of $\Sigma_{\CA_H^{\text{\rm CC}}}$ as follows.
Note that at each $\gamma \in \Sigma_{\CA _H^{\rm CC}}$, we derive
\bea
\delta (\CA _H^{\rm CC} |_{\Sigma_{\CA _H^{\rm CC}}}) (\gamma) (\eta)
&=& p (\gamma(1)) dq_{\gamma(1)} (\eta(1))\label{eq:T*B-projection}\\
\CA _H^{\rm CC} (\gamma) & = & - \CA _H (\gamma) + z(\gamma(1)) = z(\gamma(1))
\label{eq:z-projection}
\eea
from the first variation formula \eqref{eq:perturbed-1st-variation} and the definition of $\CA_H^{\text{\rm CC}}$.
Therefore the map
$$
\iota_{\CA _H^{\rm CC}} : \Sigma_{\CA _H^{\rm CC}} \rightarrow J^1B
$$
is well-defined and written as
\beastar
\iota_{\CA _H^{\rm CC}}(\gamma) & = & \left(\Pi (\gamma), \delta \CA _H^{\rm CC} (\gamma), \CA _H^{\rm CC} (\gamma)\right)\\
& = & (q(\gamma (1)), p(\gamma (1)), z(\gamma (1))) = \gamma (1) \in \psi^1_H (o_{J^1B}).
\eeastar
Since any point $y \in \psi^1_H (o_{J^1B})$ can be written as $y = \gamma(1)$  by definition,
we have proved the following general proposition \emph{without knowing that $\Sigma_{\CA _H^{\rm CC}}$
is a smooth manifold.}

\begin{prop}\label{prop:image-SigmatildeAH} For any $H$, the map $\iota_{\CA _H^{\rm CC}}$ is a bijective
continuous map such that
$$
\operatorname{Image}\left(\iota_{\CA _H^{\rm CC}}\right) = \psi^1_H (o_{J^1B})
$$
as a set.
\end{prop}

\section{Generating function transversality and the second variation}

Although the map $\iota_{\CA _H^{\rm CC}}$ is a bijective continuous map,
 we cannot talk about its smoothness as a map unless we have
some transversality result in advance so that $\Sigma_{\CA _H^{\rm CC}}$ carries
a smooth structure. This proof of smoothness  is now in order.

\subsection{Carnot path space and gauge transformation}

To study the aforementioned transversality, we need to consider the
second variation of the action functional  $\CA_H^{\text{\rm CC}}$.
We first recall the definition of the gauge transformation $\Psi_H=(\Phi_H)^{-1}$
given in Definition \ref{defn:gauge-transform}, and
the identity \eqref{eq:CAHu=CAw}
$$
\CA_H(\gamma) = \CA_0(\overline \gamma)
$$
for a general path $\gamma:[0,1] \to J^1B$. 
Therefore we will work with the expression $\CA_0(\overline \gamma)$
for our derivation of the second variation.

\begin{lem} A path $\gamma$ belongs to $\CL^{\text{\rm Carnot}}(J^1B,H)$ if and only if
the gauge-transformed curve
$$
\overline \gamma(t) = (\phi_H^t)^{-1}(\gamma(t)) \, \left(= \Phi_H(\gamma)(t)\right)
$$
is horizontal, i.e., $\dot{\overline \gamma}(t) \in \xi_{\overline \gamma(t)}$. 
\end{lem}
\begin{proof} A direct computation shows
$$
\frac{d}{dt} \overline \gamma(t) = d(\phi_H^t)^{-1}\left(\dot \gamma(t) - X_{H_t}(\gamma(t))\right).
$$
By evaluating $\lambda$ against this vector, we obtain
$$
\lambda(d(\phi_H^t)^{-1}(\dot \gamma(t) - X_{H_t}(\gamma(t)))) 
= e^{g_{(\phi_H^t)^{-1}}(\gamma(t))}\left(\lambda(\dot \gamma) + H_t(\gamma)\right).
$$
Therefore $\lambda(\dot{\overline \gamma}) = 0 $ if and only if $\lambda(\dot \gamma) + H_t(\gamma) = 0$.
This finishes the proof.
\end{proof}

Motivated by this lemma, we consider the set of horizontal curves 
\be\label{eq:CLpi}
\CL^\pi_0 (J^1B;\psi_H(o_{J^1B})) : = \{ \nu \mid \dot \nu \in \xi, \; \nu (0) \in \psi_H (o_{J^1B}) \}
\ee
so that the curves of the type $\nu = \overline \gamma$ is contained therein .
The  gauge transformation $\Psi_H$ and its inverse $\Phi_H$ relate the two path spaces
\be\label{eq:pi-to-Carnot}
\CL^\pi_0 (J^1B;\psi_H(o_{J^1B})) \stackrel{ \Psi_H}{\longrightarrow} 
\CL^{\rm Carnot}_0 (J^1B,H ;o_{J^1B})
\ee
in one-to-one correspondence, which is a Fr\'echet diffeomorphism with respect to
the smooth structures given by Theorem \ref{thm:Carnot-Frechet}.

\subsection{Constrained second variation}
\label{sec:second-variation}

In the rest of this section, we will mainly work with the space $\CL^\pi_0 \left(J^1B;\psi_H(o_{J^1B})\right)$
of horizontal paths utilizing the one-to-one correspondence \eqref{eq:pi-to-Carnot}.
We define
\be\label{eq:CApi}
\CA^\pi(\nu) := \CA_H^{\text{\rm CC}}\circ \Psi_H(\nu) = - \CA_0 (\nu) + z(\nu(1)) = z(\nu(1)).
\ee
Recalling $\phi_H^1 = \id$, the vertical tangent vector $\eta \in VT_\nu (\CL^\pi_0(J^1B;\psi_H(o_{J^1B})))$ 
satisfies the same boundary condition tangent to the fiber,
\be\label{eq:nu-bdycondition}
\eta(1) \in T_{\nu(1)}(J^1_qB)
\ee
where $q = \pi(\nu(1))$.

To take the second variation of $\CA^\pi$, we take an Ehresmann connection as follows.
Following \cite[Section 6]{LOTV}, on a one-jet bundle $J^1B$, we introduce a set of horizontal vector fields
$$
\frac{D}{\del q_i} := \frac{\del}{\del q_i} + p_i \frac{\del}{\del z}
$$
which are tangent to $\xi$. Together with $\frac{\del}{\del p_i}$, 
the set $\left\{ \frac{D}{\del q_i}, \frac{\del}{\del p_i} \right\}$ provides a Darboux frame of the
contact distribution $\xi$, i.e., a frame of $\xi$
satisfying $d\lambda( \frac{D}{\del q_i}, \frac{\del}{\del p_j}) = \delta_{ij}$
on the chart of canonical coordinates $(q_i,p_i,z)$. (In \cite{LOTV}, $\frac{D}{\del q_i}$ is denoted by $D_i$
and there is a sign difference in the choice of canonical symplectic form on $T^*B$.)

After gauge transformation, the diagram \eqref{eq:diagram-ACCH} is transformed to
\be\label{eq:diagram-Api}
\xymatrix{\CL ^\pi_{0}(J^1B,\psi_H^1(o_{J^1B})) \ar[d]^{\Pi\circ \Psi_H} \ar[r]^(0.8){\CA^\pi} & \R \\
B
}
\ee
\begin{choice}[Choice of Ehresmann connection of $\Pi\circ \Psi_H$] We take  
the Ehresmann connection of $\Pi\circ \Psi_H$ in the diagram \eqref{eq:diagram-Api} 
so that the horizontal tangent space of its associated splitting
$$
T(\CL^\pi_0(J^1B;\psi_H^1(o_{J^1B}))) = VT(\CL^\pi_0(J^1B;\psi_H^1(o_{J^1B}))) \oplus HT(\CL^\pi_0(J^1B;\psi_H^1(o_{J^1B})))
$$
is given by
$$
HT_\nu (\CL^\pi_0(J^1B;\psi_H(o_{J^1B}))) := 
(d\Pi_\nu|_{HT_\nu \CE})^{-1}\left( \span \left\{ \frac{D}{\del q_j} \right\} \right)
$$
where we have
$$
\span \left\{ \frac{D}{\del q_j} \right\} \subset \xi_{\nu(1)}, \; q=\pi_B(\nu(1))
$$
 locally (over the canonical Darboux chart at $(q,p,z)$), 
 and then glue them together over $B$. See  Subsection \ref{subsec:Morsefamily},
 especially the definition \eqref{eq:iotaCS-intro}, for a relevant discussion on this description of
 horizontal space.
\end{choice}

For the calculation of the covariant derivatives, we fix the triad metric defined by
\[
g(\cdot,\cdot) = d\lambda(\cdot, J\cdot) + \lambda \otimes \lambda,
\]
associated to a contact triad $(J^1B,\lambda, J)$, and use its associated Levi-Civita connection $\nabla$.
  
Then the variation $\eta$ satisfies the final boundary condition 
\be\label{eq:eta(1)}
\eta(1) \in \span \left\{ \frac{D}{\del q_j} \right\}\Big|_{\nu(1)} \cong 
HT_\nu (\CL^\pi_0(J^1B;\psi_H(o_{J^1B})))
\ee
in terms of the aforementioned local Darboux local frame. We mention the isomorphism
$HT_\nu (\CL^\pi_0(J^1B;\psi_H(o_{J^1B}))) \cong 
T_{\pi_B(\nu(1))}B$  is induced by the parallel transport map
followed by
$$
\frac{D}{\del q_j} \mapsto  \frac{\del}{\del q_j}
$$
by considering the final value problem
\be\label{eq:hor-cond}
\nabla_t\eta=0,\quad \eta(1)=v \in \span \left\{ \frac{D}{\del q_j} \right\}\Big|_{\nu(1)}.
\ee

To study the aforementioned transversality condition, we need to
study the second partial variation $\delta^{\rm hor}\delta^{\rm vert}\CA^\pi(\nu)$, 
first in the vertical and then in the horizontal direction, of $\CA^\pi$ at $\nu$.
  
 We start with the following formula which is a special case of the formula  \eqref{eq:DGamma}.
\begin{lem} Let $\nu$ be a Reeb chord, and  consider the tangent vector
 $$
 \eta \in T_\nu \CL^\pi(J^1B;\psi_H(o_{J^1B})).
 $$
 Then $\eta$ satisfies
\be\label{eq:tangenteq}
d\lambda(\eta(t),\dot \nu (t)) + \frac{d}{dt}(\lambda(\eta(t))) = 0.
\ee
\end{lem}
\begin{proof} 
This follows from the formula \eqref{eq:DGamma} of the covariant variation
with $H=0$,  which in turn follows from the defining equation of 
the space $\CL^\pi(J^1B;\psi_H(o_{J^1B}))$ consisting of horizontal paths.
\end{proof} 

For the calculation of the mixed second derivative
$$
\delta^{\rm hor}\delta^{\rm vert}\CA^\pi(\nu), 
$$
we first translate Proposition \ref{prop:equationofmotion}
into the following after gauge transformation.
\begin{cor}\label{cor:constantcurve} We have
$$
\Crit \delta^{\text{\rm vert}}\CA^\pi = \{\nu \mid \nu(t) \equiv \text{\rm constant}\}.
$$
\end{cor}
  
\begin{prop} Let $q \in B$ be a given point and let $\gamma \in
\CL_0^{\text{\rm Carnot}}(J^1B,H;o_{J^1B})$ and $\nu = \overline \gamma$ so that
 $\delta^{\text{\rm vert}}\CA^\pi(\nu)=0$,  i.e., $\nu$ is
a horizontal Reeb chord starting from $\psi_H^1(o_{J^1B})$ satisfying
$\nu(1) \in J^1_q B$. Then we have
\be
\delta^{\rm hor} \delta^{\rm vert} \CA^\pi(\nu) (\eta_1, \eta_2)  
= \int_0^1 d\lambda\left(\nabla_t \eta_1, \eta_2 \right)dt \label{eq:2nd-variation1}
\ee
for any $\eta_1 \in VT_{\nu} \CL^\pi_0 (J^1B;\psi_H(o_{J^1B})), \, \eta_2 \in HT_{\nu} \CL^\pi_0 (J^1B;\psi_H(o_{J^1B}))$.
\end{prop}
\begin{proof} Consider two variations of the types
\be\label{eq:vertical-var}
\eta_1 \in VT_{\nu} \CL^\pi_0 (J^1B;\psi_H(o_{J^1B})),
\ee
\be\label{eq:hor-var}
\eta_2 \in HT_{\nu} \CL^\pi_0 (J^1B;\psi_H(o_{J^1B}))
\ee
and represent the latter $\eta_2$ by the variation satisfying $\nabla_t\eta_2=0$.
Since $\nu$ is a constant curve
by Corollary \ref{cor:constantcurve}, we derive
$$
 \eta_2(t)\in \xi_y
 $$
 from the final boundary condition $\eta(1)\in \xi_y$ for $y=\nu(1)\equiv \nu(t)$. For we have
$$
\frac{d}{dt}\lambda_{\nu(t)}(\eta_2(t))=\frac{d}{dt}\lambda_y(\eta_2(t)) = \lambda_y(\nabla_t\eta_2)=0
$$
which implies $\lambda_y(\eta_2(t))\equiv 0$ and hence follows the claim.

We take a two-parameter variation $\Gamma: (-\delta,\delta)^2 \to \CL^\pi_0(J^1B;\psi_H(o_{J^1B}))$ 
with $\Gamma = \Gamma(s,u)$ such that
$$
\Gamma (0,0) = \nu, \quad
\left. \frac{\del \Gamma}{\del s}\right|_{(s,u) = (0,0)} = \eta_1, \quad
\left. \frac{\del \Gamma}{\del u}\right|_{(s,u) = (0,0)} = \eta_2.
$$
For convenience, we also write $\Gamma(s,u)(t)$ as $\Gamma(s,u,t)$.
Then using the definition \eqref{eq:CApi} we compute
\be\label{eq:dAAds}
\frac{\del}{\del s} \CA^\pi (\Gamma(s,0))
= dz\left(\frac{\del \Gamma}{\del s}(s,0,1)\right) 
= \lambda\left(\frac{\del \Gamma}{\del s}(s,0,1)\right).
\ee
Here the second equality follows from the condition that
$\frac{\del \Gamma}{\del s} (s,0,1)$  be a \emph{vertical} variation which implies
it satisfies the boundary condition \eqref{eq:nu-bdycondition}.
Then using \eqref{eq:dAAds}, we compute the second derivative
\be\label{eq:mixedd2CA}
\left. \frac{\del^2}{\del u\del s}\right|_{(s,u) = (0,0)} \CA^\pi(\Gamma(s,u))
= \left. \frac{\del}{\del u} \left[ \lambda\left(\frac{\del \Gamma}{\del s}(s,u,1)\right)\right] \right|_{(s,u)=(0,0)}.
\ee
From \eqref{eq:tangenteq}, we derive
\beastar
 \lambda \left(\frac{\del \Gamma}{\del s}(s,u,1)\right) \Big|_{(s,u)=(0,0)}
& = &  \int_0^1 \frac{\del}{\del t} \left(\lambda \left(\frac{\del \Gamma}{\del s}(s,u,t)\right) \right)\, dt
\Big|_{(s,u)=(0,0)}\\
& =  &   \int_0^1 d\lambda \left( \frac{\del \Gamma}{\del t}(s,u,t), 
\frac{\del \Gamma}{\del s}(s,u,t) \right) dt \Big|_{(s,u)=(0,0)}
\eeastar
where for the first equality we also used the vanishing
$$
 \lambda \left(\frac{\del \Gamma}{\del s}(s,u,0)\right) \Big|_{(s,u)=(0,0)}
 = \lambda(\eta_1(0)) = 0
$$
by the boundary condition \eqref{eq:eta0-eta1}.
Therefore by swapping $u$ and $s$ in \eqref{eq:mixedd2CA}, we can rewrite
\beastar
\left. \frac{\del^2}{\del u\del s}\right|_{(s,u) = (0,0)} \CA^\pi(\Gamma(s,u)) & = & 
\left. \frac{\del^2}{\del s\del u }\right|_{(s,u) = (0,0)} \CA^\pi(\Gamma(s,u))\\
& = & \left. \frac{\del}{\del s} \int_0^1 d\lambda\left( \frac{\del \Gamma}{\del t}(s,u,t), \frac{\del \Gamma}{\del u}(s,u,t) \right) dt \right|_{(s,u)=(0,0)}.
\eeastar
We compute
\beastar
&& \left. \frac{\del}{\del s} \int_0^1 d\lambda\left( \frac{\del \Gamma}{\del t}(s,u,t), \frac{\del \Gamma}{\del u}(s,u,t) \right) dt \right|_{(s,u)=(0,0)}\\
&= & \left. \int_0^1 (\nabla_s d\lambda)\left( \frac{\del \Gamma}{\del t}(s,u,t), \frac{\del \Gamma}{\del u}(s,u,t) \right) dt \right|_{(s,u)=(0,0)} \\
& & + \left. \int_0^1 d\lambda\left( \nabla_s \frac{\del \Gamma}{\del t}(s,u,t), \frac{\del \Gamma}{\del u}(s,u,t) \right) dt \right|_{(s,u)=(0,0)} \\
& & + \left. \int_0^1 d\lambda\left( \frac{\del \Gamma}{\del t}(s,u,t), \nabla_s \frac{\del \Gamma}{\del u}(s,u,t) \right) dt \right|_{(s,u)=(0,0)} \\
&=& \left. \int_0^1 d\lambda\left( \nabla_s \frac{\del \Gamma}{\del t}(s,u,t), \frac{\del \Gamma}{\del u}(s,u,t) \right) dt \right|_{(s,u)=(0,0)} \\
&=& \int_0^1 d\lambda\left( \left. \nabla_s \frac{\del \Gamma}{\del t}(s,u,t) \right|_{(s,u)=(0,0)}, \eta_2(t) \right) dt.
\eeastar
Here the second equality follows from the fact that
$\left. \frac{\del \Gamma}{\del t}\right|_{(s,u)=(0,0)} = \dot \nu$, and Corollary \ref{cor:constantcurve}
which implies that $\nu$ is a constant curve valued in $\psi^1_H(o_{J^1B}) \cap J^1_q B$.
Moreover, using the torsion freeness of $\nabla$, we derive
$$
\nabla_s \frac{\del \Gamma}{\del t} \Big|_{(u,s) = (0,0)} = \nabla_t \eta_1.
$$
Therefore this finishes the proof of \eqref{eq:2nd-variation1}.
\end{proof}

To verify our desired transversality, it is sufficient to prove 
surjectivity of the partial Hessian
$$
\delta^{\rm hor} \delta^{\rm vert} \CA^\pi (\nu),
$$
since $\delta^{\rm hor} \delta^{\rm vert} \CA^\pi(\nu)$ represents its horizontal part. This is verified in the following lemma.
\begin{lem}\label{lem:surjectivity}
$\operatorname{Coker} \delta^{\rm hor} \delta^{\rm vert} \CA^\pi (\nu) = 0.$
\end{lem} 
\begin{proof} Recall $\nu(t) \equiv y$ for all $t \in [0,1]$.
Suppose $\eta_2 \in HT_\nu \CL^\pi_0 (J^1B; \psi_H(o_{J^1B}))$ satisfies that 
$$
\delta^{\rm hor} \delta^{\rm vert} \CA^\pi (\eta_1,\eta_2) = 0
$$
for all $\eta_1 \in VT_\nu \CL^\pi_0 (J^1B; \psi_H(o_{J^1B}))$. Then we have
\be\label{eq:cokernel}
\int_0^1 d\lambda\left(\nabla_t \eta_1, \eta_2 \right)\, dt  = 0
\ee
for all $\eta_1\in VT_\nu \CL^\pi_0 (J^1B; \psi_H(o_{J^1B}))$. Using $\nabla_t \eta_2 = 0$, we rewrite
\beastar
0 & = &\int_0^1 d\lambda\left(\nabla_t \eta_1, \eta_2 \right)\, dt =
\int_0^1 \frac{d}{dt}\left(d\lambda(\eta_1, \eta_2) \right)\, dt\\
& = & d\lambda\left(\eta_1(1), \eta_2(1) \right) - d\lambda\left(\eta_1(0), \eta_2(0) \right).
\eeastar
Therefore we obtain
\be\label{eq:equality01}
 d\lambda\left(\eta_1(1), \eta_2(1) \right) = d\lambda\left(\eta_1(0), \eta_2(0) \right)
\ee
for all variation $\eta_1 \in VT_\nu \CL^\pi_0 (J^1B; \psi_H(o_{J^1B}))$. 
By definition, $\eta_i$ $i=1,2$ satisfy
$$
\eta_i(0) \in T_{\nu(0)}\psi_H^1(o_{J^1B})=T_y\psi_H^1(o_{J^1B})
$$
we have
$$
d\lambda(\eta_1(0),\eta_2(0))=0
$$
since $\psi_H^1(o_{J^1B})$ is Legendrian. Therefore it follows from \eqref{eq:equality01} and
the vanishing  $R_\lambda \intprod d\lambda = 0$ that
$$
d\lambda(\eta_1(1),\eta_2(1))=0
$$
for all $\eta_1(1) \in \xi_y$.
By the nondegeneracy of $d\lambda$ on $\xi_y$, this proves $\eta_2^\pi(1)=0$. 
 Since $\eta_2$ is assumed to be in
$HT_\nu \CL^\pi_0 (J^1B; \psi_H(o_{J^1B}))$, $\eta_2(1) = 0$.
Since $\eta_2$ is parallel, it implies $\eta_2= 0$. This finishes the proof of the lemma.
\end{proof}

\begin{cor}
The transversality of $\delta^{\rm vert}\CA^{\rm CC}_H$ to the zero section always holds. In particular,
$\Sigma_{\CA^{\rm CC}_H}$ is a smooth embedded submanifold of dimension $n$.
\end{cor}

\part{Contact analogue of Laudenbach-Sikorav's generating function}

In this part, we give a construction of `canonical' $\GFQI$
of any  Hamiltonian isotope of the zero section  $o_{J^1B}$ of
the one-jet bundle $J^1B$ as a finite dimensional approximation of
contact Hamiltonian trajectories as the contact counterpart of
Laudenbach--Sikorav's broken Hamiltonian trajectory approximation given in \cite{laud-sikorav}.
We will imitate the construction from \cite[Section 2]{laud-sikorav} by suitably adapting
it to the current contact setting of the one-jet bundle $J^1B \cong T^*B \times \R$.

Since the general procedure of the two constructions is similar \cite[Section 2]{laud-sikorav}, 
we will focus on the different points and  modifications we need to handle  the current 
Legendrian case.

\section{Generating functions of Lagrangian submanifolds: review}
\label{sec:gfqi-generation}

In this section we review basic results on the \emph{generating functions quadratic at infinity} (GFQI) of
Legendrian submanifolds from \cite{sikorav:cmh,viterbo:generating, theret} with applications to
spectral invariants of the Viterbo-type in our mind via the stable Morse theory.  GFQI is
a particular type of a Morse family in Definition \ref{defn:Morsefamily}
defined on a \emph{vector bundle} which enables one to carry out the Morse theory
via iterative handle attaching operations 
\cite{chaperon,laud-sikorav,sikorav:cmh,viterbo:generating,traynor,theret}
or via the  Morse homology construction as in \cite{milinkovic1,milinkovic2}.

\subsection{Summary of GFQI construction}
\label{subsec:gfqi}

Let $\pi_E: E \to B$ be a vector bundle and $S: E \to \R$ be a function that is quadratic at infinity, i.e., there exists a fiberwise nondegenerate quadratic function $Q:E \to \R$ such that $S = Q$ outside a compact subset of $E$.
We define the subset of $\mathbb{R}$,
\be\label{eq:spectrum}
\Spec(S) = \{S(e) \in \R \mid dS(e) = 0 \}
\ee
and call it the \emph{spectrum} of $S$.
Since $B$ is assumed to be compact and $S$ is quadratic at infinity,
$\Spec(S) \subset \R$ is a compact subset of measure zero in general.

Consider Legendrian submanifolds $R$ in the 1-jet bundle $J^1B$ of a smooth
manifold $B$. Denote by $\frak{Leg}(J^1B)$ (resp. $\frak{Leg}_0(J^1B)$) the set of Legendrian submanifolds of $J^1B$
(resp. those Legendrian isotopic to the zero section $R_0 = o_{J^1B}$).

A generic $R$ transversely intersects the `zero wall'
$$
Z:= o_{T^*B} \times \R \subset J^1B = T^*B\times \R.
$$
Let $\pi_R: R \to T^*B$ be  the natural Lagrangian projection which is the restriction to $R$ of $\pi:J^1B \to T^*B$.
The map $\pi|_R: R \to T^*B$ is an exact Lagrangian immersion.
We denote by
$$
\frak{Leg}^{\text{\rm nd}}_\star(J^1B)
$$
the set consisting of $R$'s intersecting the zero wall $Z = o_{T^*B} \times \R$ transversely.

The following definition is standard in the 1-jet bundle
and its generic structure is well-known.

\begin{defn} The \emph{wave front} of $R \subset J^1B$ is the image of
the front projection $ R \subset J^1B \to B \times \R$.
\end{defn}

Now let $S: E\to \R$ be a smooth function. 
We denote by $d^vS: E \to E^*$ the \emph{fiber derivative} of $S$, i.e., the map whose value
 at $e \in E$  is defined by
\be\label{eq:fiber-derivative}
d^vS(e)(\xi) = \frac{d}{ds}\Big|_{s=0} S(e + s \xi)
\ee
for $\xi \in E_{\pi(e)}$. This defines \emph{nonlinear} bundle map from $E$ to $E^*$.

With respect to an affine connection of $E$, we have the splitting
$$
TE = VTE \oplus HTE. 
$$
Under the canonical identification $VT_e E \cong E_{\pi(e)}$, the value
$d^vS(e)$  of the vertical derivative $d^vS$ at $e \in E$ can be canonically identified 
with the fiber derivative $dS(e)|_{VT_e E}$.
We define
$$
\Sigma_S := (d^{v}S)^{-1}(0)
$$
and call it the \emph{fiber-critical set} of $S$. Under the transversality condition
\be\label{eq:dvS-transverse}
d^vS \pitchfork o_{E^*} \Longleftrightarrow dS|_{VTE} \pitchfork o_{(VT)^*E},
\ee
$\Sigma_S$ is a smooth manifold of dimension equal to $\dim B$.
We recall the horizontal derivative map $\iota_S$ from \eqref{eq:iotaCS-intro}
$$
\iota_{S}(e) =  \left(\pi(e), D^hS(e) \circ (d\pi_e|_{HT_eE})^{-1}, S(e)\right) 
$$
which does not depend on the connection as a linear map
\be\label{eq:horizontal-derivative}
D^h S(e) = dS(e)|_{HT_eE} : HT_eE \to \R
\ee
at any zero point of $d^vS$, i.e, at $e \in \Sigma_S$. Composing this with
the natural isomorphism $HT_e E \cong T_{\pi(e)}B$, $D^h S(e)$ can be naturally
identified with an element in $T_{\pi(e)}^*B$. We denote this map by
$$
\iota_S : \Sigma_S \subset E \to J^1B
$$
\be\label{eq:iotaS}
\iota_S(e) = \left(\pi(e), dS(e) \circ (d\pi|_{HT_eE})^{-1}, S(e)\right).
\ee
We emphasize that $\iota_S$ is still well-defined as a map without assuming
transversality condition given in  \eqref{eq:dvS-transverse}. 
Under the transversality condition,
the map $\iota_S$ becomes an immersion which is also Legendrian.

\begin{rem} \label{rem:when-e-degenerate}
We remark that the map $\iota_S$ is not necessarily an embedding but
only an immersion in general for the symplectic case of cotangent bundles even when $S$ is nondegenerate.
However by the dimensional reason, the set of $S$ whose generating Legendrian submanifolds
are embedded can be shown to be open and residual in the set of GFQI. (See \cite{oh-park} for
the proof of this last statement.)
\end{rem}

\begin{defn} \label{defn:S-generates-R}
Let $S: E \to \R$ be a smooth function.
We say $S$ \emph{generates} a Legendrian submanifold
$R \subset J^1B$, if $R$ is the image of a Legendrian
immersion $\iota_S$ defined above, and call $S$ a \emph{generating function} of $R$. If, in addition, 
$S$ is quadratic at infinity, we call it a \emph{generating function quadratic at infinity ($\GFQI$)}.
\end{defn}

The following existence theorem for $\GFQI$ is proved by Laudenbach-Sikorav \cite{laud-sikorav,sikorav:cmh} 
 for the symplectic case (see also \cite{theret}), and
generalized to the contact case by Chaperon \cite{chaperon} and Chekanov \cite{chekanov:generating}.

\begin{prop}[Existence]\label{prop:existence} Let $R \in \frak{Leg}_0(J^1B)$. Then there exists a generating function
$S : E \to \R$ that generates $R$ in the sense that
$$
R = \iota_S( \Sigma_S )
$$
as explained above.
Furthermore, since $R$ is compact, we may choose $S$ so that it is quadratic at infinity.
\end{prop}

It is also proved by Viterbo and Théret in \cite{viterbo:generating}, \cite{theret} that the following
uniqueness result holds:

\begin{prop}[Uniqueness]\label{prop:uniqueness} Let $B$ and $R$ be
as above in Proposition \ref{prop:existence}. Suppose $S:E \to \R$ and $S':E' \to \R$
are GFQI of $R$. Then $S'$ can be obtained from $S$ by applying a sequence of
the following operations:
\begin{enumerate}
\item \text{\rm (Fiber-preserving diffeomorphism)} $S' = S \circ \Phi$ for some fiber-preserving
diffeomorphism $\Phi:E' \to E$,
\item \text{\rm (Stabilizations)} $S' = S + Q: E' = E \oplus F \to \R$, where $F \to B$ is a vector bundle and
$Q: F \to \R$ is a fiberwise nondegenerate quadratic form.
\end{enumerate}
\end{prop}

\begin{defn}\label{defn:Leg-equivalence} We say two $\GFQI$ $S,\, S'$ are \emph{equivalent} if
$S'$ is obtained from $S$ by applying the operations of (1) and (2) a finite number of times.
\end{defn}

\begin{rem}\label{rem:addition-by-constant}
We would like to recall that there is another operation, addition by a constant,
$S' = S +c$ for some $c \in \R$. In the symplectic-Lagrangian context, this operation
should be also contained in the definition of equivalence because it does not change
the generated Lagrangian submanifold in that context.
However it changes the generated Legendrian submanifold in the contact-Legendrian
context. It will be important to keep track of this operation of adding a constant
in the study of spectral invariant of Legendrian submanifolds as that of
\cite{oh-yso:spectral} and \cite{bhupal,sandon:homology}.
\end{rem}

\section{Broken trajectory approximation space $E$ of $\CL^{\text{\rm Carnot}}(J^1B,H;o_{J^1B})$}
\label{sec:approximation}

Assume that $B$ is a closed manifold and consider its one-jet bundle $M = J^1B$
equipped with the canonical contact form $\lambda = dz - pdq$.
For the simplicity of notation, we often simply write the base projection of $J^1B$ and of $T^*B$ by
$$
\pi_B(y)=: q(y) \quad \text{\rm and }\,  \pi_B(x) = q(x) 
$$
respectively.
We fix a Riemannian metric
$g$ on $B$ and its associated Levi-Civita connection. All norms on $TB$ and $T^*B$
will be computed in terms of this fixed metric $g$, and the exponential map on $B$
is in terms of the Levi-Civita connection.

For given integer $N \geq 1$, we consider partitions of the interval $[0,1]$  into $N$ subintervals
$$
0 = t_0 < t_1 < \cdots < t_N =1, \quad t_i = \frac{i}{N}.
$$
Let $\gamma \in \CL^{\rm Carnot}(J^1B,H;o_{J^1B})$. 
By definition, we have
\be\label{eq:Carnot-path2}
0 = \lambda(\dot \gamma) + H_t(\gamma(t)) = \big(\dot z - p(\gamma)(\dot q) + H_t\big)(\gamma(t)).
\ee
When $\gamma$ is a Hamiltonian trajectory, this is reduced to \eqref{eq:onshell-formula}
after integration over $t \in [0,1]$.
\begin{rem}
Consider a Hamiltonian trajectory $\gamma$ from $y_0$ to $y_1$. Then we can express $\gamma$ in two ways:
$$
\gamma(t) = \psi^t_H(y_0) = \phi^t_H(y_1).
$$
Unlike before, we use the symbol $\psi$ in this section because we frequently use Hamiltonian trajectories 
starting from some given points denoted by $y^+$.
\end{rem}

For each $i=1,\ldots,N$, we define a Hamiltonian isotopy over the interval $[t_{i-1},t_i]$ by
\be\label{eq:psiHit}
\psi_{H;i}^t:=\psi_H^t(\psi_H^{t_{i-1}})^{-1}
\ee
in accordance with the convention explained in the remark above. We mention $\psi_{H;i}^{t_{i-1}} = \id$
for all $i$. Therefore  the reparameterized path
\be\label{eq:reparameterized-psiHi}
t \in [0,1] \mapsto \psi_{H;i}^{t_{i-1} + \frac{t}{N}} 
\ee
can be made uniformly as close to $\id$ as possible on $[0,1]$ for all $i = 1, \cdots, N$
by letting $N \to \infty$. Starting from $y_0: = (q_0,0) = \gamma_1(0) \in o_{J^1B}$, we will
concatenate $\{\gamma_i\}$ into a broken trajectory with jumps at $t_i$'s.
For the simplicity of notation in our later calculations, we set 
\be\label{eq:psii}
\psi_i: =\psi_{H;i}^{t_i}.
\ee
Once such $N$ is chosen, we fix it for the rest of the construction, while allowing further enlargement of $N$ later if necessary.

As in \cite{laud-sikorav}, we now consider the vector bundle 
\be\label{eq:E}
E = (TB)^{\oplus (N-1)} \oplus (T^*B)^{\oplus (N-1)} \to B.
\ee
We denote a point $e \in E$ by
$$
e = (q_0, X, P),
$$
where $q_0 \in B$,
\be\label{eq:X}
X \in (T_{q_0}B)^{\oplus(N-1)}, \quad X = (X_1, \ldots, X_{N-1}),
\ee
and
\be\label{eq:P}
P \in (T^*_{q_0} B)^{\oplus(N-1)},  \quad P = (P_1, \ldots, P_{N-1}).
\ee
We put the initial condition of  each \emph{broken trajectoriy} $\gamma$ with 
$$
\gamma(0) = (q_0, 0, 0) =: y_0 \in o_{J^1B}\subset J^1B.
$$ 

 To associate a broken trajectoriy to each element $(q_0,X,P)$
 of $E$, we will use the exponential map associated with the Riemannian metric $g$. 
 For this reason, we need to ensure that each $X_k \in T_{q_0}B$ has its norm
 less than the injectivity radius at $q_0$. We define the $L_\infty$-norm of $X=(X_1,\ldots,X_{N-1})$ by
$$
\|X\|_\infty := \max\{\|X_k\|_g\mid 1\leq k \leq N-1\},
$$
where $\|X_k\|_g$ denotes the norm induced by the metric $g$. Let 
 $\iota_g$ be the injectivity radius of $(B,g)$. Since $B$ is compact, $\iota_g > 0$.
 We fix such a constant $\epsilon_0 > 0$ such that $\epsilon_0 < \iota_g$, and define
 a neighborhood 
\be\label{eq:epsilon0}
U := \{(q_0,X,P)\in E \mid \|X\|_\infty <\epsilon_0\}
\ee
of the zero section of $E$. We mention that \emph{there is no constraint on our choice co-vectors $P$.}
We will further specify the constant $\epsilon_0$ 
later so that Proposition \ref{prop:local-coord} below holds.

\emph{
For the clarity of our later exposition  we will adopt the different expression 
$$
e = (q_0, \widetilde{X}, \widetilde{P}) \in U
$$ 
for the element $e$ from  $U$, while the letters without `tilde' will be reserved for 
those from the whole $E$, not restricted to $U$.}

\medskip

\noindent{\bf Step 1:}
Recall that $[t_0,t_1]=[0,1/N]$. We define the reparameterization
\be\label{eq:gamma1}
\gamma_1(t):= (q_1(t),p_1(t),z_1(t)) : = \psi_H^{t_0 + \frac{t}{N}}(q_0,0,0) \in J^1B,\qquad t\in[0,1], \quad t_0 = 0
\ee
of the segment of the Hamilonian trajectory $t \mapsto \psi_H^t(q_0,0,0)$  restricted to 
the first subinterval $[t_0,t_1]$.
We denote the initial condition by $y_0:=(q_0,0,0)\in o_{J^1B}$, and express the endpoint of the trajectory 
(recalling $t_0 = 0$) as
\beastar
y_0^+ : = \psi_H^{t_0}(y_0) = \gamma_1(0) = (q_0,0,0)  & =:  &(q_0^+,p_0^+, z_0^+) \\
y_1^- := \psi_H^{t_1}(y_0) = \gamma_1(1) & =: &  (q_1^-, p_1^-,z_1^-)
\eeastar
componentwise.

Let $\widetilde{x}_1$ and $\widetilde{p}_1$  be the parallel transports of 
$\widetilde{X}_1 \in T_{q_0}B$ and $\widetilde{P}_1 \in T_{q_0}^*B$, 
respectively, along the base path $q_1(t)$ with respect to the Levi-Civita connection.
We will extensively use the parallel transports along the short geodesics in the rest of the paper.
For this purpose, we assume that $d(q_0,q_1) < \iota_g$ where $\iota_g$ is 
 the injectivity radius of the metric $g$.
We denote by
$\exp_q: T_qB \to B$ the exponential map issued at $q$, and denote by
$$
\operatorname{Exp}: TB \to B
$$
the \emph{collective exponential map} 
\be\label{eq:Exp}
\operatorname{Exp}(q,v): = \exp_q(v)
\ee
and its \emph{fiberwise} inverse
\be\label{eq:exp-inverse}
\text{\rm E}(q,q'): = \exp_q^{-1}(q')
\ee
for any pair $(q,q')$ with $d(q,q') < \iota_g$.

We will also utilize the following standard exponential estimates.
(See \cite{karcher}, for example.)
\begin{lem}\label{lem:exp-estimates}
Assume $B$ is compact or the metric $g$ is of
 bounded geometry, i.e., bounded curvature and injectivity radius $\iota_g > 0$. Denote by 
 $D_1\Exp(x,v), \, d_2\Exp(x,v)$ the first and second partial derivatives of $\Exp$ at $(x,v) \in TB$.
 Then there exists $C > 0$ depending only on $g$ and $\delta > 0$  below such that 
  \bea
 |D_1\Exp(x, v)(u)|  &\leq&  C|v|\, |u|,  \label{eq:D1exp}\\
 |d_2\Exp(x,v)(u)| & \leq & C|v|\, |u|, \label{eq:d2exp}
 \eea
and similarly for the map $\text{\rm E}$, provided $|v| \leq  \iota_g - \delta$ with $0 < \delta < \frac12 \iota_g$.
\end{lem}
\begin{rem}
 Here, we use the capital $D$ to emphasize that it refers to the differential with respect to the base point, 
 which depends on the connection, while  we use $d$ when the corresponding variables are fiber coordinates
 with base point fixed following the standard practice, e.g., in  \cite{karcher}.  In this regard, we also use the standard
 canonical identification $T_0(T_q B) \cong T_q B$ (see \cite[p.452-453]{spivakI} for example),
 when we denote an element $T_0(T_q B)$  by
 the same letter $u$ for the second inequality  \eqref{eq:D1exp} involving $D_1$ as for the first equality
 $u \in T_q B$. 
  \end{rem}

Now we define a path $c_1 =:c_{(\widetilde X_1,\widetilde P_1)}:[0,1]\to T^*B$  by
\be\label{eq:c1}
c_1(t) = \left(\exp_{q_1^-}(t\widetilde{x}_1),((d\exp_{q_1^-}(t\widetilde{x}_1)^*)^{-1}(\widetilde{p}_1)\right)
\ee
and consider its horizontal lift $\mu_1(t):[0,1]\to J^1B$ with respect to $\xi$, i.e., $\dot{\mu}_1(t) \in \xi_{\mu_1(t)}$ for each $t \in [0,1]$. 
This lift is given by
\be\label{eq:mu1}
\mu_1(t) = (c_1(t),z_1(t)), \quad
z_1(t) = z_1^- + \int_{c_1|_{[0,t]} }p\, dq.
\ee
In particular, we have the formulae for its end points,
$$
\mu_1(0)=(q_1^-,\widetilde{p}_1,z_1^-),\qquad \mu_1(1)=(q_1^+,p_1^+,z_1^+),
$$
where
\be\label{eq:1+}
\begin{cases} q_1^+ = \exp_{q_1^-}(\widetilde{x}_1), \\
p_1^+ = ((d\exp_{q_1^-}(\widetilde{x}_1))^*)^{-1}(\widetilde{p}_1),\\
z_1^+ = z_1^- + \int_{c_1} p\, dq.
\end{cases}
\ee
We `concatenate' them and define this concatenation, denoted by  $\gamma_1 \# \mu_1$, 
to be the \emph{piecewise continuous} curve  
\be\label{eq:concatenation}
\gamma_1 \# \mu_1(t) = 
\begin{cases} \gamma_1(2t) \quad & 0 \leq t < \frac12\\
\mu_1(2t-1) \quad & \frac12 \leq t \leq 1.
\end{cases}
\ee
This concatenation may have a jump-discontinuity at $t= \frac12$.
We denote this concatenated curve by
\be\label{eq:tildegamma1}
\widetilde \gamma_{(\widetilde X_1,\widetilde P_1)}:= \gamma_1 \# \mu_1
\ee
to emphasize its dependence on $(\widetilde X_1,\widetilde P_1)$.
We emphasize that 
$$
\gamma_{(\widetilde X_1,\widetilde P_1)}(0) = (\psi_H^{t_0})^{-1}(\gamma_1(0)) = (q_0,0,0) \in o_{J^1B}
$$
(recalling $t_0 = 0$).
We denote by $d_1$ the curve defined by 
\be\label{eq:d1}
d_1(t): = \pi_{T^*B}(\gamma_1(t)). 
\ee
Then 
\be\label{eq:tildeX1Y1in0}
\Pi_{q(d_1)}^{-1}(\widetilde x_1,\widetilde p_1) = (\widetilde X_1,\widetilde P_1)\in T_{q_0}B \oplus T_{q_0}^*B
\ee
by the definition of $(\widetilde x_1,\widetilde p_1)$.
This is the end of the first step.

In summary, we collect the 4 key points in Step 1 here for later reference:
\be\label{eq:3points-step1}
\begin{cases}
\gamma_1(0) = (q_0^+, p_0^+,z_0^+) , \quad &\gamma_1(1) = (q_1^-, p_1^-,z_1^-) \\
\mu_1(0) = (q_1^-, \widetilde p_1, z_1^-),\quad &\mu_1(1) = (q_1^+,p_1^+,z_1^+).
\end{cases}
\ee
\begin{rem}\label{rem:broken}
   Note that the \emph{initial point} of $\mu_1$ is given by
  $\mu_1(0)=(q_1^-,\widetilde{p}_1,z_1^-)$,  and
   the \emph{end point} of the Hamiltonian trajectory $\gamma_1$ is given by 
 $$
    \gamma_1(1)=(q_1^-,p_1^-,z_1^-):
$$    
  The two points share the same $q$- and $z$-coordinates but generally differ in $p$-coordinate. As a result, the two curves $\gamma_1$ and $\mu_1$ are not continuously concatenated in general as a single trajectory in $J^1B$. However, we will see later similarly as in Laudenbach-Sikorav's case 
\cite{laud-sikorav} that these two curves will be
continuously  joined together forming a single Hamiltonian trajectory for each vertical critical element
of the broken-trajectory approximation  we define below.
\end{rem}

Then the concatenation $d_1 \# c_1:[0,1] \to T^*B$ has the formula
$$
d_1\# c_1(t) = \pi_{T^*B} \left(\widetilde \gamma_{(\widetilde X_1, \widetilde P_1)}(t)\right)
$$
by the definitions \eqref{eq:c1}, \eqref{eq:d1} and \eqref{eq:tildegamma1}.

\medskip

\noindent{\bf Step 2:}
As in Step 1, we now consider the second segment $[t_1,t_2] = [1/N,2/N]$ and define the 
reparameterization 
\be\label{eq:gamma2}
\gamma_2(t) := (q_2(t), p_2(t), z_2(t)) : = \psi_{H;2}^{t_1 + \frac{t}{N}}(y_1^+) \in J^1B, \qquad t\in[0,1].
\ee
We denote its initial point by
$$
y_1^+ := (q_1^+, p_1^+, z_1^+) = \gamma_2(0).
$$
Recall that $\psi_{H;2}^t=\psi_H^t\circ(\psi_H^{t_1})^{-1}$. We write the endpoint of the trajectory as
$$
y_2^- := \gamma_2(1) = \psi_H^{t_2} (\psi_H^{t_1})^{-1} (y_1^+) = :  (q_2^-, p_2^-, z_2^-)
$$
componentwise.
We consider the parallel transport from $t=0$ to $t=1$
along the curve  $q((d_1 \# c_1) \# d_2)$.  Then we define
$\widetilde{x}_2$ and $\widetilde{p}_2$ by the equation
\be\label{eq:tildegamma20}
(\widetilde x_2, \widetilde p_2) =  \Pi_{q((d_1 \# c_1) \# d_2)}(\widetilde X_2, \widetilde P_2),
\ee
with the parallel transports of 
$\widetilde{X}_2 \in T_{q_0}B$ and $\widetilde{P}_2 \in T_{q_0}^*B$ to $q_2^-$ 
respectively, along the concatenated base path consisting of
$q_1(t)$,  the projection of $c_1(t)$ and $q_2(t)$ in $B$.
As before, we define the path $c_2$ in $T^*B$ using the exponential map at $q_2^-$, 
and its horizontal lift $\mu_2$ in $J^1B$ given by
\beastar
\mu_2(t) := (c_2(t),z_2(t)), \quad
z_2(t) = z_2^- + \int_{c_2|_{[0,t]} }p\, dq, \\
c_2(t) = \left(\exp_{q_2^-}(t\widetilde{x}_2),((d\exp_{q_2^-}(t\widetilde{x}_2))^*)^{-1}(\widetilde{p}_2)\right) =: c_{(\tildeX_2,\tildeP_2)}(t).
\eeastar
In particular, we have
$$
\mu_2(0)=(q_2^-,\widetilde{p}_2,z_2^-),\qquad \mu_2(1)=(q_2^+,p_2^+,z_2^+),
$$
where
$$
\begin{cases} q_2^+ = \exp_{q_2^-}(\widetilde{x}_2), \\
p_2^+ = ((d\exp_{q_2^-}(\widetilde{x}_2))^*)^{-1}(\widetilde{p}_2),\\
z_2^+ = z_2^- + \int_{c_2} p\, dq.
\end{cases}
$$
Then we define
\be\label{eq:tildegamma2}
\widetilde \gamma_2(t) = \widetilde \gamma_{(\widetilde X_2,\widetilde P_2)}(t):= \gamma_2 \# \mu_2(t).
\ee

In summary, we collect the 4 key points in $J^1B$ in Step 2 here:
\be\label{eq:3points-step2}
\begin{cases}
\gamma_2(0) = \mu_1(1) = (q_1^+,p_1^+,z_1^+), \quad &\gamma_2(1)   =   (q_2^-, p_2^-,z_2^-) \\
\mu_2(0) = (q_2^-, \widetilde p_2, z_2^-),\quad & \mu_2(1) = (q_2^+,p_2^+,z_2^+).
\end{cases}
\ee

\medskip

\noindent{\bf Step 3 (Induction):}

We now construct a sequence of contact Hamiltonian trajectories and horizontal curves 
over the interval $[t_{k-1},t_k]$, for each $k = 3, \ldots, N$, generalizing the construction of previous steps.

Suppose we have defined $\widetilde x_i$, $\widetilde p_i$ and the paths $c_i$, $z_i$ and $\mu_i = (c_i,z_i)$
for $i = 1, \cdots, k-1$. Then we will construct all of them for $i = k$.

Repeating the operation performed in Step 2 iteratively,
we denote by $d_i$ the curve defined by $d_i(t): = \pi_{T^*B}(\gamma_i(t))$ and consider the concatenation
$d_i \# c_i:[0,1] \to T^*B$. Then we consider the concatenation
\be\label{eq:kth-contcatenation}
(d_1\# c_1) \# (d_2 \# c_2) \# \cdots \#  (d_{k-1} \# c_{k-1}) \#d_k =: (d \# c)_k
\ee
and consider the parallel transport $\Pi_{q((d \# c)_k)}$ from $t=0$ to $t=1$ defined similarly as
in Step 2. \emph{We alert readers that our notation $(d \# c)_k$ should not be confused with $d_k \# c_k$.}

We define the Hamiltonian trajectory $\gamma_k:[0,1] \to J^1B$ by
\be\label{eq:gammak}
\gamma_k(t): =   \psi_{H;k}^{t_{k-1} + \frac{t}{N}} (y_{k-1}^+), 
\quad  t\in[0,1],
\ee
and denote its endpoint by
\be\label{eq:k-}
y_k^- := \gamma_k(1).
\ee
We also write
$\gamma_k(t) = (q_k(t), p_k(t), z_k(t))$ and 
$y_k^-= (q_k^-,p_k^-,z_k^-)$ respectively componentwise.

\begin{defn}[Coordinates $\widetilde x_k$ and $\widetilde p_k$]
We define $\widetilde{x}_k$ and $\widetilde{p}_k$ 
are the parallel transports of $\widetilde{X}_k$ and $\widetilde{P}_k$ 
\be\label{eq:tildegamma-k0}
(\widetilde x_k,\widetilde p_k) =\Pi_{q((d \# c)_k)}(\widetilde X_k,\widetilde P_k)
\ee
along the base path $q((d\#c)_k)$
of the concatenation $(d\#c)_k$.
\end{defn}

Then we define $c_k:[0,1]\to T^*B$ to be the path defined by
\be\label{eq:ck}
c_k(t) = \left(\exp_{q_k^-}(t\widetilde{x}_k),((d\exp_{q_k^-}(t\widetilde{x}_k))^*)^{-1}(\widetilde{p}_k)\right)
\ee
and compute its horizontal lift $\mu_k:[0,1]\to J^1B$ by
\be\label{eq:muk}
\mu_k(t) = (c_k(t),z_k(t)), \quad
z_k(t) = z_k^- + \int_{c_k|_{[0,t]} }p\, dq.
\ee
We set the endpoint $y_k^+:=\mu_k(1)=(q_k^+,p_k^+,z_k^+)$, where explicitly:
\be\label{eq:k+}
\begin{cases}
q_k^+ = \exp _{q_k^-} (\widetilde{x}_k), \\
p_k^+ = ((d \exp _{q_k^-} (\widetilde{x}_k))^*)^{-1}(\widetilde{p}_k), \\
z_k^+ = z_k^- + \int_{c_k} p\, dq.
\end{cases}
\ee
We define 
\be\label{eq:tildegammak}
\widetilde \gamma_k = \gamma_k \# \mu_k = :  \widetilde \gamma_{(\widetilde X_k,\widetilde P_k)}.
\ee

Putting together all the Hamiltonian trajectories and horizontal lifts constructed above, 
we obtain a collection of curves
$$
\{\gamma_1,\mu_1,\ldots,\gamma_{N-1},\mu_{N-1},\gamma_N\}
$$
that form a \emph{broken trajectory} in $J^1B$, depending on the data
 $$
 (\widetilde{X}_j,\widetilde{P}_j), \quad j =1,\ldots,N-1. 
$$ 
This construction is illustrated schematically in Figure \ref{fig:broken-traj}.

\begin{figure}
    \centering
    \includegraphics[width=1.0\linewidth]{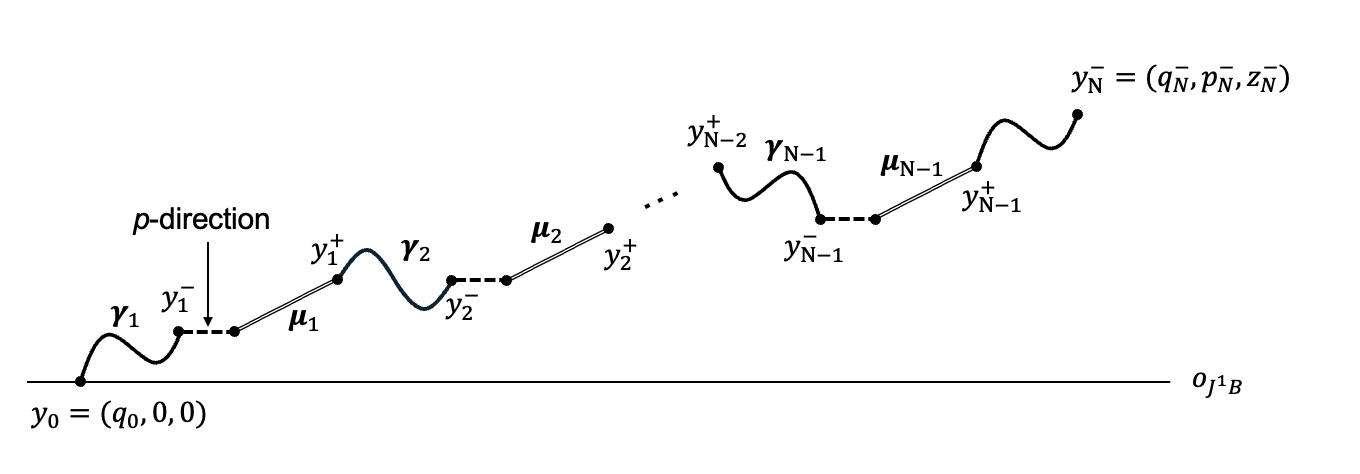}
    \caption{A broken trajectory in $J^1B$. Each arc $\gamma_k$ corresponds to a piece of contact Hamiltonian flow from $y_{k-1}^+$ to $y_k^-$, and each $\mu_k$ is the horizontal lift of a geodesic flow, ending at $y_k^+$. The dotted lines indicate jumps in the $p$-coordinate between $\gamma_k$ and $\mu_k$. The trajectory starts at $y_0 = (q_0,0,0)$ and ends at $y_N^-=(q_N^-,p_N^-,z_N^-)$.}
    \label{fig:broken-traj}
\end{figure}

This broken trajectory can be represented as a piecewise smooth curve
\bea\label{eq:broken-curve}
\tildegamma_{(q_0,\tildeX,\tildeP)} & :=&  (\tildegamma_1 \#\tildegamma_2 \# \cdots \# 
\tildegamma_{N-1}) \# \gamma_N \nonumber \\
& = & \left(\gamma_1 \# \mu_1\right) \# \left(\gamma_2 \# \mu_2\right) \# \cdots \# 
\left(\gamma_{N-1} \# \mu_{N-1}\right) \# \gamma_N
\eea
whose derivative can make a jump at each operation $\#$ inside the parenthesis 
depending on the datum $(\tildeX_j,\tildeP_j)$ from $E$ for $i = 1, \cdots N-1$.

In summary, we collect the 4 key points in $J^1B$ in Step $N-1$, the last iteration, here:
\be\label{eq:3points-step3}
\begin{cases}
\gamma_{N-1}(0) = \mu_{N-2}(1) = (q_{N-2}^+,p_{N-2}^+,z_{N-2}^+), \quad &
\gamma_{N-1}(1)   =   (q_{N-1}^-, p_{N-1}^-,z_{N-1}^-) \\
\mu_{N-1}(0) = (q_{N-1}^-, \widetilde p_{N-1}, z_{N-1}^-),\quad & \mu_{N-1}(1) 
= (q_{N-1}^+,p_{N-1}^+,z_{N-1}^+).
\end{cases}
\ee

\medskip

\noindent{\bf Step 4: Coordinate change:} 
We now mention
 that by construction all $\mu_k$, $\gamma_k$, $q_k^{\pm}$, $p_k^{\pm}$, etc., are
functions of $(q_0,\tildeX,\tildeP)$ for $k = 0, \ldots, N-1$. 
The above construction iterates up to Step $N\!-\!1$ and then ends with  $\gamma_N$,
the $N$-th $\gamma$-segment: \emph{there exists no $N$-th $\mu$-segment $\mu_N$}. 

We will view all $\gamma$-segments as functions of $(q_0,\tildeX, \tildeP)$. For example, we have
 $$
 q_{k+1}^- = q(\gamma_{k+1}(1)) = q(\psi_{k+1}(y_k^+))
 $$
where the map 
 $$
 \psi_{k+1} = \psi_{H;k+1}^{t_{k+1}} : = \psi_H^{t_{k+1}}\circ(\psi_H^{t_k})^{-1}
 $$
is  given in \eqref{eq:psii} and \eqref{eq:psiHit}. The assignment $y_k^+ \mapsto q_{k+1}^-$
 is a smooth function of  $y_k^+ = (q_k^+,p_k^+,z_k^+)$ which is determined by
 the rescaled Hamiltonian trajectory of $H$ along the segment $\gamma_{k+1}$. 
We will next show that the map
\be\label{eq:qXP-to-q1pqN}
(q_0,\widetilde{X},\widetilde{P}) \mapsto \left(q_1^-, \widetilde p_1, \cdots, q_{N-1}^-, \widetilde p_{N-1}, q_N^-\right).
\ee
defines a coordinate change by proving invertibility of its Jacobian.
 
\section{Jacobian of coordinate change}
 
We need to  study the derivatives of the map \eqref{eq:qXP-to-q1pqN}
with respect to  $(q_0,\tildeX,\tildeP)$ or equivalently  with respect to 
its parallel transport
$$
(q_0,\widetilde x,\widetilde p) =  \left(q_0,(\widetilde x_1,\widetilde p_1), (\widetilde x_2,\widetilde p_2),
\cdots,  (\widetilde x_{N-1},\widetilde p_{N-1})\right).
$$
To avoid confusion and for the clarity of exposition, we  henceforth denote
the derivatives thereof by 
$$
\delta \mu_k, \,\, 
\delta \gamma_k, \, \, \delta q_k^{\pm}, \, \, 
\delta p_k^{\pm}
$$
etc., respectively. Then  we can write its variation as
\be\label{eq:deltaqk-}
\delta q^-_{k+1} = D_1(q\circ \psi_{k+1})\delta q^+_k + d_2(q\circ \psi_{k+1})\delta p^+_k
+ d_3(q\circ \psi_{k+1})\delta z^+_k 
\ee
Here, we use the notation $D_1,d_2,d_3$ to denote partial derivatives with respect to the first, 
the second, and the third variables of a given map, respectively. 
(See  the remarks right after Lemma \ref{lem:exp-estimates} for relevant remarks.)

 For the simplicity of exposition, we consider the 
 linear map from $T_{q_k^-}B$ to $T_{q_{k+1}^-}B$ defined by
 \be\label{eq:Dk}
D^k := D_1(q\circ \psi_{k+1})d\exp_{q^-_k}(\widetilde{x}_k) 
+ d_3(q\circ \psi_{k+1})\langle \widetilde{p}_k, \cdot \rangle.
\ee
The following is a key element of the proof of Proposition \ref{prop:local-coord}.

\begin{lem}\label{lem:Dk} The linear map $D^k: T_{q_k^-}B \to T_{q_{k+1}^-}B$
is invertible if we choose $N$ sufficiently large and $\epsilon_0$ sufficiently small.
\end{lem}
\begin{proof} 
It follows from the compact support assumption on $H$ that
for a sufficiently large $N$, the $C^1$ distance between the time-one map $\psi_{k+1}$ and the identity map is of order $1/N$ for each $k$: This can be easily seen by considering $\psi_{k+1}$ as the time-one map of the reparameterization
\be\label{eq:rescaled-isotopy}
t \in [0,1] \mapsto \psi_{H;k+1}^{t_k + \frac{t}{N}}
\ee
of the isotopy 
$$
t \in [t_k, t_{k+1}] \mapsto \psi_{H;k+1}^t=\psi_H^t \circ (\psi_H^{t_k})^{-1}.
$$
Note that  the rescaled isotopy \eqref{eq:rescaled-isotopy}
 is generated by the Hamiltonian 
 \be\label{eq:Hk+1}
 \Dev_\lambda \left(t \mapsto \psi_{H;k+1}^{t_k + \frac{t}{N}}\right) = \frac1N H_{(k+t)/N}.
 \ee
 (See Definition \ref{defn:developing-map} for the definition of $\Dev_\lambda$.)
Then it follows from \eqref{eq:Hamilton-eq-Darboux} that 
\be\label{eq:der-qpsi}
\|D_1(q\circ \psi_{k+1}) -\Pi\| \leq \frac{C_1}{N}, \, \quad \|d_3(q\circ \psi_{k+1})\| \leq \frac{C_2}{N}
\ee
for some universal constant $C_1, \, C_2 > 0$:
Here $\Pi$ is the parallel transport along the unique short geodesic from $q_k^-$ to $q_{k+1}^-$. 
Uniqueness of the short geodesic is guaranteed if the distance between them is smaller than the injectivity radius $\iota_g$, 
which holds when $N$ is large enough.

On the other hand, it also follows from the compact support hypothesis of $H$ that 
there exists a sufficiently large $R_0> 0$ such that
$$
\supp H \subset D_{R_0}(J^1B)
$$
where we put
\be\label{eq:DRJ1B}
D_{R}(J^1B): = D_{R}(T^*B) \times [0,R]
\ee
where $D_{R}(T^*B)$ is the disc-bundle of radius $R$. 

When $\| \tildeP_k \| > R_0$, we have
\be\label{eq:d3R0}
d_3(q\circ \psi_{k+1}) = 0.
\ee
By using the fact that $d\exp_{q_k^-}(0) = \id$ and by the standard exponential estimate 
arising from Lemma \ref{lem:exp-estimates}, we obtain
$$
\|d\exp_{q^-_k}(\widetilde x_k) - \Pi \| \leq C\|\widetilde x_k\| \leq C \epsilon_0,
$$
for some $C>0$ where $\Pi$ is a similar parallel transport from $q_k^-$
to $\exp_{q_k^-}(\widetilde x_k)$. In particular,
\be\label{eq:dexpqk}
\|d\exp_{q^-_k}(\widetilde x_k) \| \geq (1 -  C \epsilon_0)
\ee
for any $\widetilde x_k$ satisfying $\|\widetilde x_k\| < \epsilon_0$.
\begin{sublem}\label{sublem:Dk-Pi} We have 
$$
\|D^k - \Pi\| \leq \frac12
$$
for all $k = 1, 2, \cdots$, provided $0 < \epsilon_0 < \frac1{4C}$ and $N > 4(2C_1 + C_2 R_0)$.
\end{sublem}
 \begin{proof} It follows from \eqref{eq:Dk} that 
$$
\|D^k - \Pi\| \leq \|D_1(q\circ \psi_{k+1}) \cdot d\exp_{q^-_k}(\widetilde x_k) - \Pi\| 
+ \|d_3(q\circ \psi_{k+1})\langle \widetilde p_k, \cdot \rangle \|.
$$
 By \eqref{eq:d3R0}, we obtain
 $$
\| d_3(q\circ \psi_{k+1})\langle \widetilde p_k, \cdot \rangle \| \leq \frac{C_2}{N} \cdot R_0.
$$
 By the inequalities \eqref{eq:der-qpsi} and \eqref{eq:dexpqk}, we can write
 $$
 D_1(q\circ \psi_{k+1})  = \Pi + B_1, \quad   d\exp_{q^-_k}(\widetilde x_k) = \id + B_2
 $$
 with $\|B_1\| \leq \frac{C_1}{N}$ and $\|B_2\| \leq C \epsilon_0$. Then we expand
$$
D_1(q\circ \psi_{k+1}) \cdot d\exp_{q^-_k}(\widetilde x_k) - \Pi = (\Pi + B_1)\cdot (\id + B_2) - \Pi
= B_1 + \Pi \cdot B_2 + B_1 \cdot B_2.
$$
Therefore we have
\beastar
\|D_1(q\circ \psi_{k+1}) \cdot d\exp_{q^-_k}(\widetilde x_k) - \Pi\|
& \leq & (\|(B_1 \| + \|\Pi \cdot B_2\| + \|B_1 \cdot B_2\|)\\
& \leq &  C \epsilon_0 + \left(\frac{C_1}{N} \cdot C \epsilon_0\right) \\
& \leq &   2C \epsilon_0 + \frac{2C_1}{N}
\eeastar
where the last inequality follows whenever $\epsilon_0< 1, \, N > 1$.
Combining the above two inequalities, we have derived
$$
\|D^k - \Pi\| \leq 2C \epsilon_0 + \frac{2C_1}{N} + \frac{C_2}{N} \cdot R_0 = 2C \epsilon_0 + \frac{2C_1 + C_2 R_0}{N}
$$
Therefore if we first choose
$0 < \epsilon_0 < \frac1{4C}$ and $N > 4(2C_1 + C_2 R_0)$, then $\|D^k - \Pi\| \leq \frac12$
which finishes the proof of the sublemma.
\end{proof}

As a corollary, we have shown that $D^k$ is invertible, if we choose
\be\label{eq:epsilon0N}
 0 < \epsilon_0 < \frac1{4C}, \quad N > 4(2C_1 + C_2 R_0).
\ee
This finishes the proof of Lemma \ref{lem:Dk}.
\end{proof}

\section{A fiberwise re-scaling transformation on $E \to B$}

In the next section, we will define a finite dimensional approximation $S : E \to \R$ of 
the action functional $\CA_H^{\text{\rm CC}}$.
Since the broken trajectories are well-defined only in the open neighborhood $U \subset E$,
we need to find a way of defining them globally on $E$, \emph{without affecting the vertical critical point set}.

We fix $0 < \epsilon_0 < \iota_g$ that was chosen in Section \ref{sec:approximation}.
Then for each $0< \delta \ll \epsilon_0$, we choose a smooth function 
$\rho_{\delta,\epsilon_0}: \R_{\geq 0} \to (0,\infty)$
of the form
\be\label{eq:defn-rhodelta}
\begin{cases}
\rho_{\delta,\epsilon_0} (r) = 1  \quad & {\rm for } \quad 0\leq r \leq \delta, \\
0 < r \rho_{\delta,\epsilon_0} (r) < \epsilon_0\quad & {\rm for } \quad 0 \leq r < \infty,\\
r \rho_{\delta,\epsilon_0}'(r) + \rho_{\delta,\epsilon_0}(r) > 0  \quad & {\rm for} \quad 0 \leq r < \infty.
\end{cases}
\ee
\begin{prop}\label{prop:rhodelta}
Such a function can be always chosen.
\end{prop}
\begin{proof}
Instead of constructing $\rho_{\delta,\epsilon_0}$ directly, we first construct
a smooth strictly increasing function
\[
f_{\delta,\epsilon_0}(r) := r\rho_{\delta,\epsilon_0}(r).
\]
The conditions in \eqref{eq:defn-rhodelta} are equivalent to requiring
\[
f_{\delta,\epsilon_0}(r) = r \quad \text{for } 0\le r\le \delta,
\qquad
0 < f_{\delta,\epsilon_0}(r) < \epsilon_0 \quad \text{for } r\ge0,
\qquad
f'_{\delta,\epsilon_0}(r) >0.
\]
To construct such a function, consider
\be\label{eq:defn-f}
f(r): = \begin{cases} 
r \quad & r \in [0,\delta],\\
A_{\epsilon_0} \arctan(r) \quad & r \in [\epsilon_0,\infty)
\end{cases}
\ee
with a choice of $A_{\epsilon_0} = \frac{2 \epsilon_0}{\pi}$.
Then for all $r\ge0$ we have
\[
0 \le f(r) < \epsilon_0.
\]
We compute
\[
f'(r) = \frac{A_{\epsilon_0}}{1+r^2} >0
\]
on $[\epsilon_0,\infty)$.
We also note
\[
f(\epsilon_0) = \frac{2\epsilon_0}{\pi}\arctan(\epsilon_0).
\]
Since $f(\delta)=\delta$, we can make $\delta < f(\epsilon_0)$ by taking $\delta$
sufficiently small.
Therefore any smooth monotone interpolation of $r$ and $A_{\epsilon_0}\arctan(r)$
on the intermediate region $[\delta,\epsilon_0]$ produces a smooth strictly
increasing function $f_{\delta,\epsilon_0}$ satisfying the above conditions.
Finally we define
\[
\rho_{\delta,\epsilon_0}(r) =
\begin{cases}
1 & r=0,\\
\dfrac{f_{\delta,\epsilon_0}(r)}{r} & r>0.
\end{cases}
\]
Then $\rho_{\delta,\epsilon_0}$ is smooth on $[0,\infty)$ and satisfies
\eqref{eq:defn-rhodelta}. This finishes the proof.
\end{proof}

We consider a family of  maps 
$$
\psi_{\delta,\epsilon_0}: TB \oplus T^*B \to D_{\epsilon_0}(TB) \times_B T^*B
$$ 
\be\label{eq:psidelta}
\psi_{\delta,\epsilon_0}(q,v,\beta) = \left(q, \rho_{\delta,\epsilon_0}(|v|) \, v, 
\left(\rho_{\delta,\epsilon_0}(|v|)\right)^{-1} \beta\right)
\ee
which is a \emph{fiberwise re-scaling} diffeomorphism and preserves the natural pairing
\be\label{eq:<,>}
\langle \cdot, \cdot \rangle: TB \oplus T^*B \to \R.
\ee
We will first prove that the Jacobian of $\psi_{\delta,\epsilon_0}$ is invertible and hence it is a
local diffeomorphism. Here $\times_B$ is the fiber product of 
$\pi_B: TB \to B$ and $\pi_B^*: T^*B \to B$.

We denote by  $(X,P)$ and $(\tildeX,\tildeP)$ an element of the domain
$TB\oplus T^*B$ and $D_{\epsilon_0}(TB) \times_B T^*B$ respectively.
Recall that we cannot control the size of $\tildeP$ where these variables
are defined. Note that as $\delta \to 0$, $\|\tildeP\|$ can be arbitrarily large,
while we will make the vectors $\widetilde{X}$ satisfy
\be\label{eq:norm-tildeX}
\|\widetilde{X}\| =  \rho_{\delta,\epsilon_0}(\|X\|) \|X\| < \epsilon_0,
\ee
\emph{We will fix a sufficiently small $\delta > 0$ and 
the associated function $\rho_{\delta,\epsilon_0}$ above so that
the inequality \eqref{eq:norm-tildeX} holds once and for all.} 

For the later purpose,
we state the following obvious derivative formula of the map $\psi_{\delta,\epsilon_0}$.

\begin{lem}\label{lem:XPtotildeXP} 
Let $q \in B$ and 
 take a local frame $\{e_i\}_{i=1}^n$ of $TB$ and its dual frame $\{f_j\}_{j=1}^n$ of $T^*B$
near a given point $q \in B$. We regard $X$, $P$ (resp. 
$\tildeX$ and $\tildeP$)
 as vectors of $\R^n$ and its dual space. Put $r: = \|X\|$. Then
we have
\be\label{eq:Jacobian-formula}
\begin{cases}
\frac{\del X \circ \psi_{\delta,\epsilon_0}}{\del X} = \frac{ \rho_{\delta,\epsilon_0}'(r)}{r}\, X_jX_k \,
|e_j\rangle \langle e_k|
+ \rho_{\delta,\epsilon_0}(r)\,  \delta_{jk}\,
|e_j\rangle \langle e_k|, \quad  \frac{\del X\circ \psi_{\delta,\epsilon_0}}{\del P} = 0, \\
\\
\frac{\del P \circ \psi_{\delta,\epsilon_0}}{\del X} = - \frac{\rho_{\delta,\epsilon_0}'(r)}{ \rho_{\delta,\epsilon_0}^2(r)r}\,
X_jX_k\, |e_j\rangle\langle f_k|, \quad \frac{\del P \circ \psi_{\delta,\epsilon_0}}{\del P}
 = \frac{1}{\rho_{\delta,\epsilon_0}(r)}\,  \delta_{jk}\, |f_j\rangle \langle f_k|
\end{cases}
\ee
where we use Dirac's bra-ket notation and Einstein's summation convention.
\end{lem}

Writing $\psi_{\delta,\epsilon_0}(X,P) =: (\tildeX,\tildeP)$, we note that the Jacobian matrix
\be\label{eq:Dpsideltae0}
(D\psi_{\delta,\epsilon_0}) = \left (\begin{matrix}  \frac{\del \tildeX}{\del X} & \frac{\del \tildeX}{\del P} \\
\frac{\del \tildeP}{\del X} & \frac{\del \tildeP}{\del P} \end{matrix}
\right) = \left (\begin{matrix}  \frac{\del \tildeX}{\del X} &0 \\
\frac{\del \tildeP}{\del X} & \frac{\del \tildeP}{\del P} \end{matrix}
\right)
\ee
is a triangular matrix.

With this preparation, we are ready to prove
the invertibility of the map $\psi_{\delta,\epsilon_0}$.
\begin{prop}\label{prop:psidelta} 
The fiberwise map 
$$
\psi_{\delta,\epsilon_0}: TB \oplus T^*B \to D_{\epsilon_0}(TB) \times_B T^* B
$$
is a diffeomorphism onto its image
provided $\epsilon_0$ is sufficiently small.
\end{prop}
\begin{proof} 
In \eqref{eq:Jacobian-formula}, 
we have computed the Jacobian of $\psi_{\delta,\epsilon_0}$. It follows that
the matrix 
$$
\frac{\del P \circ \psi_{\delta,\epsilon_0}}{\del P} = \frac{1}{\rho_{\delta,\epsilon_0}(r)}\,  \delta_{jk}|f_j\rangle \langle f_k|
= \frac{1}{\rho_{\delta,\epsilon_0}(r)}\, \id
$$
is invertible.  We also obtain
\beastar
\frac{\del X \circ \psi_{\delta,\epsilon_0}}{\del X} & = & \frac{ \rho_{\delta,\epsilon_0}'(r)}{r}\, X_jX_k \,
|e_j\rangle \langle e_k|
+ \rho_{\delta,\epsilon_0}(r)\,  \delta_{jk}\, |e_j\rangle \langle e_k|\\
& = & \frac{ \rho_{\delta,\epsilon_0}'(r)}{r}\, |X\rangle \langle X|
+ \rho_{\delta,\epsilon_0}(r)\,  \id
\eeastar
therefrom. We will prove this matrix is also invertible. For this purpose, we will show that its kernel is zero.
To prove this, let $v$ satisfy $\frac{\del X \circ \psi_{\delta,\epsilon_0}}{\del X}(v) = 0$, i.e.,
\be\label{eq:0=}
0 = \frac{ \rho_{\delta,\epsilon_0}'(r)}{r}\, |X\rangle \langle X|v \rangle
+ \rho_{\delta,\epsilon_0}(r) |v \rangle. \nonumber \\
\ee
Therefore we obtain
$$
v = - \frac{ \rho_{\delta,\epsilon_0}'(r)}{r\rho_{\delta,\epsilon_0}(r)}\,\langle X|v \rangle X.
$$
In particular $v$ is parallel to $X$.
Therefore by taking the inner product of the last equation with  $X$, we obtain
$$
\langle X|v \rangle = - \frac{ \rho_{\delta,\epsilon_0}'(r)}{r\rho_{\delta,\epsilon_0}(r)}\,\langle X|v \rangle \|X\|^2.
$$
If $v \neq 0$, then we cancel $\langle X|v \rangle$ away and obtain
$$
0 = \frac{ \rho_{\delta,\epsilon_0}'(r)}{r}\, r^2
+  \rho_{\delta,\epsilon_0}(r) =
r\rho_{\delta,\epsilon_0}'(r)+\rho_{\delta,\epsilon_0}(r) .
$$

This contradicts to the inequality $r \rho_{\delta,\epsilon_0}'(r) + \rho_{\delta,\epsilon_0}(r) > 0$
from \eqref{eq:defn-rhodelta}.
This proves that the matrix $\frac{\del X \circ \psi_{\delta,\epsilon_0}}{\del X}$ is invertible.
Combining the above with the triangular form of the Jacobian \eqref{eq:Dpsideltae0}, 
we have shown that the Jacobian of $\psi_{\delta,\epsilon_0}$ 
is invertible and so it is a local diffeomorphism.

It remains to prove that the map $\psi_{\delta,\epsilon_0}$ is one to one.
Suppose $\psi_{\delta,\epsilon_0}(q, X,P) = \psi_{\delta,\epsilon_0}(q',X',P')$. Since $\pi_B \circ \psi_{\delta,\epsilon_0} = \id$, we
have $q' = q$.
Then we have
\be\label{eq:standing-hypo}
\rho_{\delta,\epsilon_0}(\|X\|)X = \rho_{\delta,\epsilon_0}(\|X'\|)X', \quad \rho_{\delta,\epsilon_0}(\|X\|)^{-1} P = \rho_{\delta,\epsilon_0}(\|X'\|)^{-1} P'
\ee
so that $X$ and $X'$ are parallel, more specifically, we have
$$
X' = \frac{\rho_{\delta,\epsilon_0}(\|X\|)}{\rho_{\delta,\epsilon_0}(\|X'\|)} X =:  c X
$$
for some constant $c$.
We have now only to prove $c = 1$. By taking the norm of the above, we derive
$$
\frac{r'}{r} =\frac{\rho_{\delta,\epsilon_0}(r)}{\rho_{\delta,\epsilon_0}(r')}, \quad r = \|X\|, \, r'= \|X'\|.
$$
Since $r\rho_{\delta,\epsilon_0}(r)$ 
 is monotonically increasing, this equality implies $r = r'$, i.e., $c = 1$
which in turn implies that $X = X'$ and  $\rho_{\delta,\epsilon_0}(\|X\|) = \rho_{\delta,\epsilon_0}(\|X'\|)$. Finally 
this together with the second equation of the
standing hypothesis \eqref{eq:standing-hypo} also implies $P = P'$. 

Therefore combining them all, $\psi_{\delta,\epsilon_0}$ is one to one and so concludes the proof
of Proposition \ref{prop:psidelta}.
\end{proof}

Now we lift $\psi_{\delta,\epsilon_0}$ to  the map $\Psi_{\delta,\epsilon_0}$ defined on the product 
$$
E= E_N = (TB \oplus T^*B)^{\oplus(N-1)} 
$$
of the form 
\bea\label{eq:Psidelta}
&{}& \Psi_\delta(q, (X_1, P_1),(X_2,P_2),\cdots,( X_{N-1},P_{N-1})) \nonumber \\
& = & \left(q, \psi_{\delta,\epsilon_0}(X_1,P_1), \cdots, \psi_{\delta,\epsilon_0}(X_{N-1},P_{N-1})\right)
\eea
for each $0 < \delta < \epsilon_0$ such that $\Psi_\delta(e) = e$ provided 
$\|X\| = \max_{k} \|X_k\| < \delta$. Since $\epsilon_0$ will be fixed from now on, we
omit it and just write
$$
\Psi_{\delta,\epsilon_0}=: \Psi_\delta
$$
from the notation of $\Psi_{\delta,\epsilon_0}$. We also write
$$
(q, (X_1, P_1),(X_2,P_2),\cdots,( X_{N-1},P_{N-1})) = :(q,X,P)
$$
by a slight abuse of notation.

Each $\Psi_\delta$ is a diffeomorphism onto its image defined by
$$
\Psi_\delta(X,P) = : (\tildeX,\tildeP)
$$
where we put
\be\label{eq:tilderohXP}
\widetilde{X}_k := \rho_{\delta,\epsilon_0} (\|X_k\|)\cdot X_k, \quad
\widetilde{P}_k := (\rho_{\delta,\epsilon_0} (\|X_k\|))^{-1} P_k
\ee
for each $k$.
Recall that we cannot control the size of $\tildeP$ where these variables
are defined. Note that as $\delta \to 0$, $\|\tildeP_k\|$ can be arbitrarily large,
while we will make the vectors $\widetilde{X}_k$ satisfy
\be\label{eq:norm-tildeXk}
\|\widetilde{X}_k\| = \rho_{\delta,\epsilon_0}(\|X_k\|)\cdot\|X_k\| < \delta,
\ee
so that $(q_0,\tildeX,\tildeP)\in U$ for all $(q_0,X,P) \in E$,
whenever $(q_0,\tildeX,\tildeP)$ is a vertical critical point
of our proposed generating function which we will denote by $S$ later.

\section{Construction of a canonical GFQI}
\label{sec:definitionGFQI}

For each $0 < \delta < \epsilon_0$, we put
\bea\label{eq:Udelta}
U_\delta & = & U_\delta(E): = \{(q,X,P) \mid \|X\| < \delta\} 
=  (D_\delta(TB))^{N-1} \times_B (T^*B)^{N-1} \nonumber \\
&\cong&
\left(D_\delta(TB) \times_B T^*B\right)^{N-1}.
\eea
 We also define $U_{\epsilon_0}$ similarly
for $\epsilon_0 < \iota_g$.

Recalling that $\Image \Psi_\delta \subset U_{\epsilon_0}$ by
construction, we can now define the composition $S = S_\delta : = \widetilde S \circ \Psi_\delta$, i.e.,
$$
S(q, X,P): = \widetilde S(q, \tildeX,\tildeP), \quad (\tildeX,\tildeP) = \Psi_\delta(X,P) 
$$
which is now defined on the whole space $E$ if $\widetilde S$ is originally defined on $U_{\epsilon_0} \subset E$.

The pairing \eqref{eq:<,>} naturally induces a fiberwise quadratic function on 
$$
E = (TB \oplus T^*B)^{\oplus(N-1)}
$$
by taking the sum of the pairing $\langle \cdot, \cdot \rangle$ on $TB \oplus T^*B$,
which we still denote by $\langle \cdot, \cdot \rangle: E \to \R$.

The following lemma is obvious. 

\begin{lem} \label{lem:obvious}  We have
$$
\langle \widetilde{X}_k,\widetilde{P}_k \rangle = \langle X_k,P_k \rangle
$$
on $TB \oplus T^*B$ for all $k$, and hence $\langle \widetilde{X},\widetilde{P} \rangle = \langle X, P \rangle$ on $E$.
\end{lem}

We consider the restriction $ \gamma|_{t_{\ell-1}}^{t_\ell}$ of the path $\gamma : [0,1] \to M$ to the subinterval $[t_{\ell-1}, t_\ell]$.  
\begin{lem} Let $\gamma_\ell:[0,1] \to J^1B$ be the rescaled path
$$
\gamma_\ell(t) = \gamma\left(t_{\ell-1} + \frac{t}{N}\right).
$$
Then we have 
$$
\CA_{H} (\gamma|_{t_{\ell-1}}^{t_\ell}): = \CA_{H_\ell} (\gamma_\ell)
$$
for the Hamiltonian 
\be\label{eq:Hell}
H_\ell = \frac1N H\left(t_{\ell-1}+\frac{t}{N}, \cdot\right)
\ee
\end{lem}
\begin{proof} By definition, we have
$$
H_\ell = \Dev_\lambda\left(t \mapsto \psi_{H;\ell}^{t_{\ell-1} + \frac{t}{N}}\right).
$$
The equality \eqref{eq:Hell} follows from a direct calculation.
Once this is derived, the equality of the action integral immediately follows.
\end{proof}

We are now ready to introduce our main function $A_k: E \to \R$ as the composition
\be\label{eq:Ake}
A_k(e) = A_k^U\circ \Psi_\delta(e)
\ee
where we define 
\be\label{eq:AkUe}
A_k^U(e) =
A_k^U(q_0, \tildeX, \tildeP) : = - \CA _{H,k} (\gamma_k) + \int_{\gamma_k} dz
\ee
for $ k = 1, \cdots, N$. Here we recall the notation $(\tildeX, \tildeP) = \Psi_\delta(X,P)$
and that the right hand side defined for the trajectory associated thereto
remains well-defined, even when $(X, P)$ lies outside the domain $U$.

\begin{rem}
Since each $\gamma_k$ is a Hamiltonian trajectory, $\CA_{H,k}(\gamma_k)$ always vanishes. 
But we add this term to \eqref{eq:AkUe} in order to emphasize that it is consistent with 
the effective action functional $\CA^{\rm CC}_H$ given in Definition \ref{defn:CAHCC},
and its first variation explicitly involves the Hamiltonian $H$. See also Remark \ref{rem:adoption}.
\end{rem}

We can derive that \eqref{eq:Ake} becomes
$$
A_k^U(q_0,X,P) = z(\gamma_k(1))  - z(\gamma_k(0))
$$
by rewriting the right hand side thereof,
where $\gamma_k (t) = (q_k (t), p_k (t), z_k (t))$ is the Hamiltonian
trajectory from 
$y_{k-1}^+  =  (q_{k-1}^+,p_{k-1}^+,z_{k-1}^+)$
to
$y_k^- = (q_k^-,p_k^-,z_k^-)$.
Here we write $y_0^+  =  (q_0, 0, 0) \in o_{J^1B}$ \emph{instead of $y^+ = (q,0,0)$} to emphasize 
that it is the initial condition of the piecewise smooth curves that we will consider now.

Then by definition, we have
\be\label{eq:zk+-}
z_{k-1}^+ = z\left(\gamma_k(0)\right),\quad z_k^- = z\left(\gamma_k(1)\right).
\ee
Therefore we have obtained
\be\label{eq:Ake2}
A_k^U(q_0,X,P) = z_k^- -z_{k-1}^+
\ee
regarding each summand of the right hand side as a function of $(q_0,X,P) \in E$.
 
We now take the sum
$$
A^U = A_1^U + \cdots + A_N^U
$$
and define a function $S: E \to \R$ to be
$$
S(q_0,X,P) =\widetilde S(q_0, \tildeX,\tildeP) = \sum_{k=1}^{N-1} \langle \widetilde{P}_k, \widetilde{X}_k \rangle + A^U.
$$
Then we can simply write 
\be\label{eq:S}
S(q_0,X,P) = \langle X, P\rangle + A = \langle \tildeX, \tildeP \rangle + A
\ee
 on $E$ where $A = A^U \circ \Psi_\delta$.
 
We now consider $E$ as a fiber bundle
$$
\Pi_E: E = (TB)^{\oplus (N-1)}\oplus (T^*B)^{\oplus (N-1)} \to B
$$
over the \emph{twisted projection} of taking the new `final point' projection
$$
e=(q_0,X,P) \mapsto (q_0,\widetilde{X},\widetilde{P}) \mapsto q^-_N:
$$
Here the input $e \in E$ is first rescaled to $(q_0, \widetilde{X}, \widetilde{P})= \Psi_\delta(e) \in U$, and then mapped to the point $q_N^-$,
via the cotangent projection $\pi: T^*B \to B$, of the final point of the broken trajectory 
\eqref{eq:broken-curve}
constructed from this rescaled data $\Psi_\delta(e)$. (See Figure \ref{fig:broken-traj}.) This replaces the `initial point' projection $(q_0, X, P) \mapsto q_0$ considered earlier.

\section{Analysis of the vertical critical point equation}
\label{sec:proofGFQI}

In this subsection, \emph{we do some critical point analysis of $S$ on $E$} 
for a fixed $N$ so large that Proposition \ref{prop:local-coord} and \ref{prop:psidelta} hold. 
By definition of $\rho_{\delta,\epsilon_0}$ and $\Psi_\delta$, we have
$$
\Psi_\delta|_{U_\delta} = \id|_{U_\delta}.
$$
Then utilizing Proposition \ref{prop:psidelta} and the definitions of $\widetilde S$ and 
$S$ from \eqref{eq:S}, we will show that
\be\label{eq:Sigma-coincide}
\Sigma_S = \Sigma_{S|_{U_\delta}} = \Sigma_{\widetilde S|_{U_\delta}}
\ee
\emph{provided $N$ is sufficiently large}. 
After then we have only to do the (vertical) critical point analysis of 
$\widetilde S$ on $U_{\epsilon_0}$ and make 
$$
\Sigma_{\widetilde S} \subset U_\delta.
$$
The rest of the present section will be occupied by this analysis on $\widetilde S$ on $U_{\epsilon_0}$
until we restore $S$ back at the end of the present section. 

For these purposes, 
we rewrite $A_k(e)$ as follows, as mentioned in Remark \ref{rem:adoption},
the expression of which plays an important role in concluding our construction of GFQI.
\be\label{eq:Ake=}
A_k(e) = \int_{\gamma_k} dz = \int_0^1 \left(p \frac{\del H_k}{\del p}
(t,\gamma_k(t)) - H_k (t,\gamma_k(t))\right)\, dt. 
\ee
where we have
\be\label{eq:Hk}
H_k = \frac{1}{N} H\left(t_{k-1} + \frac{t}{N}, \cdot\right)
\ee
 which is \eqref{eq:Hell} with $\ell$ replaced by $k$, and 
\eqref{eq:Ake=} follows from \eqref{eq:onshell-formula}.

An immediate consequence of this expression of $H_k$ is the following.
\begin{cor}\label{cor:subset}
There exists a sufficiently large $N > 0$ such that
$$
\Sigma_{\widetilde S} \subset U_\delta(E).
$$
\end{cor}
\begin{proof} Recall that $\widetilde S$ is defined on $U_{\epsilon_0}$ and $0 < \delta < \epsilon_0$ is
already fixed. It follows from the formula \eqref{eq:Hk} of the rescaled Hamiltonian $H_k$ that
 $\|H_k\|_{C^2} \leq \frac{C}{N}$. 
 In particular, using \eqref{eq:Hamilton-eq-Darboux}, we obtain
\beastar
\|\dot \gamma_k(t)\| & = & \max\{\|(\dot q,\dot p)(\gamma_k(t))\|, \|\dot z(\gamma_k(t))\|\} \\
& \leq &  \max\left\{\|X_{H_k}(t, \gamma_k(t))\|, \left|\left(p(\gamma_k(t)) \frac{\del H}{\del p}(\gamma_k(t))\right) - H_t(\gamma(t))\right|\right\}
\leq \frac{C}{N}
\eeastar
for all $t \in [0,1]$ if $N$ is sufficiently large. Here we also use the estimate 
$$
\|p(\gamma_k(t))\| \leq \int_0^t \left\|\frac{D H}{\del q}\right\|(\gamma(u))\, du \leq \frac{C}{N}
$$
so that
$$
|\dot z\circ \gamma_k(t)| \leq  \left|p(\gamma_k(t)) \frac{\del H}{\del p}(\gamma_k(t)) - H_t(\gamma(t))\right| \\
\leq \frac{C}{N}.
$$
 This implies that 
after parallel transport of the vector field $X_{H_k}(t,\gamma_k(t))$  
along $q((d_1 \# c_1)\# (d_2 \# c_2) \# \cdots (d_{k-1}\#c_{k-1}) \# d^t_k)$ with $d^t_k(s) =d_k(ts)$
back from $t \in [t_{k-1},t_k]$ to $t_0 = 0$, the transported vector field 
$\X_k(t) = (\widetilde X_k(t), \widetilde P_k(t), \widetilde z_k(t))$ in 
$T_{(q_0,0)}(J^1B) = T_{q_0}B \oplus T^*_{q_0}B \oplus \R$  have their norms satisfy
$$
\|\widetilde X_k(t)\| \leq \frac{C}{N},  \quad |\widetilde z_k| \leq \frac{C}{N}
$$
for all $k = 1, \, \cdots, N-1$.  By solving the ODE $\dot x = \X_k(t,x)$ defined on the vector space 
$T_{q_0}B\oplus T^*_{q_0}B \oplus \R$,  this estimate implies 
that the translated curve $\tildegamma_{(\tildeX_k,\tildeP_k)}$ satisfies
\be\label{eq:image-tildegamma}
\Image \left(\tildegamma_{(\tildeX_k,\tildeP_k)}\right) \subset D_{C/N}(J^1B)
\ee
for all $k$ and  $(q_0, \tildeX, \tildeP) \in U_{(N/C)}(E)$ by the definitions, \eqref{eq:DRJ1B} of 
$D_R(J^1B)$, \eqref{eq:Udelta} of $U_\delta$ and 
\eqref{eq:tildegammak}  of $\tildegamma_k= \tildegamma_{(\tildeX_k,\tildeP_k)}$.
In particular it implies
$$
\|(\tilde X_k, \tildeP_k)\| = \|\pi_{T^*B}(\tildegamma_{(\tildeX_k,\tildeP_k)}(1))\| < \frac{C}{N}.
$$
Therefore by taking $N$ sufficiently large, we have finished the proof.
\end{proof}

Next, we recall the notations of $\widetilde x_k, \widetilde p_k$ introduced in Step 3 of the construction of the broken trajectory; they are obtained by parallel transporting $\tildeX_k$ and $\tildeP_k$ to the base point $q_k^-$. Since the
parallel transport preserves inner product, we have 
$$
\widetilde S = \sum_{k=1}^{N-1} \langle \widetilde{p}_k,\widetilde{x}_k \rangle + A^U
$$ 
and hence
\be\label{eq:deltaS}
\delta \widetilde S = \sum_{k=1}^{N-1} (\langle \delta \widetilde{p}_k, \widetilde{x}_k \rangle 
+ \langle \widetilde{p}_k, \delta \widetilde{x}_k \rangle) + \sum_{k=1}^N \delta A_k^U.
\ee
For the notational simplicity, we will drop the superscript `U' from $A^U$ and write
$$
A_k = A_k^U
$$
in the rest of our calculations below as long as there is no danger of confusion.

From \eqref{eq:Ake2} we know that
\be\label{eq:deltaAi}
\delta A_k = \delta z^-_k - \delta z^+_{k-1}.
\ee
The following lemma reveals some similarity of the variation $\delta A_k$ in the current
contact case and in Laudenbach-Sikorav's symplectic case \cite{laud-sikorav}.
But we would also like to highlight the new appearance of the extra term
$$
 (e^{g_{(k-1)k}} - 1)\lambda (\delta y^+_{k-1})
$$
unlike the case of latter.
\begin{rem}
The appearance of  this term is purely of contact geometric nature
and makes the study of the variation $\delta A_k$ in the rest of 
calculations much more nontrivial and interesting than the symplectic case
of \cite{laud-sikorav}. The study requires us to perform 
a very careful and precise tensorial calculations unlike that of \cite{laud-sikorav}.
Readers might want to compare the degree of complexity of the calculations performed here 
and those in \cite{laud-sikorav}.
\end{rem}

\begin{lem}\label{lem:deltaAi}
\be\label{eq:deltaAi2}
\delta A_k = \langle p_k^-, \delta q_k^- \rangle - \langle p_{k-1}^+ , \delta q_{k-1}^+ \rangle
+ (e^{g_{(k-1)k}} - 1)\lambda (\delta y^+_{k-1})
\ee
where $g_{(k-1)k} = g_{\psi_{H;k}}(y_{k-1}^+)$ for all $k = 1,\ldots, N$.
\end{lem}

\begin{proof}
Note that we have
\be\label{eq:lambda-deltayk}
\lambda (\delta y^\pm_k) = \delta z^\pm_k - \langle p^\pm_k, \delta q^\pm_k \rangle 
\ee
Since $\delta \gamma_k$ is a variation of Hamiltonian trajectories, we have
$$
\delta \gamma_k (t) = d\psi_{H;k}^{t_{k-1}+\frac{t}{N}}(\delta y^+_{k-1})
$$
for $0 \leq t \leq 1$. 
By \eqref{eq:lambda-deltayk}, we have
\beastar
\delta A_k &=& \langle p_k^-, \delta q_k^- \rangle - \langle p_{k-1}^+ , \delta q_{k-1}^+ \rangle
+ \lambda (\delta y^-_k) - \lambda (\delta y^+_{k-1}) \\
&=& \langle p_k^-, \delta q_k^- \rangle - \langle p_{k-1}^+ , \delta q_{k-1}^+ \rangle
+ (e^{g_{(k-1)k}} - 1)\lambda (\delta y^+_{k-1}).
\eeastar
since $\delta y^-_k = \delta \gamma_k (1)$.
\end{proof}

\begin{lem}
We have the recurrence relation
\be\label{eq:lambda-deltayk+rec}
\lambda (\delta y^+_k)
= e^{g_{(k-1)k}}\lambda (\delta y^+_{k-1}) + \langle \delta \widetilde{p}_k , \widetilde{x}_k \rangle
- \langle \left( D_1\operatorname{Exp}(q^-_k,\widetilde{x}_k) \right)^* (p^+_k) - p^-_k, \delta q^-_k \rangle
\ee
for $k = 1, \ldots, N-1$ with $\lambda (\delta y^+_0) = 0$.
\end{lem}
\begin{proof}
Note that
\bea\label{eq:z+k-z-k}
z^+_k - z^-_k &=& \int_{c_k} pdq = \int_0^1 pdq(\dot c _k (t))dt  \nonumber\\
&=& \int_0^1 \left\langle \left((d\exp_{q_k^-}(t\widetilde{x}_k))^*\right)^{-1}(\widetilde{p}_k),
d\exp_{q_k^-}(t\widetilde{x}_k)(\widetilde{x}_k) \right\rangle\, dt \nonumber \\
&=& \langle \widetilde{p}_k, \widetilde{x}_k \rangle
\eea
for each $k$.
Then we have
$$
\lambda (\delta y^+_k) = \delta z^+_k - \langle p^+_k, \delta q^+_k \rangle
= \delta z^-_k + \langle \delta \widetilde{p}_k, \widetilde{x}_k \rangle + \langle \widetilde{p}_k, \delta \widetilde{x}_k \rangle 
- \langle p^+_k, \delta q^+_k \rangle.
$$
Since
$$
q^+_k = \exp_{q^-_k}(\widetilde{x}_k) = \operatorname{Exp}(q^-_k,\widetilde{x}_k)
$$
we have
$$
\delta q^+_k = D_1\operatorname{Exp}(q^-_k,\widetilde{x}_k) \delta q^-_k 
+ (d_2\operatorname{Exp}(q^-_k,\widetilde{x}_k))\delta \widetilde{x}_k.
$$
Then we have
\bea
\langle p^+_k, \delta q^+_k \rangle 
&=& \langle \left( D_1\operatorname{Exp}(q^-_k,\widetilde{x}_k) \right)^* (p^+_k), \delta q^-_k \rangle
+ \langle (d\exp_{q^-_k} (\widetilde{x}_k))^*(p^+_k), \delta \widetilde{x}_k \rangle \nonumber \\
&=& \langle \left( D_1\operatorname{Exp}(q^-_k,\widetilde{x}_k) \right)^* (p^+_k), \delta q^-_k \rangle
+ \langle \widetilde{p}_k , \delta \widetilde{x}_k \rangle \label{eq:pdqk+}
\eea
since
$$
d_2\operatorname{Exp}(q^-_k, \widetilde{x}_k) = d\exp_{q^-_k}(\widetilde{x}_k).
$$
Now we have
\beastar
\lambda (\delta y^+_k)
&=& \delta z^-_k + \langle \delta \widetilde{p}_k, \widetilde{x}_k \rangle 
+ \langle \widetilde{p}_k, \delta \widetilde{x}_k \rangle 
- \langle p^+_k, \delta q^+_k \rangle \\
&=& \delta z^-_k + \langle \delta \widetilde{p}_k, \widetilde{x}_k \rangle 
- \langle \left( D_1\operatorname{Exp}(q^-_k,\widetilde{x}_k) \right)^* (p^+_k), \delta q^-_k \rangle.
\eeastar
Finally, by using \eqref{eq:deltaAi}, \eqref{eq:deltaAi2}, and \eqref{eq:lambda-deltayk},
we have the recurrence relation
\beastar
\lambda (\delta y^+_k)
&=& \delta z^+_{k-1} + \langle p^-_k, \delta q^-_k \rangle - \langle p^+_{k-1}, \delta q^+_{k-1} \rangle 
+ (e^{g_{(k-1)k}} - 1)\lambda(\delta y^+_{k-1}) \\
& &+ \langle \delta \widetilde{p}_k, \widetilde{x}_k \rangle 
- \langle \left( D_1\operatorname{Exp}(q^-_k,\widetilde{x}_k) \right)^* (p^+_k), \delta q^-_k \rangle \\
&=& e^{g_{(k-1)k}}\lambda (\delta y^+_{k-1}) + \langle \delta \widetilde{p}_k , \widetilde{x}_k \rangle
- \langle \left( D_1\operatorname{Exp}(q^-_k,\widetilde{x}_k) \right)^* (p^+_k) - p^-_k, \delta q^-_k \rangle.
\eeastar
\end{proof}

The following proposition is the contact counterpart of 
\cite[Lemma 2.1]{laud-sikorav}. However we would like to attract readers' attention to
the fact that its proof is much more difficult than that of \cite[Lemma 2.1]{laud-sikorav}.

\begin{prop}\label{prop:local-coord}
Let $\epsilon_0 > 0$ be the constant appearing in \eqref{eq:epsilon0}, and assume $\epsilon_0 < \iota_g$.
Then there exists a sufficiently large integer $N$ and small $\epsilon_0> 0$ such that the
collection $\{(q_k^-, \widetilde{p}_k)\}_{k=1}^{N-1}$  and $q_N^-$ form a system of local coordinates on $U$.
\end{prop}
\begin{proof} We will choose $\epsilon_0$ and $ N > 4(2C_1 + C_2 R_0)$ given as in \eqref{eq:epsilon0N}.
We need to study the Jacobian of the map
$$
(q_0,\widetilde{X},\widetilde{P}) \mapsto \left(q_1^-, \widetilde p_1, \cdots, q_{N-1}^-, \widetilde p_{N-1}, q_N^-\right).
$$
For this purpose, we  first express $\delta x_k$ as a linear combination 
\be\label{eq:deltaxk3}
\delta \widetilde x_k = \sum_{i=1}^{k+1}a_k^i \delta q^-_i + \sum_{i=1}^k b_k^i \delta \widetilde p_i
\ee
of 
$$
\{\delta q^-_i \}_{i=1}^{k+1}  \cup \{ \delta \widetilde p_i\}_{i=1}^k.
$$
After some tedious calculations, the details of which are postponed to Appendix \ref{sec:deltax},
we obtain the following formula (See equation \eqref{eq:Dkdeltaxk}.):
\bea\label{eq:deltakx1}
& &\left(D_1(q\circ \psi_{k+1})d\exp_{q^-_k}(\widetilde{x}_k) 
+ d_3(q\circ \psi_{k+1})\langle \widetilde{p}_k, \cdot \rangle\right) \delta \widetilde{x}_k 
\nonumber\\
&=& \delta q^-_{k+1}
- \left(D_1(q\circ \psi_{k+1})D_1\operatorname{Exp}(q^-_k,\widetilde{x}_k) 
+ d_3(q\circ\psi_{k+1})\langle p^-_k, \cdot \rangle\right) \delta q^-_k 
\nonumber\\
& &- \left(d_2(q\circ \psi_{k+1})(d\exp_{q^-_k}(\widetilde{x}_k)^*)^{-1} 
+ d_3(q\circ \psi_{k+1})\langle \cdot, \widetilde{x}_k \rangle \right)\delta \widetilde{p}_k 
\nonumber\\
& &- d_3(q\circ \psi_{k+1}) \sum_{i=1}^{k-1} e^{g_{ik}}(\langle \delta \widetilde{p}_i, \widetilde{x}_i \rangle 
- \langle r_i, \delta q^-_i \rangle)
\eea
where we put
\be\label{eq:ri}
r_i = D_1\operatorname{Exp}(q^-_i,\widetilde{x}_i)^*(p^+_i) - p^-_i.
\ee
Therefore we can rewrite \eqref{eq:deltakx1} into
\bea\label{eq:deltaxk2}
\delta \widetilde{x}_k & = & 
(D^k)^{-1}\Big[\delta q^-_{k+1} 
- \left(D_1(q\circ \psi_{k+1})D_1\operatorname{Exp}(q^-_k,\widetilde{x}_k) 
+ d_3(q\circ\psi_{k+1})\langle p^-_k, \cdot \rangle\right) \delta q^-_k  \nonumber \\
& &- \left(d_2(q\circ \psi_{k+1})(d\exp_{q^-_k}(\widetilde{x}_k)^*)^{-1} 
+ d_3(q\circ \psi_{k+1})\langle \cdot, \widetilde{x}_k \rangle \right)\delta \widetilde{p}_k \nonumber \\
& &- d_3(q\circ \psi_{k+1}) \sum_{i=1}^{k-1} e^{g_{ik}}(\langle \delta \widetilde{p}_i, \widetilde{x}_i \rangle 
- \langle r_i, \delta q^-_i \rangle)\Big]
\eea
if we make a sufficiently fine partition of the interval $[0,1]$ (i.e., by letting $N$ sufficiently large), 
and choose $\epsilon_0$ sufficiently small.

From the formula \eqref{eq:deltaxk2}, it is easy to see that the coordinate change map inside 
the big square parenthesis  has its Jacobian that has the form of an upper triangular matrix with nonzero diagonal 
elements with respect to the basis
$$
\left((\delta q^-_k, \delta \widetilde{p}_k)_{k=1,\ldots,N-1}, \delta q^-_{k+1}\right).
$$
Since each $D^k$ is invertible, we conclude that the coordinate change map 
$$
\left((q^-_k, \widetilde{p}_k)_{k=1,\ldots,N-1}, q^-_N\right) \mapsto (q_0, \widetilde{X}, \widetilde{P})
$$
is invertible. For readers' convenience, we provide details of this claim in
Appendix \ref{sec:invertibility}.
This finishes the proof of Proposition \ref{prop:local-coord}.
\end{proof}

Now we prove the following proposition.

\begin{prop}\label{prop:vert-crit}
The differential of $\widetilde S$ is written as
\be\label{eq:deltaS2}
\delta \widetilde S = \sum_{k=1}^{N-1} e^{g_k}(\langle \delta \widetilde{p}_k, \widetilde{x}_k \rangle 
- \langle r_k, \delta q^-_k \rangle)
+ \langle p^-_N, \delta q^-_N \rangle
\ee
where 
$$
g_k := \sum_{i=k+1}^{N} g_{(i-1)i} = g_{\psi_{k+1}}(y^+_k) + \cdots + g_{\psi_{N}}(y^+_{N-1})
$$
and
$r_k = \left( D_1\operatorname{Exp}(q^-_k,\widetilde{x}_k) \right)^* (p^+_k) - p^-_k$ as defined
in \eqref{eq:ri}. 
\end{prop}
\begin{proof}
By \eqref{eq:deltaS}, \eqref{eq:deltaAi2} and \eqref{eq:pdqk+}, we have
\beastar
\delta \widetilde S &=& \sum_{k=1}^{N-1} (\langle \delta \widetilde{p}_k, \widetilde{x}_k \rangle 
+ \langle \widetilde{p}_k, \delta \widetilde{x}_k \rangle) \\
& &+ \sum_{k=1}^{N} \left(\langle p^-_k, \delta q^-_k \rangle - \langle p^+_{k-1}, \delta q^+_{k-1} \rangle 
+ (e^{g_{(k-1)k}} - 1)\lambda(\delta y^+_{k-1})\right) \\
&=& \sum_{k=1}^{N-1} (\langle \delta \widetilde{p}_k, \widetilde{x}_k \rangle 
+ \langle \widetilde{p}_k, \delta \widetilde{x}_k \rangle
+ \langle p^-_k, \delta q^-_k \rangle - \langle p^+_k, \delta q^+_k \rangle) \\
& &+ \langle p^-_N, \delta q^-_N \rangle + \sum_{k=1}^{N-1} (e^{g_{k(k+1)}} - 1)\lambda(\delta y^+_k) \\
&=& \sum_{k=1}^{N-1} (\langle \delta \widetilde{p}_k, \widetilde{x}_k \rangle - \langle r_k, \delta q^-_k \rangle
+ (e^{g_{k(k+1)}} - 1)\lambda(\delta y^+_k)) + \langle p^-_N, \delta q^-_N \rangle.
\eeastar
For simplicity of the exposition, we write
$$
\Delta_k := \langle \delta \widetilde{p}_k, \widetilde{x}_k \rangle - \langle r_k, \delta q^-_k \rangle.
$$
By the recurrence relation \eqref{eq:lambda-deltayk+rec}, we have
\beastar
\delta \widetilde S &=& \sum_{k=1}^{N-1} \left(\Delta_k 
+ (e^{g_{k(k+1)}} - 1)(e^{g_{(k-1)k}}\lambda(\delta y^+_{k-1}) + \Delta_k)\right)
+ \langle p^-_N, \delta q^-_N \rangle \\
&=& \sum_{k=1}^{N-2} \left(e^{g_{k(k+1)}}\Delta_k + (e^{g_{(k+1)(k+2)}} - 1)e^{g_{k(k+1)}}\lambda(\delta y^+_k)\right) \\
& &+ e^{g_{(N-1)N}}\Delta_{N-1} + \langle p^-_N, \delta q^-_N \rangle \\
&=& \cdots
= \sum_{k=1}^{N-1} e^{g_{(N-1)N}}\cdots e^{g_{k(k+1)}}\Delta_k
+ \langle p^-_N, \delta q^-_N \rangle
\eeastar

To obtain the vertical critical points, we fix $q_N^-$ and compute the first variation of $\widetilde S$. Then the last term $\langle p_N^-, \delta q_N^-\rangle$ in the above expression for $\delta \widetilde S$ vanishes, and we obtain the vertical critical equation
$$
\widetilde{x}_k = 0, \quad r_k = 0
$$
for all $k=1,\ldots, N-1$.
\end{proof}

The following is an immediate corollary of this proposition.
\begin{cor}\label{cor:smoothing}
For any vertical critical point $e = (q_0, X, P) \in E$ of $\widetilde S$, we have 
$$ 
\widetilde{x}_k = 0, \quad r_k = 0 
$$
for all $k$. In particular, the broken Hamiltonian trajectory 
\be\label{eq:gamma-concatenation}
\gamma_1 \# \cdots \# \gamma_N
\ee
associated to $\{\gamma_k\}_{1 \leq k \leq N}$ joins into one smooth trajectory 
$\gamma$ from $o_{J^1B}$ to $\psi^1_H (o_{J^1B})$.
\end{cor}
\begin{proof}
Note that the condition $\widetilde{x}_k = 0$ gives rise to 
$$
q_k^- = q_k^+, \quad \widetilde{p}_k = p_k^+, \quad z_k^- = z_k^+
$$
where the last equality follows from \eqref{eq:z+k-z-k}.
In addition, the condition $r_k = 0$ also implies
$$
\gamma_{k} (1) = (q_k^-, p_k^-, z_k^-) = (q_k^+, p_k^+, z_k^+) = \gamma_{k+1} (0).
$$
Therefore the concatenation \eqref{eq:gamma-concatenation} associated to
$\{\gamma_k\}_{1\leq k\leq N}$ becomes one continuous and piecewise smooth trajectory 
$\gamma$ leaving from $o_{J^1B}$ at $t = 0$ and reaching $\psi^1_H(o_{J^1B})$ at $t=1$.
Since $\gamma$ is \emph{continuous} and satisfies the ODE
$$
\dot \gamma (t) = X_H(t, \gamma(t))
$$
\emph{away from a finite number of points $\{t_k\}$}, it must be smooth: This is a consequence of 
the standard existence, uniqueness and differentiability theorem of first order ODE 
(see \cite[Section 32]{arnold:ODE}, for example) or can be directly shown
by the standard boot-strap argument.
Therefore such  $\gamma$ is a smooth contact Hamiltonian trajectory of $H$.
\end{proof}

We denote by $\Sigma_{\widetilde S}$ the vertical critical locus of $\widetilde S$.
Recall that the Legendrian immersion $\iota_{\widetilde S} : \Sigma_{\widetilde S} \to J^1B$ is given by
$$
\iota_{\widetilde S} (e) = (\Pi_E(e), \delta \widetilde S(e), \widetilde S(e)).
$$
For each $e \in \Sigma_{\widetilde S}$, the differential satisfies 
$\delta \widetilde S(e) = \langle p_N^-, \delta q_N^- \rangle$ since $e$ is a vertical critical point.
Moreover, $e$ corresponds to a single smooth Hamiltonian trajectory $\gamma$ connecting $o_{J^1B}$ to $\psi_H^1(o_{J^1B})$, and the value $\widetilde S(e)$ is given by the $z$-coordinate difference:
$$
\widetilde S(e) = \int_\gamma dz = z(\gamma(1))-z(\gamma(0))=z_N^-.
$$
Hence, we recover
$$
\iota_{\widetilde S}(e)=(q_N^-,p_N^-,z_N^-)=\gamma(1) \in \psi_H^1(o_{J^1B}),
$$
which shows that $\widetilde S$ generates $\psi_H^1(o_{J^1B})$.

This would have finished the proof of the following theorem
\emph{if we have been working on $S = \widetilde S \circ \Psi_\delta$}. Recall that we have
been working on $\widetilde S$ on $U_{\epsilon_0}$.

\begin{prop}[Compare with Lemma 2.6, \cite{laud-sikorav}] Let $0 < \epsilon_0 < \iota_g$ and $0 < \delta < \epsilon_0$
be as in \eqref{eq:norm-tildeXk}. Then provided $N> 0$ is sufficiently large,
\begin{enumerate}
\item
the vertical derivative
$\delta^{\rm vert}(A \circ \Psi_\delta)$ on $E = E_N$ is  bounded, 
\item $\Sigma_{\widetilde{S} \circ \Psi_\delta} = \Sigma_{S} \cap U_{\epsilon_0}$, which in particular implies 
$\psi_H^1(o_{J^1B}) = R_S$ for $S = \widetilde{S} \circ \Psi_\delta$. 
\end{enumerate}
\end{prop}
\begin{proof}
We have already shown in \eqref{eq:image-tildegamma} that if we choose $\epsilon_0 < \iota_g$ and 
$0 < \delta \ll \epsilon_0$ and $N$ sufficiently large,
$$
\Image \tildegamma_{(\tildeX_k,\tildeP_k)} \subset U_{\epsilon_0}(J^1B)
$$
(see \eqref{eq:tildegammak} for the definition of $\tildegamma_{(\tildeX_k,\tildeP_k)}$),
and $\Sigma_{\widetilde S} \subset U_\delta(E)$ from Corollary \ref{cor:subset}
for all $k$.
Furthermore $\Psi_\delta: E \to U_{\epsilon_0}$ is well-defined.
The definition $S = \widetilde S \circ \Psi_\delta$ also implies
$$
d^v S = d^v \widetilde S \circ D\Psi_\delta.
$$
Since $\Psi_\delta: E \setminus U_\delta(E) \to 
U_{\epsilon_0}(E)\setminus U_\delta(E)$  is a 
diffeomorphism and $\Sigma_{\widetilde S} \cap (U_{\epsilon_0}(E)\setminus U_\delta(E)) = \emptyset$,
this completes the proof of
$$
R_S = \psi_H^1(o_{J^1B}).
$$
\end{proof}

By this proposition, we now conclude the proof of the main theorem.

\appendix

\section{The Carnot path space and horizontal curves of the contact distribution}
\label{sec:Carnot-horizontal}

In this section, we provide some relationship between $\CL^{\text{\rm Carnot}}(M,H)$ and
the space of horizontal curves of contact distribution $\xi$ on $M$.

We first recall the notion of \emph{horizontal curves} of a smooth distribution on a manifold $M$
in general. (See \cite{mitchell} for some background on the general study of the space of horizontal curves
of a distribution satisfying H\"ormander's condition.) We closely follow the exposition of
\cite{mitchell} on the Carnot-Carath\'eodory metrics on the distribution satisfying H\"ormander's condition.

\begin{defn} Let $\Delta$ be a distribution of a smooth (connected) manifold $M$. An absolutely
continuous curve $\alpha$ on $M$ is said to be horizontal if it is a.e. tangent to the
distribution $\Delta$.
\end{defn}

Now we recall the following notion of H\"ormander's condition.

\begin{defn}[H\"ormander \cite{hormander:hypoelliptic}] 
Let $\{X_1, X_2, \cdots, X_k\}$ be a local basis of vector fields for $\Delta$ near $m \in M$.
If these vector fields, along with all their commutators, span $T_mM$, then these vector fields are said to
satisfy H\"ormander's condition at $m$.
\end{defn}

It is well-known and easy to check that the above definition does not depend on the choice of local
basis $\{X_1,X_2, \cdots, X_k\}$. Then the following fundamental local transitivity of
space of horizontal curves is proved by Chow \cite{chow}.

\begin{thm}[Chow]\label{thm:local-transitivity}
 If a smooth distribution $\Delta$ satisfies H\"ormander's condition at $m \in M$, then any point $p \in M$
which is sufficiently close to $m$ may be joined to $m$ by a horizontal curve.
\end{thm}

An immediate consequence of this theorem is the following

\begin{cor} Suppose $M$ is connected and equipped with a Riemannian metric.
Let $\Delta$ be a distribution that satisfies H\"ormander's condition.
Consider the function $d_c: M \times M \to \R_{\geq 0}$ defined by
$$
d_c(p,q): = \inf_{\ell \in C_{p,q}}\{\leng(\ell)\}
$$
where $C_{p,q}$ is the set of all horizontal curves which join $p$ to $q$ and $\leng(\ell)$ is the length of
$\ell$ associated to the given Riemannian metric. Then $d_c$ defines a finite metric.
We call $d_c$ the Carnot-Carath\'eodory distance associated to the distribution $\Delta$.
\end{cor}

The following is another interesting result proved by Mitchell \cite{mitchell}.
Denote by $V_i(m)$ the subspace of $T_mM$ spanned by all commutators of the $X_j$'s of order $\leq i$
(including the $X_j$'s). It is easy to see that $V_i(m)$ does not depend on the choice of the local basis
$\{X_j\}$.

\begin{defn}[Mitchell \cite{mitchell}] A distribution $\Delta$ is said to be \emph{generic}
if, for each $i\geq 1$, $\dim(V_i(m))$ is independent of the point $m$.
\end{defn}

\begin{thm}[Theorem 2 \cite{mitchell}]\label{thm:mitchell-dim} For a generic distribution $\Delta$ the Hausdorff dimension of the
metric space $(M,d_c)$ is
$$
Q = \sum_{i=1} i (\dim V_i - \dim V_{i-1}).
$$
\end{thm}

Now we specialize the above discussion to the case of contact distribution $\xi$ of contact manifolds $(M,\xi)$.

\begin{prop} The contact distribution $\xi$ satisfies H\"ormander's condition which is
also generic. Moreover the associated Carnot-Carath\'eodory metric space $(M,d_c)$
has Hausdorff dimension $2n+2$.
\end{prop}
\begin{proof} Let $m \in M$ and take a Darboux chart $(q_1, \cdots, q_n, p_1, \cdots, p_n,z)$
on $U$,
so that $\xi = \ker \left(dz - \sum_i p_i dq_i\right)$. Then we can choose a local basis of $\xi$ formed by
$$
\left\{\frac{D}{\del q_i} = \frac{\del}{\del q_i} + p_i \frac{\del}{\del z}, \frac{\del}{\del p_i}\right\}
$$
with $k = 2n$. (See \cite{LOTV} for this choice.) Write $X_i = \frac{D}{\del q_i}$ and $Y_i = \frac{\del}{\del p_i}$.
A direct calculation shows
$$
\{X_i,X_j\} = 0, \, \{Y_i,Y_j\} = 0, \, \{X_i,Y_j\} = - \frac{\del}{\del z}.
$$
Therefore we have
\beastar
V_1(m) & = & \span_\R\left \{\frac{D}{\del q_i} = \frac{\del}{\del q_i} + p_i \frac{\del}{\del z}, \frac{\del}{\del p_i}\right\}\Big|_m\\
V_2(m) & = & \span_\R\left \{\frac{D}{\del q_i} = \frac{\del}{\del q_i} + p_i \frac{\del}{\del z}, \frac{\del}{\del p_i}, \, \frac{\del}{\del z}\right\}\Big|_m = T_m M\\
V_i(m) & = & T_mM \quad \text{\rm for all $i \geq 2$}.
\eeastar
Obviously  the ranks of these vector spaces do not change on $U$ and so $\xi$ is generic
and $V_2(m)$ already spans $T_mM$. Therefore $\xi$ satisfies H\"ormander's condition and is generic.
Furthermore we have derived from Theorem \ref{thm:mitchell-dim}
 the Hausdorff dimension of the associated Carnot-Carath\'eodory metric $d_c$ given by
$$
1 \times (2n-0) + 2 \times (2n+1 - 2n) = 2n+2.
$$
\end{proof}

\begin{cor} \label{cor:xi-locally-transitive}
The set of horizontal paths of $\xi$ is locally transitive.
\end{cor}

\section{Formula of $\delta \widetilde{x}$ with respect to the variations $(\delta q^-, \delta \widetilde{p})$}
\label{sec:deltax}

In this section we express $\delta \widetilde{x}_k$ 
with respect to the coordinate $(q^-_i, \widetilde{p}_i)$. 
Note that
\be\label{eq:deltaqk-}
\delta q^-_{k+1} = D_1(q\circ \psi_{k+1})\delta q^+_k + d_2(q\circ \psi_{k+1})\delta p^+_k
+ d_3(q\circ \psi_{k+1})\delta z^+_k. 
\ee

Recall that
\be\label{eq:qk+}
q^+_k = \exp_{q^-_k}(\widetilde{x}_k),
\ee
implies that
\be\label{eq:deltaqk+}
\delta q^+_k = D_1\operatorname{Exp}(q^-_k, \widetilde{x}_k) \delta q^-_k 
+ d_2\operatorname{Exp}(q^-_k, \widetilde{x}_k) \delta \widetilde{x}_k.
\ee

Since 
\be \label{eq:pk+}
p^+_k = (d\exp_{q^-_k}(\widetilde{x}_k)^*)^{-1}(\widetilde{p}_k) = \Pi_{q^-_k, q(\mu_k(t))}^{q^+_k} (\widetilde{p}_k)
\ee
where $\Pi_{q^-_k, q(\mu_k(t))}^{q^+_k}$ denotes the parallel transport map from $q^-_k$ to $q^+_k$ along $q(\mu_k(t))$,
we have
\be\label{eq:deltapk+}
\delta p^+_k = \Pi_{q^-_k, q(\mu_k(t))}^{q^+_k} \delta \widetilde{p}_k.
\ee

Since
$$
z^+_k - z^-_k = \langle \widetilde{p}_k, \widetilde{x}_k \rangle, \qquad z^-_k - z^+_{k-1} = A_k,
$$
we have
\bea
z^-_k &=& \sum_{i=1}^{k-1} \langle \widetilde{p}_i, \widetilde{x}_i \rangle + \sum_{i=1}^k A_i \label{eq:zk-}\\
z^+_k &=& z^-_k + \langle \widetilde{p}_k, \widetilde{x}_k \rangle 
= \sum_{i=1}^{k} \left(\langle \widetilde{p}_i, \widetilde{x}_i \rangle + A_i\right). \label{eq:zk+}
\eea
In particular, the total generating function satisfies
$$
S=\sum_{k=1}^{N-1} \langle \widetilde{p}_k,\widetilde{x}_k \rangle + \sum_{k=1}^N A_k = z^-_N.
$$
Therefore, by comparing this with \eqref{eq:zk-}, we observe that $z_k^-$ is given by the same expression as $S$, truncated at $k$. This directly implies that $\delta z_k^-$ can be obtained from the formula for $\delta S$ 
(namely, \eqref{eq:deltaS2}) by simply replacing $N$ with $k$. Explicitly, we have
\be\label{eq:deltazk-}
\delta z^-_k = \sum_{i=1}^{k-1} e^{g_{ik}}(\langle \delta \widetilde{p}_i, \widetilde{x}_i \rangle 
- \langle r_i, \delta q^-_i \rangle) + \langle p^-_k, \delta q^-_k \rangle
\ee
where
$$
g_{ik} = \sum_{j=i+1}^k g_{(j-1)j} = g_{\psi_{i+1}}(y^+_i) + \cdots + g_{\psi_{k}}(y^+_{k-1})
$$
and
$$
r_i = D_1\operatorname{Exp}(q^-_i,\widetilde{x}_i)^*(p^+_i) - p^-_i.
$$

By applying these results to \eqref{eq:deltaqk-}, we have
\bea
& &\left(D_1(q\circ \psi_{k+1})d\exp_{q^-_k}(\widetilde{x}_k) 
+ d_3(q\circ \psi_{k+1})\langle \widetilde{p}_k, \cdot \rangle\right) \delta \widetilde{x}_k
\nonumber \\
&=& \delta q^-_{k+1}
- \left(D_1(q\circ \psi_{k+1})D_1\operatorname{Exp}(q^-_k,\widetilde{x}_k) 
+ d_3(q\circ\psi_{k+1})\langle p^-_k, \cdot \rangle\right) \delta q^-_k
\nonumber  \\
& &- \left(d_2(q\circ \psi_{k+1})(d\exp_{q^-_k}(\widetilde{x}_k)^*)^{-1} 
+ d_3(q\circ \psi_{k+1})\langle \cdot, \widetilde{x}_k \rangle \right)\delta \widetilde{p}_k 
\nonumber\\
& &- d_3(q\circ \psi_{k+1}) \sum_{i=1}^{k-1} e^{g_{ik}}(\langle \delta \widetilde{p}_i, \widetilde{x}_i \rangle 
- \langle r_i, \delta q^-_i \rangle) \label{eq:Dkdeltaxk}
\eea
where we put
$$
r_i = D_1\operatorname{Exp}(q^-_i,\widetilde{x}_i)^*(p^+_i) - p^-_i
$$
as in \eqref{eq:ri}.  It is already shown in Sublemma \ref{sublem:Dk-Pi}
that  the linear map
$$
D^k := D_1(q\circ \psi_{k+1})d\exp_{q^-_k}(\widetilde{x}_k) 
+ d_3(q\circ \psi_{k+1})\langle \widetilde{p}_k, \cdot \rangle
$$
introduced in \eqref{eq:Dk} is invertible for any sufficiently large $N$. 

We summarize the above calculations into the following lemma.

\begin{lem}\label{lem:deltaxk}
Provided $N$ is sufficiently large,
we can express $\delta \widetilde{x}_k$  as
\bea\label{eq:deltaxk2-appendix}
\delta \widetilde{x}_k &=& (D^k)^{-1}\Big[\delta q^-_{k+1} 
- \left(D_1(q\circ \psi_{k+1})D_1\operatorname{Exp}(q^-_k,\widetilde{x}_k) 
+ d_3(q\circ\psi_{k+1})\langle p^-_k, \cdot \rangle\right) \delta q^-_k  \nonumber \\
& &- \left(d_2(q\circ \psi_{k+1})(d\exp_{q^-_k}(\widetilde{x}_k)^*)^{-1} 
+ d_3(q\circ \psi_{k+1})\langle \cdot, \widetilde{x}_k \rangle \right)\delta \widetilde{p}_k \nonumber \\
& &- d_3(q\circ \psi_{k+1}) \sum_{i=1}^{k-1} e^{g_{ik}}(\langle \delta \widetilde{p}_i, \widetilde{x}_i \rangle 
- \langle r_i, \delta q^-_i \rangle)\Big]
\eea

\end{lem}

\section{Invertibility of the coordinate change map}
\label{sec:invertibility}

We now verify
the claim at the end of the proof of Proposition~\ref{prop:local-coord} in the following lemma.

\begin{lem}\label{lem:invertibility}
The coordinate change map
$$
\Big((q_k^-,\widetilde p_k)_{k=1}^{N-1},\ q_N^-\Big) \longmapsto (q_0,\tildeX,\tildeP)
$$
is locally invertible, provided $N$ is sufficiently large and $\epsilon_0$ is sufficiently small.
\end{lem}

We work at the level of differentials. Put the input variations in the order
$$
\Big((\delta q_k^-,\delta \widetilde p_k)_{k=1}^{N-1},\ \delta q_N^-\Big)
$$
and consider the collection of outputs
$$
\Big(\delta q_0,\ (D^1(\delta\widetilde x_1),\delta\widetilde p_1),\ \ldots, \
(D^{N-1}(\delta\widetilde x_{N-1}),\delta\widetilde p_{N-1})\Big).
$$
(Here $\delta q_0$ is regarded as $\delta q_1^-$ transported back to $T_{q_0}B$.)
From equation~\eqref{eq:deltakx1}, we obtain an identity of the form
\begin{equation}\label{eq:triangular-recurrence}
D^k(\delta\widetilde x_k)
= \delta q_{k+1}^-
\;+\;(\text{a linear combination of } \delta q_i^-,\,\delta\widetilde p_i \text{ with } i\le k),
\end{equation}
for each $k=1,\ldots,N-1$.
Hence, with respect to the ordering
$$
\left(\delta q_1^-,\delta q_2^-,\delta\widetilde p_1,\delta q_3^-,\delta\widetilde p_2,\ldots,\delta q_N^-\right)
$$
and
$$
\left(\delta q_0,\delta\widetilde p_1,D^1(\delta\widetilde x_1),\delta\widetilde p_2,D^2(\delta\widetilde x_2),
\ldots,D^{N-1}(\delta\widetilde x_{N-1})\right),
$$
the linear map induced by \eqref{eq:triangular-recurrence} is block upper triangular, and
its diagonal blocks are the identity maps on $T_{q_{k+1}^-}B$ (coming from the term $\delta q_{k+1}^-$)
together with the identity on each $\delta\widetilde p_k$.

Recall from Sublemma \ref{sublem:Dk-Pi} that each of
the linear maps $D^k$ is invertible for $k = 1, \cdots, N-1$.
We now pass from $(D^k\delta\widetilde x_k,\delta\widetilde p_k)$ 
to $(\delta\widetilde x_k,\delta\widetilde p_k)$ using the invertibility of $D^k$.
This process is simply the block-diagonal change of variables
\[
(D^k(\delta\widetilde x_k),\delta\widetilde p_k)\ \longmapsto\ (\delta\widetilde x_k,\delta\widetilde p_k),
\qquad
\delta\widetilde x_k = (D^k)^{-1}\left(D^k(\delta\widetilde x_k)\right),
\]
which is invertible since each $D^k$ is invertible 
(for $N$ large and $\epsilon_0$ small).

Finally, by definition, $(\widetilde x_k,\widetilde p_k)$ are obtained from
$(\tildeX_k,\tildeP_k)$ by the parallel transport along a fixed concatenated base path
(depending smoothly on the data). For each fixed base path, parallel transport is a linear isomorphism
between the fibers, hence its differential gives an isomorphism
\[
(\delta\widetilde x,\delta\widetilde p)\ \longleftrightarrow\ (\delta\tildeX,\delta\tildeP).
\]
Therefore the total differential of the coordinate change map is invertible.
This proves that the coordinate change map is locally invertible.

\bibliographystyle{amsalpha}

\bibliography{biblio2}

\end{document}